\RequirePackage{fix-cm}
\documentclass[envcountsame,smallextended]{svjour3}       
\smartqed  
\usepackage{amsthm}
\usepackage{amssymb}
\usepackage{graphicx}
\usepackage[top=35mm, bottom=25mm, left=20mm, right=20mm]{geometry}
\usepackage[numbers,sort&compress]{natbib}
\usepackage{amsmath}
\usepackage{amsfonts}
\usepackage{amssymb}
\usepackage{enumitem}
\usepackage{graphicx}
\usepackage{subcaption}
\usepackage{tikz}
\usetikzlibrary{arrows.meta, positioning}
\usepackage{pgfplots}
\usetikzlibrary{patterns}
\usepackage{multicol}
\usepackage{hyperref}
\usepackage{xcolor}
\usepackage{cancel}
\usepackage{ulem}
\usepackage[active]{srcltx}

\numberwithin{equation}{section}
\usepackage{subcaption}

\usepackage{booktabs}
\usepackage{multirow}

\usepackage{algorithm}%
\usepackage{algorithmicx}%
\usepackage{algpseudocode}%

\newcommand\R{\mathbb{R}}
\newcommand\Rinf{\overline{\mathbb{R}}}
\newcommand\inter[1]{ {\rm \textbf{int}}(#1)} 
\newcommand\closure[1]{ {\rm \textbf{cl}}(#1)} 
\newcommand\dom[1]{ \bs{{\rm dom}}(#1)} 
\newcommand\Dom[1]{ \bs{{\rm Dom}}(#1)} 
\newcommand\dist{ \bs{{\rm dist}}} 
\newcommand\gf{\varphi} 
\newcommand\gh{\psi} 
\newcommand\fgam[3]{#1_{#3}^{#2}}

\newcommand\fgamepsk[3]{#1^{#2,\varepsilon_k}_{#3}}
\newcommand\fgamepsko[3]{#1^{#2,\varepsilon_{k+1}}_{#3}}
\newcommand\hfgam[3]{#1^{#3, #2}}
\newcommand\prox[3]{ \bs{{\rm prox}}_{#2#1}^{#3}}

\newcommand\ov[1]{\overline{#1}}
\newcommand\mb{\mathbf{B}}
\newcommand\bs[1]{\boldsymbol{#1}}
\newcommand\argmint[1]{\mathop{\bs{\arg\min}}\limits_{#1}}
\newcommand\Nz{\mathbb{N}_0}
\newcommand\fv{\widehat{\gf}}
\newcommand\Tprox[3]{ \bs T_{#2#1}^{#3}}
\newcommand\maj[3]{\mathcal{M}_{#1}^{#2}(#3)}
\newcommand\argmin[1]{\bs{\arg\min}_{#1}}
\newcommand\st{\mathcal{S}}

\journalname{}

\begin{document}

\title{
First-order majorization-minimization meets high-order majorant: Boosted inexact high-order forward-backward method
}

\titlerunning{First-order majorization-minimization meets high-order majorant}        

\author{Alireza Kabgani         \and
        Masoud Ahookhosh 
}


\institute{A. Kabgani, M. Ahookhosh \at
              Department of Mathematics, University of Antwerp, Antwerp, Belgium. \\
              \email{alireza.kabgani@uantwerp.be, masoud.ahookhosh@uantwerp.be}             \\
              The Research Foundation Flanders (FWO) research project G081222N and UA BOF DocPRO4 projects with ID 46929 and 48996 partially supported the paper's authors.
}

\date{Received: date / Accepted: date}
\maketitle

\begin{abstract}
This paper introduces a first-order majorization–minimization framework based on a high-order majorant for continuous functions, incorporating a non-quadratic regularization term of degree $p>1$. Notably, it is shown to be valid if and only if the function is 
$p$-paraconcave, thus extending beyond Lipschitz and H\"{o}lder gradient continuity for  $p \in (1,2]$, and implying concavity for $p>2$.  In the smooth setting, this majorant recovers a variant of the classical descent lemma with quadratic regularization. Building on this foundation, we develop a high-order inexact forward–backward algorithm (HiFBA) and its line-search–accelerated variant, named Boosted HiFBA. For convergence analysis, we introduce a high-order forward–backward envelope (HiFBE), which serves as a Lyapunov function. We establish subsequential convergence under suitable inexactness conditions, and we prove global convergence with linear rates for functions satisfying the Kurdyka–\L{}ojasiewicz inequality. Our preliminary experiments on linear inverse problems
and regularized nonnegative matrix factorization highlight the efficiency of HiFBA and its boosted variant, demonstrating their potential for solving challenging nonconvex optimization problems.


\end{abstract}

\section{Introduction}
\label{intro}

This paper addresses the structured nonsmooth and nonconvex optimization problem 
\begin{equation}\label{eq:mainproblemcom}
{\mathop {\mathrm{\bs\min}}\limits_{x\in \R^n}}\ \gf (x) := f(x) + g(x),
\end{equation}
where $f:\R^n\to \R$ is smooth (potentially nonconvex), and  $g:\R^n\to \Rinf:=\R\cup\{+\infty\}$ is proper, lower semicontinuous, and possibly both nonconvex and nonsmooth. In such problems, the function $f$ commonly stands for a \textit{data fidelity} term, while $g$ typically represents a \textit{regularizer}, imparting additional structure to the solution set. The above formulation is broad enough to encompass many prominent applications across signal and image processing, compressed sensing, machine learning, data science, statistics, and so on; see, e.g., \cite{Ahookhosh21,beck2019optimization,combettes2011proximal,demarchi2023constrained,Themelis18,trandinh2022hybrid} and references therein.  However, in the classical setting, it is commonly assumed that $f$ has a \textit{Lipschitz continuous gradient}, which is not the case in many practical applications, thus limiting the applicability of classical optimization methods. This motivates the quest for relaxed assumptions on the function $f$, aiming to develop algorithms to address \eqref{eq:mainproblemcom} that are both theoretically sound and computationally efficient.

A well-established strategy for dealing with the problem \eqref{eq:mainproblemcom} is the so-called \textit{majorization–minimization} (MM) framework; see, e.g., \cite{lange2016mm,sun2016majorization}.
The central idea is to construct a simple and tractable surrogate model for the function $\varphi$ that simultaneously approximates and majorizes this function, which will be minimized at each iteration of our scheme. More specifically, the framework relies on the existence of a majorant $\mathcal{M}(x,y)$ for the smooth function $f$, satisfying $f(y)\leq \mathcal{M}(x,y)$ for all $x,y\in \R^n$, leading to a composite majorant $\psi(x,y):=  \mathcal{M}(x,y)+g(y)$ for $\varphi$. Minimizing the surrogate $\psi(x,y)$ with respect to $y$ constitutes the core step of MM. We emphasize that, depending on the properties of $f$, the majorant $\mathcal{M}(x,y)$ can be constructed using first-, second-, or higher-order derivatives, which leads to first-, second-, or higher-order methods; see, e.g., \cite{Ahookhosh21, beck2017first, bauschke2017descent, Nesterov2023a}.

The class of first-order majorants has received considerable attention over the past few decades because of its simplicity and low memory requirements, making it well-suited for large-scale optimization problems; see, e.g., \cite{Ahookhosh21,bauschke2017descent,bonettini2025linesearch,combettes2005signal,Stella17,Themelis18}. The most common example is the quadratic model, valid whenever $f$ has a Lipschitz continuous gradient with constant $L>0$, i.e.,
\begin{equation}\label{eq:lipdesc}
f(y)\leq \mathcal{M}(x,y): = f(x)+\langle \nabla f(x), y-x\rangle +\frac{L}{2}\Vert y - x\Vert^2,\qquad \forall x, y\in \R^n,
\end{equation}
which is commonly referred to as the {\it first-order descent lemma}; see, e.g., \cite[Proposition A.24]{bertsekas1999nonlinear}. Assuming the H\"older continuity of the gradient of $f$ with modulus $\nu\in (0,1]$ and constant $L>0$ ensures
\begin{equation}\label{eq:holderianDescentLemma}
    f(y)\leq \mathcal{M}(x,y): = f(x) + \langle \nabla f(x), y - x \rangle +\frac{L}{1 + \nu}\Vert y-x \Vert ^{1+\nu},\qquad \forall x, y\in \R^n,
\end{equation}
which generalizes \eqref{eq:lipdesc} and is referred to as the {\it H\"{o}lderian descent lemma}; see, e.g., 
\cite{Nesterov15univ}, \cite[Lemma~1]{yashtini2016global} and see Fact~\ref{fact:holder:declem} in the current work.
Beyond Lipschitz and  H\"{o}lder continuity, many important functions instead exhibit \textit{relative smoothness} with respect to a convex distance-generating function $h$, which means that there exists $L>0$ such that $Lh-f$ is convex. This leads to the inequality
\begin{equation}\label{eq:bregmanDescentLemma}
f(y)\leq \mathcal{M}(x,y): = f(x)+\langle \nabla f(x), y-x\rangle +L\bs D_h(y,x),\qquad \forall x, y\in \R^n,
\end{equation}
where $\bs D_h(y, x) = h(y) - h(x) - \langle \nabla h(x), y - x \rangle$ denotes the Bregman distance associated with the kernel $h$, which is referred to as the {\it non-Euclidean descent lemma}; see \cite{bauschke2017descent} and its multi-block variants \cite{ahookhosh2021multi,ahookhosh2021block,khanh2022block}. These descent lemmas serve as cornerstones for many first-order optimization methods, including gradient descent \cite{bertsekas1999nonlinear}, accelerated gradient methods \cite{beck2009fast,nesterov1983method}, forward-backward splitting \cite{combettes2005signal,fukushima1981generalized}, accelerated forward-backward methods \cite{attouch2016rate,villa2013accelerated}, and line-search-accelerated forward-backward methods \cite{Ahookhosh21,bonettini2025linesearch,Themelis18}, among others.

Forward-backward splitting (proximal gradient) methods, among others, have proven particularly effective, often exhibiting superior numerical performance compared to classical subgradient-based and bundle methods. A generic iteration of the forward-backward method is given by
\begin{equation}\label{eq:fbsGeneral}
x^{k+1}\in \argmint{y\in\R^n}\left\{f(x^k)+\langle \nabla f(x^k), y-x^k\rangle+g(y)+\frac{1}{\gamma}\bs \zeta(y, x^k)\right\},
\end{equation}
where $\bs\zeta(\cdot,\cdot)$ is commonly a proximity term such as $\bs\zeta(y, x^k)=\frac{1}{2}\Vert y - x^k\Vert^2$, $\bs\zeta(y, x^k)=\frac{1}{1 + \nu}\Vert y-x^k\Vert ^{1+\nu}$, 
or $\bs\zeta(y, x^k)=\bs D_h(y,x^k)$ with respect to the majorants \eqref{eq:lipdesc}, \eqref{eq:holderianDescentLemma}, and \eqref{eq:bregmanDescentLemma}, respectively, with $\gamma\in (0, L^{-1})$. Although this method with the squared norm and the Bregman distance has been well studied (e.g. \cite{Ahookhosh21, attouch2016rate, bauschke2017descent, bonettini2025linesearch, combettes2005signal, fukushima1981generalized,Themelis18,villa2013accelerated}), there has been little study of its variant with the majorization term $\bs\zeta(y, x^k)=\frac{1}{1 + \nu}\Vert y-x^k\Vert ^{1+\nu}$; see, e.g. \cite{Bredies09,Nesterov15univ}. 

In the current study, we introduce and characterize the notion of a \textit{high-order majorant} for continuous functions and then focus on the specific case of the \textit{high-order descent lemma} for smooth functions, which yields the majorant
\begin{equation}\label{eq:highOrdDescentLemma}
    f(y)\leq \mathcal{M}(x,y): = f(x)+\langle \nabla f(x), y-x\rangle +\frac{L_p}{p}\Vert y - x\Vert^p,\qquad \forall x, y\in \R^n,
\end{equation}
for some $p>1$ and $L_p>0$. Indeed, we seek the largest class of functions that satisfies the above inequality, encompassing problems with the gradient of $f$ that is Lipschitz or H\"{o}lder continuous. Here, the term ``high-order" refers to the power $p>1$ of the regularization term in \eqref{eq:highOrdDescentLemma}.
This majorant naturally leads to a {\it high-order forward-backward splitting method}, with the iteration of the form
\begin{equation}\label{eq:fbsGeneral:p}
    x^{k+1}\in \argmint{y\in\R^n}\left\{f(x^k)+\langle \nabla f(x^k), y-x^k\rangle+g(y)+\frac{1}{p\gamma}\Vert y - x^k\Vert^p\right\},
\end{equation}
and to a corresponding {\it high-order forward-backward envelope}, defined as
\begin{equation}\label{eq:fbsGeneral:fun}
   \fgam{\gf}{p}{\gamma}(x) = \mathop{\bs{\inf}}\limits_{y\in \R^n}\left\{f(x)+\langle \nabla f(x), y-x\rangle+g(y)+\frac{1}{p\gamma}\Vert y - x\Vert^p\right\},
\end{equation}
for $\gamma\in (0, L_p^{-1})$.
Solving \eqref{eq:fbsGeneral:p} approximately offers a flexible inexact setting by allowing an appropriate choice for $p$ based on the smoothness properties of $f$, thus establishing a novel paradigm to deal with \eqref{eq:mainproblemcom}.

\subsection{{\bf Contribution}}\label{sec:contribution}

The contributions of this paper are summarized as follows:

\begin{description}[wide, labelwidth=!, labelindent=0pt]
    \item[{\bf(i)}] {\bf First-order majorization-minimization with high-order majorants.} 
    We introduce a high-order majorant for the continuous functions $f$, constructed using the limiting subdifferential and a high-order regularization term of degree $p>1$, as defined in Definition~\ref{def:power majorization}. Notably, our characterizations of high-order majorants establish that a function holds a high-order majorant if and only if it is paraconcave. Here, we recognize two cases: (i) $p\in (1,2]$; and (ii) $p>2$. For $p\in (1,2]$, we show that the $p$-paraconcave functions form the broadest class of functions that admit a high-order majorant, and this class includes, in particular, functions with Lipschitz or H\"{o}lder continuous gradients. For $p>2$, it is shown that only concave functions admit a high-order majorant, a property particularly relevant when minimizing the difference of convex, weakly convex, or relatively smooth functions. Under Fr\'{e}chet differentiability of $f$, this notion reduces to a high-order descent lemma, which employs the first-order derivative of $f$ and a regularizer of degree $p>1$. In fact, this is a natural variant of the classical descent lemma, which serves as a foundation for designing first-order majorization-minimization methods with high-order majorants. 

    \item[{\bf(ii)}] {\bf High-order forward-backward envelope.} 
    We introduce the high-order forward-backward splitting mapping (HiFBS) and the corresponding high-order forward-backward envelope (HiFBE) for $p>1$, extending the corresponding classical notions that stand with $p=2$; see \cite{patrinos2013proximal, Stella17, Themelis18}. We then establish fundamental properties of HiFBE, which serve as key tools to analyze convergence behavior and rates of the boosted inexact high-order forward-backward method presented in Section~\ref{sec:HiFBEMET}. We also examine the relationships among the first-order reference points of HiFBE and those of the original objective function.

   \item[{\bf(iii)}] {\bf Inexact boosted high-order forward-backward method.} 
   By approximating the solution of the HiFBS mapping, we construct an inexact oracle and develop the high-order inexact forward-backward algorithm (HiFBA). Building on this, we introduce the boosted variant of HiFBA (Boosted HiFBA), which enhances it by applying a suitable nonmonotone line search that permits search directions not to be descent. Under suitable assumptions on the approximation quality of the HiFBS operator, we establish well-definedness and subsequential convergence of Boosted HiFBA, as well as global convergence and local linear convergence rates under the Kurdyka-{\L}ojasiewicz (KL) property. Preliminary numerical experiments on regularized inverse problems and regularized nonnegative matrix factorization demonstrate the promising performance of HiFBA, providing empirical support for theoretical foundations.
          
\end{description}

\subsection{{\bf Organization}}\label{sec:contribution}
The remainder of this paper is organized as follows. Section~\ref{sec:prelim} introduces the necessary preliminaries and notation. In Section~\ref{sec:first:majormin}, we present a detailed characterization of the high-order majorant and the descent lemma and introduce the class of paraconcave functions. Section~\ref{sec:hifbe} defines the high-order forward-backward splitting mapping (HiFBS) and the high-order forward-backward envelope (HiFBE), establishing their fundamental properties. Section~\ref{sec:HiFBEMET} develops high-order forward–backward methods that incorporate inexact strategies and provide convergence analysis under the Kurdyka-\L{}ojasiewicz property. Section~\ref{sec:numerical} reports preliminary numerical experiments. Finally, Section~\ref{sec:disc} concludes the paper with closing remarks.



\section{Preliminaries and notation}
\label{sec:prelim}

Throughout this paper, $\R^n$ denotes the $n$-dimensional \textit{Euclidean space}, 
while $\Vert \cdot \Vert$ and $\langle \cdot, \cdot \rangle$ represent the \textit{Euclidean norm}
 and \textit{inner product}, respectively. 
We denote the set of \textit{natural numbers} by $\mathbb{N}$ and let $\Nz := \mathbb{N} \cup \{0\}$. 
The set $\mb(\ov{x}; r)$ is the \textit{open ball} centered at $\ov{x} \in \R^n$ with radius $r > 0$. 
The \textit{interior} and \textit{closure} of a set $C \subseteq \R^n$ are denoted by $\inter{C}$ and
$\closure{C}$, respectively. 
The \textit{distance} between $x \in \R^n$ and a set $C$ is defined as $\dist(x, C) := \bs\inf_{y \in C} \Vert y - x \Vert$. 
We adopt the convention $\infty - \infty = \infty$.
The \textit{effective domain} of $\gh: \R^n \to \Rinf := \R \cup \{+\infty\}$ is $\dom{\gh} := \{x \in \R^n \mid \gh(x) < +\infty\}$
and $\gh$ is called \textit{proper} if $\dom{\gh} \neq \emptyset$. 
The set $\mathcal{L}(\gh, \lambda) := \{x \in \R^n \mid \gh(x) \leq \lambda\}$ is the \textit{sublevel set} of $\gh$ at $\lambda \in \R$. 
The set of \textit{minimizers} of $\gh$ over $C \subseteq \R^n$ is denoted by $\argmin{x \in C} \gh(x)$. 
The function $\gh$ is \textit{lower semicontinuous} (lsc) at $\ov{x} \in \R^n$ if $\bs\liminf_{k \to \infty} \gh(x^k) \geq \gh(\ov{x})$ for every sequence $x^k \to \ov{x}$. 
It is \textit{coercive} if $\bs\lim_{\Vert x \Vert \to +\infty} \gh(x) = +\infty$. 
For a set-valued map $\Psi: \R^n \rightrightarrows \R^n$, the \textit{domain} is defined as $\Dom{\Psi} := \{x \in \R^n \mid \Psi(x) \neq \emptyset\}$.

If $p > 1$, the gradient of $\frac{1}{p} \Vert x \Vert^p$ is $\nabla \left(\frac{1}{p} \Vert x \Vert^p \right) = \Vert x \Vert^{p-2} x$, 
where the convention $\frac{0}{0} = 0$ is adopted for $x = 0$. A point $\ov{x}$ is called a \textit{limiting point} of a sequence $\{x^k\}_{k \in \Nz}$ if $x^k \to \ov{x}$, and it is called a \textit{cluster point} if there exists a subsequence $x^j \to \ov{x}$ with $j\in J$
for some infinite subset $J \subseteq \Nz$. 

For two positive real numbers $a$ and $b$  and $p \in [0,1]$, it holds that
\begin{equation}\label{eq:intrp:p01}
    (a+b)^p\leq a^p+b^p.
\end{equation}
Invoking Young's inequality and \eqref{eq:intrp:p01}, if $a, b\geq 0$ and $p, q>1$ such that $\frac{1}{p}+\frac{1}{q}=1$ 
(i.e., $q=\frac{p}{p-1}$) we come to
\begin{equation}\label{eq:intrp:p02}
(ab)^\frac{1}{p}\leq \left(\frac{a^p}{p}+\frac{b^q}{q}\right)^{\frac{1}{p}}\leq \frac{a}{p^{\frac{1}{p}}}+\frac{b^{\frac{q}{p}}}{q^{\frac{1}{p}}}\leq a+b^{\frac{q}{p}}=a+b^{\frac{1}{p-1}}.
\end{equation}

A proper function $\gh: \R^n \to \Rinf$  is  \textit{Fr\'{e}chet differentiable} at $\ov{x}\in \inter{\dom{\gh}}$ 
with \textit{Fr\'{e}chet derivative}  
$\nabla \gh(\ov{x})$
 if 
\[
\mathop{\bs\lim}\limits_{x\to \ov{x}}\frac{\gh(x) -\gh(\ov{x}) - \langle \nabla \gh(\ov{x}) , x - \ov{x}\rangle}{\Vert x - \ov{x}\Vert}=0.
\]
For a set $C\subseteq\R^n$, the notation $\gh\in \mathcal{C}^{k}(C)$ indicates that $\gh$ is $k$-times continuously differentiable on $C$,
where $k\in \mathbb{N}$. 
The \textit{Fr\'{e}chet/regular} and \textit{Mordukhovich/limiting subdifferentials} of a function $\gh: \R^n\to\Rinf$ at $\ov{x}\in \dom{\gh}$ are  defined as~\cite{Mordukhovich2018,Rockafellar09}
\[
\widehat{\partial}\gh(\ov{x}):=\left\{\zeta\in \R^n\mid~\mathop{\bs\liminf}\limits_{x\to \ov{x}}\frac{\gh(x)- \gh(\ov{x}) - \langle \zeta, x - \ov{x}\rangle}{\Vert x - \ov{x}\Vert}\geq 0\right\},
\]
and
\[
\partial \gh(\ov{x}):=\left\{\zeta\in \R^n\mid~\exists x^k\to \ov{x}, \zeta^k\in \widehat{\partial}\gh(x^k),~\text{with}~\gh(x^k)\to \gh(\ov{x})~\text{and}~ \zeta^k\to \zeta\right\}.
\]
We have $\widehat{\partial}\gh(\ov{x})\subseteq\partial \gh(\ov{x})$.
The class of functions with H\"{o}lder continuous gradients has recently garnered increased attention for its applications in optimization \cite{Berger2020Quality,bolte2023backtrack,Cartis2017Worst,Nesterov15univ,yashtini2016global}.
\begin{definition}[H\"{o}lder continuity]\label{def:hold}
A function $\gh: \R^n \to \R$ has a \textit{$\nu$-H\"{o}lder continuous gradient} on $\R^n$ with $\nu \in (0, 1]$ if it is Fr\'{e}chet differentiable and  there exists a constant $L_\nu\geq 0$ such that
\begin{equation}\label{eq:nu-Holder continuous gradient}
\Vert \nabla \gh(y)- \nabla \gh(x)\Vert \leq L_\nu \Vert y-x\Vert^\nu, \qquad \forall x, y\in \R^n.
\end{equation}
\end{definition}
The class of functions with $\nu$-H\"{o}lder continuous gradients is denoted by $\mathcal{C}^{1, \nu}_{L_\nu}(\R^n)$ and termed \textit{weakly smooth}, and if $\nu=1$, $\gh$ is called \textit{$L$-smooth}. The following descent lemma from \cite[Lemma~1]{yashtini2016global} applies to $\gh\in \mathcal{C}^{1, \nu}_{L_{\nu}}(\R^n)$. 
\begin{fact}[H\"olderian descent lemma]\label{fact:holder:declem}\cite{yashtini2016global}
Let $\gh\in \mathcal{C}^{1, \nu}_{L_{\nu}}(\R^n)$ with $\nu\in (0, 1]$. Then, for all $x, y\in \R^n$, 
\begin{equation}\label{eq:upperbounf for c1,alp}
\left\vert\gh(y)- \gh(x) - \langle \nabla \gh(x), y - x \rangle\right\vert\leq\frac{L_{\nu}}{1 + \nu}\Vert y-x \Vert ^{1+\nu}.
\end{equation}
\end{fact}


We recall the concept of uniform level boundedness, a crucial property of optimal value mappings.
\begin{definition}[Uniform level boundedness]\label{def:lboundlunif}\cite[Definition~1.16]{Rockafellar09}
A function $\Phi: \R^n \times \R^m\to \Rinf$ with values given by $\Phi(u, x)$ is \textit{level-bounded in $u$ locally uniformly in $x$} if, for each $\ov{x}\in \R^m$ and $\lambda\in \R$, there exists a neighborhood $V$ of $\ov{x}$ along with a bounded set $B\subseteq \R^n$ such that
$\{u\mid \Phi(u, x)\leq \lambda\}\subseteq B$ for all $x\in V$.
\end{definition}

Recent studies have revisited the \textit{high-order proximal operator} (HOPE) and \textit{high-order Moreau envelope} (HOME) for their utility in developing practical algorithms for nonsmooth, nonconvex optimization problems 
\cite{Ahookhosh2025,Kabgani24itsopt,KecisThibault15}.
We briefly review these concepts and related properties.

\begin{definition}[High-order proximal operator and Moreau envelope]
Let $p>1$, $\gamma>0$, and $\gh: \R^n \to \Rinf$ be a proper function. 
The \textit{high-order proximal operator} (\textit{HOPE}) of $\gh$ with parameter $\gamma$,
$\prox{\gh}{\gamma}{p}: \R^n \rightrightarrows \R^n$, is 
    \begin{equation}\label{eq:Hiorder-Moreau prox}
       \prox{\gh}{\gamma}{p} (x):=\argmint{y\in \R^n} \left\{\gh(y)+\frac{1}{p\gamma}\Vert x- y\Vert^p\right\},
    \end{equation}     
and the \textit{high-order Moreau envelope} (\textit{HOME}) of $\gh$ with parameter $\gamma$, 
$\fgam{\gh}{\gamma,p}{}: \R^n\to \R\cup\{\pm \infty\}$, 
is 
    \begin{equation}\label{eq:Hiorder-Moreau env}
    \fgam{\gh}{\gamma,p}{}(x):=\mathop{\bs{\inf}}\limits_{y\in \R^n} \left\{\gh(y)+\frac{1}{p\gamma}\Vert x- y\Vert^p\right\}.
    \end{equation}
    \end{definition}

We recall the notion of high-order prox-boundedness from \cite{Kabgani24itsopt}.
\begin{definition}[High-order prox-boundedness]\label{def:s-prox-bounded}\cite[Definition~3.3]{Kabgani24itsopt}
A proper function $\gh:\R^n\to \Rinf$ is \textit{high-order prox-bounded} with order $p>1$ if there exist
$\gamma>0$ and $x\in \R^n$ such that
$\fgam{\gh}{\gamma,p}{}(x)>-\infty$. 
The supremum of all such $\gamma$ is denoted by $\gamma^{\gh, p}$ and is referred to as the threshold of high-order prox-boundedness for $\gh$.
\end{definition}

\begin{remark}\label{rem:prox-bounded}
If $\gh:\R^n\to \Rinf$ is convex or bounded below, it satisfies Definition~\ref{def:s-prox-bounded} with $\gamma^{\gh, p}=+\infty$.
\end{remark}

We conclude this section by demonstrating the well-definedness of the HOME and HOPE notions.
 \begin{fact}[Well-definedness of HOME and HOPE]\label{fact:level-bound+locally uniform}\cite[Theorem~3.4]{Kabgani24itsopt}
Let $p>1$ and $\gh: \R^n\to \Rinf$ be a proper lsc function that is high-order prox-bounded with threshold $\gamma^{\gh, p}>0$. For each $\gamma\in (0, \gamma^{\gh, p})$ and $x\in \R^n$,  the set $\prox{\gh}{\gamma}{p}(x)$ is nonempty and compact, and $\hfgam{\gh}{p}{\gamma}(x)$ is finite.
 \end{fact}



\section{First-order majorization-minimization with high-order majorant}
\label{sec:first:majormin}

In this section, we explore the construction and characterization of the \textit{high-order majorant} for continuous functions. This concept is instrumental in developing \textit{first-order majorization-minimization methods} that employ a high-order regularization term for objective functions of the form $\gf(x) = f(x) + g(x)$, where $f:\R^n\to\R$ is a continuous function, and $g:\R^n \to \Rinf$ is a proper lsc function.
In the context of the majorization–minimization framework, the term ``first-order'' refers to the use of the gradient (or subgradient) of the objective function in the construction of methods. In contrast, the term ``high-order" in the majorant indicates that the regularization term in the majorant involves a power $p>1$.
The framework presented here generalizes those introduced in \eqref{eq:lipdesc} and \eqref{eq:holderianDescentLemma} for $L$-smooth and weakly smooth functions, respectively.

We begin by formally defining the concept of a \textit{high-order majorant}. 
\begin{definition}[High-order majorant]\label{def:power majorization}
Let $\gh: \R^n\to \R$ be a continuous function. We say that $\gh$ possesses a \textit{high-order majorant} with power $p>1$ on $\R^n$ if there exists a constant $L_p>0$ such that, for all $x, y\in \R^n$,
\begin{equation}\label{eq:firstmaj:gcase}
\gh(y)\leq \gh(x)+\langle \zeta, y-x\rangle +\frac{L_p}{p}\Vert y - x\Vert^p, \qquad \forall \zeta\in \partial \gh(x).
\end{equation}
The class of all such functions is denoted by $\maj{L_p}{p}{\R^n}$.
\end{definition}
Having such a majorant for the continuous component $f$, particularly when $f$ is Fr\'{e}chet differentiable, is crucial for designing methods to solve problem~\eqref{eq:mainproblemcom} using structures such as \eqref{eq:fbsGeneral:p} with a regularization term of degree $p>1$. When $p\in (0, 1)$, the function $y\mapsto \Vert x-y\Vert^p$ is nonconvex for each fixed $x\in\R^n$, which may cause the function $\fgam{\gf}{p}{\gamma}(x)$ defined in \eqref{eq:fbsGeneral:fun} to equal $-\infty$ for all $x\in \R^n$, even in simple cases, such as $f(x)=0$, $g(x)=x$,  $x\in \R$, and any $\gamma>0$. The case $p=1$ introduces nonsmoothness into the regularization term in \eqref{eq:fbsGeneral:p} and \eqref{eq:fbsGeneral:fun}, complicating the majorization-minimization process. Thus, we restrict our analysis to $p>1$. 

The following remark elaborates on Definition~\ref{def:power majorization}, highlighting key properties and distinctions.
\begin{remark}\label{rem:onfom}
\begin{enumerate}[label=(\textbf{\alph*}), font=\normalfont\bfseries, leftmargin=0.7cm]
\item\label{rem:onfom:benin}
Unlike the inequality \eqref{eq:upperbounf for c1,alp} in the H\"{o}lderian descent lemma, the inequality \eqref{eq:firstmaj:gcase} 
does not provide an upper bound $\frac{L_p}{p}\Vert y - x\Vert^p$ 
for 
$\vert\gh(y)- \gh(x) - \langle \zeta, y - x \rangle\vert$ when $\zeta\in \partial\gh(x)$ and $p>1$.
In fact, imposing such a two-sided bound, i.e., for each $x, y\in \R^n$,
\begin{equation}\label{eq:onabsvalfmaj}
\vert\gh(y)- \gh(x) - \langle \zeta, y - x \rangle\vert\leq\frac{L_p}{p}\Vert y - x\Vert^p, \qquad \forall \zeta\in \partial \gh(x),
\end{equation}
would imply that
\[
0\leq \mathop{\bs\lim}\limits_{y\to x}\frac{\vert\gh(y) -\gh(x) - \langle \zeta , y - x\rangle\vert}{\Vert y - x\Vert}\leq\mathop{\bs\lim}\limits_{y\to x}\frac{L_p}{p}\Vert y - x\Vert^{p-1}=0,
\]
which means $\gh$ is Fr\'{e}chet differentiable. Thus, the one-sided nature of \eqref{eq:firstmaj:gcase} 
is essential for the class $\maj{L_p}{p}{\R^n}$ to include nonsmooth functions. See part~\ref{rem:onfom:nsm} for a nonsmooth example.
Moreover, the inequality \eqref{eq:onabsvalfmaj} holds for $p\in (1, 2]$ if and only if $\gh$ has a $(p-1)$-H\"{o}lder continuous gradient \cite[Theorem~4.1]{Berger2020Quality}. For $p>2$, it is even more restrictive, implying that $\gh$ must be an affine function, as discussed in Remark~\ref{rem:onaffine}.

\item If $\gh$ is convex, 
then the subgradient inequality $0\leq \gh(y)- \gh(x) - \langle \zeta, y - x \rangle$ holds
for all $x, y \in \R^n$ and $\zeta \in \partial \gh(x)$. Thus, if $\gh \in \maj{L_p}{p}{\R^n}$ as well, it satisfies \eqref{eq:onabsvalfmaj}, implying that $\gh$ is Fr\'{e}chet differentiable and either has $(p-1)$-H\"{o}lder continuous gradient for $p\in (1,2]$ or is an affine function for $p>2$; see part~\ref{rem:onfom:benin}.

\item \label{rem:onfom:nsm} By Fact \ref{fact:holder:declem}, for $\nu\in (0, 1]$, we have $\mathcal{C}^{1, \nu}_{L_{\nu}}(\R^n)\subseteq \maj{L_p}{p}{\R^n}$ with $p=1+\nu$ and $L_p\leq L_\nu$. 
Moreover, every concave $\gh$ possesses a high-order majorant for any $p> 1$ and $L_p>0$, since 
\begin{equation}\label{eq:concavegrad}
    \gh(y)\leq \gh(x)+\langle \zeta, y-x\rangle, \quad \forall x, y\in \R^n, \forall \zeta\in \partial \gh(x),
\end{equation}
which guarantees \eqref{eq:firstmaj:gcase}.
However, for $p \in (1, 2]$, the class $\maj{L_p}{p}{\R^n}$ strictly contains $\mathcal{C}^{1, p-1}_{L_{p-1}}(\R^n)$,
even when the definition of $\maj{L_p}{p}{\R^n}$ is restricted to Fr\'{e}chet differentiable functions. This is because 
Fr\'{e}chet differentiable concave functions satisfying \eqref{eq:firstmaj:gcase} need not have $(p-1)$-H\"{o}lder continuous gradients. 
Furthermore, the class $\maj{L_p}{p}{\R^n}$ extends beyond the union of $\mathcal{C}^{1, \nu}_{L_{\nu}}(\R^n)$ for $\nu \in (0,1]$ and the class of concave functions. Since $\maj{L_p}{p}{\R^n}$ is closed under addition (see Proposition~\ref{prop:calcul}), functions such as $\gh_1(x) = |x|^{\frac{3}{2}} - |x|$ (nonsmooth) and $\gh_2(x) = |x|^{\frac{3}{2}} - x^4$ (smooth) belong to $\maj{L_{\frac{3}{2}}}{\frac{3}{2}}{\R^n}$, even though they do not lie in $\mathcal{C}^{1, \frac{1}{2}}_{L_{\frac{1}{2}}}(\R^n)$ or the class of concave functions.
\end{enumerate}
\end{remark}

Next, we characterize the class $\maj{L_p}{p}{\R^n}$.
\begin{theorem}[Characterization of high-order majorant]\label{th:parandfmaj}
Let $\gh: \R^n\to \R$ be a continuous function and $p>1$. The following statements are equivalent
\begin{enumerate}[label=(\textbf{\alph*}), font=\normalfont\bfseries, leftmargin=0.7cm]
\item\label{th:parandfmaj:a} $\gh\in \maj{L_p}{p}{\R^n}$ for some $L_p>0$,

\item\label{th:parandfmaj:b} there exists a constant $c_1>0$ such that, for all $x,y\in \R^n$,
\[
\langle \zeta_y - \zeta_x , y- x\rangle\leq c_1\Vert y - x\Vert^p, \qquad \forall \zeta_x\in \partial \gh(x),\forall \zeta_y\in \partial \gh(y),
\]

\item\label{th:parandfmaj:c} there exists a constant $c_2>0$ such that, for all $x,y\in \R^n$ and $\lambda\in [0,1]$,
\[
\lambda \gh(x)+(1-\lambda) \gh(y)\leq \gh(\lambda x+(1-\lambda)y)+c_2\Vert x - y\Vert^p,
\]

\item\label{th:parandfmaj:d} there exists a constant $c_3>0$ such that, for all $x,y\in \R^n$ and $\lambda\in [0,1]$,
\[
\lambda \gh(x)+(1-\lambda) \gh(y)\leq \gh(\lambda x+(1-\lambda)y)+c_3 \lambda(1-\lambda)\Vert x - y\Vert^p,
\]

\item\label{th:parandfmaj:e} there exists a constant $c_4>0$ such that, for all $x,y\in \R^n$ and $\lambda\in [0,1]$,
\[
\lambda \gh(x)+(1-\lambda) \gh(y)\leq \gh(\lambda x+(1-\lambda)y)+c_4~\bs\min\{\lambda, 1-\lambda\}\Vert x - y\Vert^p.
\]
\end{enumerate}
\end{theorem}
\begin{proof}
The characterization of the class $\maj{L_p}{p}{\R^n}$ leverages properties of \textit{paraconvex} and \textit{paraconcave} functions established in \cite{Jourani} and revisited in \cite{Rahimi2024} (see Remark~\ref{rem:onth:parandfmaj}~\ref{rem:onth:parandfmaj:a}).
Utilizing the inequality \eqref{eq:firstmaj:gcase},  \ref{th:parandfmaj:a} immediately implies \ref{th:parandfmaj:b}.
From \cite[Corollary~7.1]{Jourani}, for a continuous function $\phi:\R^n\to \R$, there exists a constant $c_1>0$ such that, for all $x,y\in \R^n$,
\[
\langle \zeta_y - \zeta_x , y- x\rangle\geq - c_1\Vert y - x\Vert^p, \qquad \forall \zeta_x\in \partial \phi(x),\forall \zeta_y\in \partial \phi(y),
\]
if and only if there exists $c_2>0$ such that, for all $x,y\in \R^n$ and $\lambda \in [0,1]$,
\[
\phi(\lambda x+(1-\lambda)y) \leq \lambda \phi(x)+(1-\lambda) \phi(y) +c_2\Vert x - y\Vert^p.
\]
Setting $\phi = -\gh$ yields the equivalence of \ref{th:parandfmaj:b} and \ref{th:parandfmaj:c}. By a similar argument, the equivalence among \ref{th:parandfmaj:b}, \ref{th:parandfmaj:c}, and \ref{th:parandfmaj:d} is demonstrated in \cite[Corollary~7.1]{Jourani}.
For the equivalence of \ref{th:parandfmaj:c} and \ref{th:parandfmaj:e}, refer to \cite[Proposition~2.1]{Jourani}; see also \cite[Proposition~3.3]{Rahimi2024}. Finally, by \cite[Theorem~3.1]{Jourani}, \ref{th:parandfmaj:c} implies \ref{th:parandfmaj:a}.
\end{proof}
A similar equivalence among statements in Theorem~\ref{th:parandfmaj}~\ref{th:parandfmaj:a},~\ref{th:parandfmaj:b},~and \ref{th:parandfmaj:d} for convex $L$-smooth functions, corresponding to the case $p = 2$, has been demonstrated in \cite[Theorem~2.1.5]{Nesterov2018}.
\begin{remark}\label{rem:onth:parandfmaj}
\begin{enumerate}[label=(\textbf{\alph*}), font=\normalfont\bfseries, leftmargin=0.7cm]
\item\label{rem:onth:parandfmaj:a} The inequality in Theorem~\ref{th:parandfmaj}~\ref{th:parandfmaj:c} characterizes 
the class of \textit{paraconcave} functions, which has received less attention in the literature compared to its counterpart, \textit{paraconvexity}. A function $\gh: \R^n \to \R$ is paraconcave if $-\gh$ is paraconvex, as defined in \cite{daniilidis2005filling,Jourani,Rahimi2024,Rolewicz00,rolewicz2005paraconvex}.
Paraconvex functions have been extensively studied, particularly for $p = 2$, which corresponds to \textit{weak convexity}, a property prevalent in optimization problems
\cite{Bohm21,Goujon2024Learning,Kabgani25itsdeal,Kungurtsev2021zeroth,Liao2024error,Montefusco2013fast,Pougkakiotis2023Zeroth,Yang2019Weakly}.
We note that the weak convexity of $\gh$ is equivalent to the existence of some
$\rho>0$ such that $\gh+\frac{\rho}{2}\Vert \cdot\Vert^2$ is convex. In this case, we also call 
$\gh$, \textit{$\rho$-weakly convex}.

\item If $\gh: \R^n \to \R$ is a continuous function satisfying the inequality in Theorem~\ref{th:parandfmaj}~\ref{th:parandfmaj:c},
 then $\gh$ is locally Lipschitz on $\R^n$, as shown in \cite[Proposition~2.2]{Jourani}. Consequently,  $\partial \gh(x)$ is nonempty for all $x \in \R^n$.
\end{enumerate}
\end{remark}
Let us now provide necessary and sufficient conditions for functions admitting a high-order majorant.
\begin{theorem}[Conditions for the high-order majorant]\label{th:parandfmaj:con}
Let $\gh: \R^n\to \R$ be a continuous function.
\begin{enumerate}[label=(\textbf{\alph*}), font=\normalfont\bfseries, leftmargin=0.7cm]
\item \label{th:parandfmaj:p2} If \eqref{eq:firstmaj:gcase} holds for $p>2$, then $\gh$ is concave.
\item \label{th:parandfmaj:p0} If $\gh$ is concave, then \eqref{eq:firstmaj:gcase} holds for any $p> 1$ and $L_p>0$.
\item \label{th:parandfmaj:hol} $\mathcal{C}^{1, \nu}_{L_{\nu}}(\R^n)\subseteq \maj{L_p}{p}{\R^n}$ for any $\nu\in (0, 1]$, with $p=1+\nu$ and $L_p \leq L_\nu$.

\item  \label{th:parandfmaj:sec} Suppose that $\gh\in \mathcal{C}^{1}(\R^n)$ and, 
for each open convex set $S\subseteq \R^n$ with $0\notin S$, assume $\gh\in \mathcal{C}^{2}(S)$. If there exist $\nu\in (0,1]$ and $c\geq 0$ such that
\begin{equation}\label{eq:seccond}
    \frac{c(1+\nu)}{\Vert x\Vert^{3-\nu}}\left(\Vert x\Vert^2 \mathcal{I} - (1-\nu)xx^T\right)\succeq\nabla^2 \gh(x),
\end{equation}
for all $x\in S$, where $\mathcal{I}$ is the identity matrix, then 
$\gh\in \maj{L_\nu}{1+\nu}{\R^n}$ for some $L_\nu>0$.
\end{enumerate}
\end{theorem}
\begin{proof}
The inequality \eqref{eq:firstmaj:gcase} implies that for all $x, y\in \R^n$,
\[
\langle \zeta_y- \zeta_x, y-x \rangle\geq -\frac{2L_{p}}{p}\Vert x - y \Vert ^{p}, \qquad\forall \zeta_x\in \partial (-\gh(x)), \forall \zeta_y\in \partial (-\gh(y)).
\]
By \cite[Theorem 7.1]{Jourani}, $-\gh$ is $p$-paraconvex.
Since $p>2$,  \cite[Proposition~1]{Rolewicz00} yields that $-\gh$ is convex on $\R^n$.  Thus, $\gh$ is concave on $\R^n$ and verifies Assertion~\ref{th:parandfmaj:p2}.
Assertion~\ref{th:parandfmaj:p0} follows from~\eqref{eq:concavegrad} and 
Assertion~\ref{th:parandfmaj:hol} is a direct result of Fact \ref{fact:holder:declem}.
To establish Assertion~\ref{th:parandfmaj:sec},
let $S$ be a nonempty open convex set not containing $\ov{x}=0$.
From \eqref{eq:seccond}, the function $\phi(x):= c\Vert x\Vert^{1+\nu}-\gh(x)$ is convex on $S$ \cite[Theorem 2.1.4]{Nesterov2018}. 
Let $y\in \R^n$ be arbitrary. Choose an open convex set $S$ such that $y\in S$ and  $\ov{x}=0$ is a boundary point of it. Considering the sequence $x^k := \lambda_k y$, where $\lambda\in (0, 1)$ and $\lambda_k\downarrow 0$, we have $x^k\in S$ \cite[Theorem~2.2.2]{bazaraa2006nonlinear}, $x^k\to \ov{x}$, and 
$\langle \nabla \phi(x^k)-\nabla\phi(y), x^k - y\rangle\geq 0$. Taking the limit as $k\to \infty$ yields convexity of $\phi$ on 
$\R^n$.
Hence, $\gh$ is the sum of an H\"{o}lderian function and a concave function. Utilizing Assertions~\ref{th:parandfmaj:p0}~and~\ref{th:parandfmaj:hol} give our desired result.
\end{proof}
As noted in Remark~\ref{rem:onfom}~\ref{rem:onfom:benin}, the inequality \eqref{eq:firstmaj:gcase} offers advantages over \eqref{eq:onabsvalfmaj} for defining the high-order majorant class $\maj{L_p}{p}{\R^n}$. An additional benefit is highlighted in the subsequent remark.
\begin{remark}\label{rem:onaffine}
In Theorem~\ref{th:parandfmaj:con}~\ref{th:parandfmaj:p2}, we established that if \eqref{eq:firstmaj:gcase} holds for $p>2$, then $\gh$ must be concave. Imposing the stronger, two-sided inequality \eqref{eq:onabsvalfmaj} for $p>2$ places a much stricter constraint on the function. Specifically, \eqref{eq:onabsvalfmaj} implies that, for each $x, y\in \R^n$,
\[
\langle \zeta_y- \zeta_x, y-x \rangle\geq -\frac{2L_{p}}{p}\Vert x - y \Vert ^{p}, \qquad\forall \zeta_x\in \partial (-\gh(x)), \forall \zeta_y\in \partial (-\gh(y)),
\]
and
\[
\langle \zeta_y- \zeta_x, y-x \rangle\geq -\frac{2L_{p}}{p}\Vert x - y \Vert ^{p}, \qquad\forall \zeta_x\in \partial \gh(x), \forall \zeta_y\in \partial \gh(y).
\]
By \cite[Theorem 7.1]{Jourani}, both $\gh$ and $-\gh$ are $p$-paraconvex.
Since $p>2$,  \cite[Proposition~1]{Rolewicz00} implies that $\gh$ and $-\gh$ are convex on $\R^n$.  Thus, $\gh$ is an affine function on $\R^n$.
\end{remark}
Here, we present calculus rules for the class $\maj{L_p}{p}{\R^n}$ that enable the construction of new functions possessing the high-order majorant.
\begin{proposition}[Calculus rules]\label{prop:calcul}
Let $p > 1$. The following calculus rules hold for the class $\maj{L_p}{p}{\R^n}$.
\begin{enumerate}[label=(\textbf{\alph*}), font=\normalfont\bfseries, leftmargin=0.7cm]
\item \label{prop:calcul:mul} If $\gh\in \maj{L_p}{p}{\R^n}$ and $\alpha>0$, then $\alpha \gh(x)\in \maj{\widehat{L}_p}{p}{\R^n}$ with 
$\widehat{L}_p\leq \alpha L_p$.

\item \label{prop:calcul:sum} If $\gh_i\in \maj{L^i_p}{p}{\R^n}$ for $i=1, \ldots, N$, then the sum $\gh(x) = \sum_{i=1}^{N}\gh_i(x)\in \maj{L_p}{p}{\R^n}$ with 
$L_p\leq \sum_{i=1}^{N}L^i_p$.

    \item \label{prop:calcul:a} If $\gh\in \maj{L_p}{p}{\R}$ is non-decreasing and $\phi: \R^n \to \R$ is concave and $\nu$-H\"{o}lder continuous with $\nu \in (0,1]$, i.e., there exists $L_\nu > 0$ such that $\vert\phi(x) - \phi(y)\vert \leq L_\nu \| x - y \|^\nu$ for all $x, y \in \R^n$, then the composition $\vartheta(x) = \gh(\phi(x))$ belongs to $\maj{L}{p\nu}{\R^n}$ for some $L>0$.

\item \label{prop:calcul:b} If $\gh: \R^n \to \R$ is concave and $L$-Lipschitz, and $\Phi: \R^m \to \R^n$ has a $\nu$-H\"{o}lder continuous Jacobian $J_\Phi$, i.e., there exists $L_\nu > 0$ such that
    \[
\Vert \Phi(y) - \Phi(x) -   J_\Phi(x) (y-x)\Vert \leq \frac{L_\nu}{1+\nu}\Vert x - y\Vert^{1+\nu},\qquad \forall x, y\in \R^m,
    \]
    then the composition $\vartheta(x) = \gh(\Phi(x))$ belongs to $\maj{\widehat{L}}{1 + \nu}{\R^n}$ with $\widehat{L} \leq L L_\nu$.

\end{enumerate}
\end{proposition}
\begin{proof}
 The proof of Assertions~\ref{prop:calcul:mul}~and~\ref{prop:calcul:sum} is straightforward.
For Assertion~\ref{prop:calcul:a}, fix any $x, y \in \R^n$ and $\lambda \in (0,1)$. By Theorem~\ref{th:parandfmaj}, there exists $c>0$ such that
\begin{align*}
\lambda \psi(\phi(x)) + (1 - \lambda) \psi(\phi(y))
&\leq \psi\left(\lambda \phi(x) + (1 - \lambda) \phi(y)\right) + c\vert\phi(x) - \phi(y) \vert^p \\
&\leq \psi\left( \phi(\lambda x + (1 - \lambda) y) \right) + c L_\nu^p \Vert x-y\Vert^{p\nu},
\end{align*}
which proves the claim. To prove Assertion~\ref{prop:calcul:b}, assume that $x, y\in \R^n$ and $\zeta\in \partial\vartheta(x)$ are arbitrary.
Hence, there exists some $\eta\in \partial\gh(\Phi(x))$ such that $\zeta = J_\Phi(x)^T\eta$.
From the inequality, 
\eqref{eq:concavegrad},
\begin{align*}
\vartheta(y) =  \gh(\Phi(y))&\leq  \gh(\Phi(x))+\langle\eta, \Phi(y) - \Phi(x)\rangle+\langle \zeta, y-x\rangle-\langle\zeta, y-x\rangle
\\&= \vartheta(x)+\langle\eta, \Phi(y) - \Phi(x)-J_\Phi(x)(y-x)\rangle+\langle \zeta, y-x\rangle
\\&\leq \vartheta(x)+\langle \zeta, y-x\rangle +\frac{LL_\nu}{1+\nu}\Vert x - y\Vert^{1+\nu},
\end{align*}
ensuring $\vartheta(x)\in \maj{\widehat{L}}{1 + \nu}{\R^n}$ with $\widehat{L} \leq L L_\nu$.
\end{proof}

In the following subsection, we present several practical examples formulated as problem \eqref{eq:mainproblemcom}, where the objective $\gf(x) = f(x) + g(x)$ has a Fr\'{e}chet differentiable component $f$ admitting a high-order majorant.

\subsection{{\bf High-order descent lemma}}
\label{subsec:HODL}

Thus far, our discussion of the high-order majorant has not required differentiability. We now turn to the case where the function $f$ in problem \eqref{eq:mainproblemcom} is smooth, which plays a central role in the forward–backward methods developed in later sections. This motivates the following definition.

\begin{definition}[High-order descent lemma]
A Fr\'{e}chet differentiable function $f:\R^n\to \R$ is said to satisfy the \textit{high-order descent lemma} with power $p>1$ and constant $L_p > 0$ if it belongs to the class $\maj{L_p}{p}{\R^n}$, i.e.,
\begin{equation}\label{eq:hdlsm}
 \gh(y)\leq \gh(x)+\langle \nabla \gh(x), y-x\rangle +\frac{L_p}{p}\Vert y - x\Vert^p,\qquad \forall x, y\in \R^n.
 \end{equation}
\end{definition}

If $f:\R^n\to \R$ is Fr\'{e}chet differentiable and satisfies the high-order descent lemma with power $p>1$,
 then the composite function $\gf(x) = f(x)+g(x)$ is majorized by the surrogate function $\mathcal{M}(x,y)$ given by
\begin{equation}\label{eq:majforcom}
\gf(y)\leq \mathcal{M}(x,y): = f(x)+\langle \nabla f(x), y-x\rangle+g(y) +\frac{L_p}{p}\Vert y - x\Vert^p,\qquad \forall x, y\in \R^n.
\end{equation}
The majorant $\mathcal{M}(x,y)$ is central to the development of the \textit{first-order majorization-minimization} methods 
with high-order regularization, introduced in Section~\ref{sec:HiFBEMET}. 

Next, let us present several classes of optimization problems of the form \eqref{eq:mainproblemcom} in which the smooth component $f$ satisfies the high-order descent lemma, thereby demonstrating the broad applicability of our framework.

\subsubsection{{\bf Composite optimization under relative smoothness}}

The concept of \textit{relative smoothness} is fundamental to a class of algorithms known as Bregman proximal methods \cite{ahookhosh2019accelerated,bauschke2017descent,chen1993convergence,dragomir2022optimal,lu2018relatively,takahashi2022new,takahashi2025approximate}. These methods are particularly valuable for problems where the smooth component $f$ in \eqref{eq:mainproblemcom} does not possess a Lipschitz or H\"{o}lder continuous gradient. The relative smoothness property, also referred to as the \textit{Lipschitz-like/convexity condition} \cite{bauschke2017descent}, is the cornerstone of the Bregman descent lemma \eqref{eq:bregmanDescentLemma}.
While Bregman methods provide a powerful tool for such problems, we demonstrate that these problems can also be addressed within our proposed framework. We show that any optimization problem of the form \eqref{eq:mainproblemcom} with a relatively smooth function $f$ can be reformulated in such a way that its new smooth component satisfies the high-order descent lemma, i.e., it belongs to the class $\maj{L_p}{p}{\R^n}$ for some $p > 1$.

We begin by formally recalling the definition of a relatively smooth function.
\begin{definition}[Relative smoothness]\label{def:relsmo}
Let $\gh, h:\R^n \to \R$ be Fr\'{e}chet differentiable functions, with $h$ convex. The function $\gh$ is said to be
\textit{relative smooth relative to $h$} if there exists a constant $L> 0$ such that the function $x\mapsto Lh(x)-\gh(x)$ is convex, i.e., this function is called \textit{$L$-smooth} related to $h$.
\end{definition}
Now, consider an optimization problem of the form \eqref{eq:mainproblemcom} where $f$ is $L$-smooth related to a convex function $h$. 
The original problem can be rewritten as
\begin{equation}\label{eq:rel}
{\mathop {\mathrm{\bs\min}}\limits_{x\in \R^n}}\ \gf (x) := -(Lh(x) -f(x)) + \left(g(x)+Lh(x)\right).
\end{equation}
By defining $\widetilde{f}:=-(Lh(x) -f(x))$, Theorem~\ref{th:parandfmaj:con}~\ref{th:parandfmaj:b} implies $\widetilde{f}\in \maj{L_p}{p}{\R^n}$ for any $p>1$ and $L_p>0$ as $\widetilde{f}$ is concave. With the new nonsmooth part defined as $\widetilde{g}:= g(x)+Lh(x)$, the problem is now in the form
\begin{equation}\label{eq:rel2}
{\mathop {\mathrm{\bs\min}}\limits_{x\in \R^n}}\ \gf (x) := \widetilde{f}(x)+ \widetilde{g}(x),
\end{equation}
where $\widetilde{f}$ satisfies the high-order descent lemma. This reformulated problem is amenable to the methods developed in this paper.

In the following example, we illustrate the advantage of this approach.
\begin{example}[Matrix factorization]\label{ex:relsmooth:pp}
Matrix factorization is a fundamental technique in machine learning, data mining, and signal processing, where the goal is to approximate a given data matrix $X\in\R^{m\times n}$ by a product of two low-rank matrices. 
Let the target rank $r$ satisfy $r \ll \bs\min\{m,n\}$. We then consider the optimization problem
\begin{equation}\label{eq:mainp:ex:relsmooth:pp}
{\mathop {\mathrm{\bs\min}}\limits_{U\in\R^{m\times r},\;V\in\R^{n\times r}}}\; \gf(U,V) \;=\; f(U,V) + g(U,V),
\end{equation}
where 
$f(U,V)=\frac{1}{2}\Vert X- U V^\top\Vert_F^2$ and $\Vert\cdot\Vert_F$ denotes the Frobenius norm.  
The term $f(U,V)$ measures the reconstruction error between $X$ and its approximation $U V^\top$, while $g(U,V)$ introduces structural constraints or regularization.
 For instance, for the \textit{nonnegative matrix factorization} (NMF) one imposes nonnegativity constraints
\[
 g(U,V) := \iota_{\{U\ge 0\}}(U) + \iota_{\{V\ge 0\}}(V),
\]
where $\iota_C$ is the indicator of a set $C$.
As another example, for the \textit{regularized NMF}, one adds smooth or nonsmooth penalties as
\[
 g(U,V) := \iota_{\{U\ge 0\}}(U) + \iota_{\{V\ge 0\}}(V)
 + \gh_1(U)+ \gh_2(V),
\]
where $\gh_1$ and $\gh_2$ are regularizers promoting specific structures (e.g., sparsity).

For problem \eqref{eq:mainp:ex:relsmooth:pp}, the gradient of $f$ is not \textit{globally} Lipschitz on $\R^{mr+nr}$. Thus, the classical proximal gradient or forward-backward algorithms, relying on a global Lipschitz constant, cannot be applied directly.

Consider the convex kernel
\begin{equation}\label{eq:ex:relsmooth:pp:1}
 h(U,V) := \frac{a}{4}\left(\Vert U\Vert_F^2 + \Vert V\Vert_F^2\right)^2+
 \frac{b}{2}\left(\Vert U\Vert_F^2 + \Vert V\Vert_F^2\right),\qquad a>0,\;b>0.
\end{equation}
Then $f$ is $1$-smooth related to $h$ provided that $a\geq 3$ and $b\geq \Vert X\Vert_F$ \cite[Proposition~2.1]{Mukkamala2019Beyond}.
Consequently, the function 
$h-f$ is convex, and problem \eqref{eq2:mainp:ex:relsmooth:pp} can be equivalently expressed as
\begin{equation}\label{eq2:mainp:ex:relsmooth:pp}
{\mathop {\mathrm{\bs\min}}\limits_{U\in\R^{m\times r},\;V\in\R^{n\times r}}}\; \gf(U,V) \;=\; -(h(U,V)-f(U,V)) + \left(g(U,V)+h(U,V)\right).
\end{equation}
Since $\widetilde{f}(U,V)=-(h(U,V)-f(U,V))$ is concave, 
$\widetilde{f}\in \maj{L_p}{p}{\R^{m\times r}\times \R^{n\times r}}$ for any $p>1$ and $L_p>0$. 
Note that even with this reformulation, the gradient of $\widetilde{f}$ is not globally Lipschitz on $\R^{mr+nr}$, 
and therefore, a classical proximal gradient method with a global Lipschitz step-size cannot be applied to solve 
\eqref{eq2:mainp:ex:relsmooth:pp}. 
Nevertheless, the proposed framework in this paper enables the construction of a \textit{high-order majorant} for $\widetilde{f}$ 
with an arbitrary order $p>1$. 
This flexibility allows the development of iterative schemes of the form \eqref{eq:fbsGeneral:p}, 
where, in particular, choosing $p=2$ yields a practical algorithm for solving 
\eqref{eq:mainp:ex:relsmooth:pp}; see Subsection~\ref{subsec:RegularizedNMF}.
\end{example}
\begin{remark}
Let in Example~\ref{ex:relsmooth:pp}, 
 $\gh_1$ be $\rho_1$-weakly convex, $\gh_2$ be $\rho_2$-weakly convex,
$\widehat{\rho}:=\bs\max\{\rho_1, \rho_2\}$, and suppose
$b\geq \widehat{\rho}$.
Setting $H(U,V):=\frac{b}{2}\left(\Vert U\Vert_F^2 + \Vert V\Vert_F^2\right)$,
both
$g(U,V)+ H(U,V)$ and $h(U,V)-H(U,V)$ are convex.
Hence $g(U,V)+h(U,V) =\left(g(U,V)+ H(U,V)\right)+\left(h(U,V)-H(U,V)\right)$ is convex.
Choose $a\geq 3$ and $b\geq\bs\max\{\Vert X\Vert_F, \widehat{\rho}\}$. Then, $h-f$ and $g+h$ are convex. Hence, \eqref{eq2:mainp:ex:relsmooth:pp} can be rewritten as
\begin{equation}\label{eq2:mainp:ex:relsmooth:pp}
{\mathop {\mathrm{\bs\min}}\limits_{U\in\R^{m\times r},\;V\in\R^{n\times r}}}\; \gf(U,V) \;=\; -(h(U,V)-f(U,V)) + \left(g(U,V)+h(U,V)\right).
\end{equation}
Since $\widetilde{f}(U,V)=-(h(U,V)-f(U,V))$ is concave, 
$\widetilde{f}\in \maj{L_p}{p}{\R^{m\times r}\times \R^{n\times r}}$ for any $p>1$ and $L_p>0$. Moreover, $\widetilde{g}(U,V)=-(g(U,V)+h(U,V))$ is proper, lsc, and convex. Hence, its proximal mapping
$\prox{\widetilde{g}}{\gamma}{p}$ is single-valued for every $\gamma>0$ \cite[Proposition 12.15]{Bauschke17}.
\end{remark}
\begin{remark}
As discussed in Section~\ref{intro}, forward–backward schemes of the form \eqref{eq:fbsGeneral}
that use $\zeta(y,x^k)=D_h(y,x^k)$ typically require a step-size
$\gamma\in(0,L^{-1})$, where $L$ is the relative smoothness constant of $f$ with respect to a kernel $h$.
In the high–order setting, this becomes the condition
$\gamma\in(0,L_p^{-1})$ when $f\in\maj{L_p}{p}{\R^n}$; see, e.g., Algorithm~\ref{alg:inexact}.
After the reformulation \eqref{eq:rel2}, however, the step-size constraint no longer depends on an
\textit{a priori} global constant $L$, since $\widetilde{f}\in\maj{L_p}{p}{\R^n}$ for any $L_p>0$.
Hence, by choosing a smaller $L_p$, the admissible range for $\gamma$ widens; see Subsection~\ref{subsec:RegularizedNMF} for this advantage.
\end{remark}

\subsubsection{{\bf Constrained optimization}}
A constrained optimization problem of the form $\bs\min_{x\in C} f(x)$, where $f:\R^n\to\R$ is Fr\'{e}chet differentiable and $C\subseteq\R^n$ is nonempty and closed, can be reformulated as the unconstrained composite problem
\begin{equation}\label{eq:optimlogexp:conc2}
 {\mathop {\mathrm{\bs\min}}\limits_{x\in \R^n}}\ \gf (x) :=f(x) + \bs{\rm Ind}_C(x),
\end{equation}
where $\bs{\rm Ind}_C$ is the indicator function of the set $C$.
This fits the structure of \eqref{eq:mainproblemcom} with the nonsmooth part $g(x) = \bs{\rm Ind}_C(x)$.
The problem is amenable to our framework whenever the smooth function $f$ satisfies the high-order descent lemma. This occurs in several important scenarios, including the case when $f$ has a H\"{o}lder continuous gradient or, more generally, when $f$ is $p$-paraconcave for $p\in(1,2]$; see, e.g., \cite{Cartis2017Worst,Vinod2022Constrained,yashtini2016global}.
Furthermore, when $f$ is concave, it satisfies the descent lemma for any $p>1$ and $L_p>0$. This case is central to concave minimization \cite{delpia2022proximity,Konno2001,mangasarian1996machine,Mangasarian2007Absolute,rinaldi2010concave,yu2023strong}.

A key advantage of our framework is its flexibility: the exponent $p$ in regularization can be tailored to the properties of $f$. For instance, we may set $p=1+\nu$ if $f \in \mathcal{C}^{1, \nu}_{L_{\nu}}(\R^n)$, or choose any $p>1$ if $f$ is concave. The following example illustrates a practical instance that fits this structure.
\begin{example}[Sparsity maximization via concave minimization]\label{ex:constprob}
Consider the problem of finding the sparsest point in a polyhedron $\mathcal{P} \subseteq \R^n$ as
\begin{equation}\label{eq:l0prob}
    {\mathop {\mathrm{\bs\min}}\limits_{x\in \R^n}}\ \Vert x\Vert_0\qquad\text{subject to}\qquad x\in \mathcal{P},
\end{equation}
where $\Vert x\Vert_0$ denotes  the number of nonzero components of $x$. This NP-hard combinatorial problem appears in machine learning, pattern recognition, and signal processing \cite{rinaldi2010concave}. 
A common strategy approximates the discontinuous $\ell_0$-norm by a smooth concave surrogate, leading to
\begin{equation}\label{eq:l0prob:conc}
{\mathop {\mathrm{\bs\min}}\limits_{x, y\in \R^n}}\ \gh(x,y):=\sum_{i=1}^n(1-e^{-\alpha y_i})\qquad\text{subject to}\qquad x\in \mathcal{P},\ \ \ -y_i\leq x_i\leq y_i,\ i=1, \ldots, n,
\end{equation}
where $\alpha > 0$ controls the quality of the approximation. The function $\gh$ is continuously differentiable and, more importantly, concave.
Hence, $\gh(x,y)\in \maj{L_p}{p}{\R^{2n}}$ for any $p>1$ and $L_p>0$. Consequently, this concave approximation of a hard combinatorial problem is directly amenable to the high-order forward-backward methods developed in Section~\ref{sec:HiFBEMET}.
\end{example}

\subsubsection{{\bf Difference of functions optimization}}
Optimization problems whose objective is expressed as the difference of two functions, particularly with convex components, are widely studied \cite{AragonArtacho2018Accelerating,deOliveira2020ABC,HiriartUrruty1985dc}.
We consider this class within our framework. Let $f:\R^n\to \R$ be Fr\'{e}chet differentiable and $g:\R^n\to \Rinf$ be proper and lsc. Consider
\begin{equation}\label{eq:optimlogexp:dc}
 {\mathop {\mathrm{\bs\min}}\limits_{x\in \R^n}}\ \gf (x) := -f(x) + g(x).
\end{equation}
This fits \eqref{eq:mainproblemcom} by identifying the smooth part with $-f(x)$ and the nonsmooth part with $g(x)$. The key requirement for our methods is that the smooth component $-f(x)$ satisfies the high-order descent lemma. We show that several common assumptions on $f$ (such as convexity, weak convexity, or relative smoothness) ensure this:

\begin{description}[wide, labelwidth=!, labelindent=0pt]
    \item[{\bf(i)}] {\bf  Difference of convex (DC) programming.}
    If both $f$ and $g$ are convex, then \eqref{eq:optimlogexp:dc} is a classical DC program, and the smooth component $-f(x)$ is concave. By Theorem~\ref{th:parandfmaj:con}~\ref{th:parandfmaj:p0}, any concave function belongs to $\maj{L_p}{p}{\R^n}$ for all $p>1$ and all $L_p > 0$. As such, the objective possesses a high-order majorant.

    \item[{\bf(ii)}] {\bf  Difference with a weakly convex function.} Suppose $f$ is weakly convex.
    Then $-f(x)$ is paraconcave with $p=2$ \cite{Jourani,Rahimi2024,Rolewicz00}. Consequently, $-f \in \maj{L_2}{2}{\R^n}$ for some $L_2 > 0$, and the problem again fits our framework. This case is particularly relevant in recent works where $f$ is weakly convex and $g$ need not be convex \cite{AragonArtacho2024Coderivative,ksenia2024Minimizing,deOliveira2025Progressive}.

     \item[{\bf(iii)}] {\bf  Difference with a relatively smooth function.} Suppose $f$ is relatively smooth with respect to a convex kernel $h$; meaning that, there exists $L>0$ such that $\gh(x):= Lh(x) - f(x)$ is convex. Then we may rewrite \eqref{eq:optimlogexp:dc} as
    \begin{equation}\label{eq:optimlogexp:dc:rg}
     {\mathop {\mathrm{\bs\min}}\limits_{x\in \R^n}}\ \gf (x) :=g(x)+ \gh(x) - Lh(x).
    \end{equation}
   Identifying the smooth part as $\widetilde{f}(x):= -Lh(x)$ and the nonsmooth part as $\widetilde{g}(x):= g(x) + \gh(x)$ results in the desired structure. Since $h$ is convex and $L>0$, the function $\widetilde f$ is concave and thus belongs to $\maj{L_p}{p}{\R^n}$ for any $p>1$ and $L_p> 0$.
\end{description}

Let us provide a practical example of this reformulation strategy.

\begin{example}[Generalized phase retrieval and DC decomposition]\label{ex:phaseret}  
A generalized version of the phase retrieval problem aims to recover a signal $x \in \R^n$ from nonlinear and nonconvex measurements. Given measurement vectors $a_i \in \R^n$ and measurements $b_i \geq 0$, the goal is to find an $x$ such that
\begin{equation}\label{eq:P_ret_1}
\theta(\langle a_i, x \rangle) \approx b_i, \qquad i=1, \ldots, m,
\end{equation}
where $\theta: \R \to [0, +\infty)$ is a coercive, convex, and twice continuously differentiable function \cite{huang2021dc}. The standard phase retrieval problem corresponds to the special case where $\theta(t) = t^2$.
A least-squares approach, augmented with a regularizer $\gh(x)$, leads to the nonconvex optimization problem
\begin{equation}\label{eq:P_ret_2}
{\mathop {\mathrm{\bs\min}}\limits_{x\in \R^n}}\ \gf (x) :=  \sum _{i=1}^m\left( \theta(\langle a_i, x \rangle) - b_i \right) ^2 +\lambda \gh(x),
\end{equation}
where $\lambda > 0$ is a regularization parameter \cite{bolte2018first,takahashi2022new}.
The data fidelity term, $H(x) := \sum_{i=1}^m (\theta(\langle a_i, x \rangle) - b_i)^2$, is nonconvex. By expanding the square, $H(x)$ can be expressed as the difference of two convex functions, i.e., $H(x) = H_1(x) - H_2(x)$, where
$H_1(x) = \sum_{i=1}^m \left( \theta(\langle a_i, x \rangle)^2 + b_i^2 \right)$ and 
$H_2(x)= 2 \sum_{i=1}^m b_i \theta(\langle a_i, x \rangle)$.
We can fit the problem \eqref{eq:P_ret_2} into our framework by identifying the components as
$f(x) := -H_2(x)$ and $g(x) := H_1(x) + \lambda \gh(x)$.
Since $H_2(x)$ is convex, $f(x)$ is concave. Consequently, $f(x)$ satisfies the high-order descent lemma and belongs to $\maj{L_p}{p}{\R^n}$ for any $p > 1$ and any $L_p > 0$. Furthermore, $g$ is proper and lsc if the regularizer $\gh(x)$ is.
This decomposition recasts the generalized phase retrieval problem into the form \eqref{eq:mainproblemcom}, where the smooth component is concave. Hence, the problem is directly accessible by the high-order forward-backward methods developed in this paper.
\end{example}

\subsubsection{{\bf H\"{o}lderian optimization problems}}
For a function $f\in \mathcal{C}^{1, \nu}_{L_{\nu}}(\R^n)$ with $\nu\in (0, 1]$ and $L_\nu>0$,
the high-order descent lemma holds with $p=1+\nu$. Consequently, for any proper lsc function $g:\R^n\to\Rinf$, the composite objective $\gf(x):= f(x)+g(x)$ admits a majorant of the form \eqref{eq:majforcom}
which aligns with the structure of \eqref{eq:mainproblemcom} and is compatible with the optimization methods proposed in this paper.
This class of composite optimization problems has attracted significant attention in recent years due to its versatility and applicability
\cite{Berger2020Quality,bolte2023backtrack,Cartis2017Worst,Nesterov15univ,yashtini2016global}.

To illustrate the applicability of our framework, we consider the case of a regularized linear inverse problem, a common instance of H\"{o}lderian optimization. 
\begin{example}[Regularized linear inverse problem]\label{ex:RLInv}
A prototypical example of H\"{o}lderian optimization is the regularized linear inverse problem, often arising in signal reconstruction. Here, the goal is to recover a sparse signal $x \in \R^n$ from noisy measurements $b \in \R^m$, 
modeled as $Ax = b + \nu$, where $A \in \R^{m \times n}$ and $\nu \in \R^m$ represent noise. 
This can be formulated as
\begin{equation}\label{eq:lininv:gform}
    {\mathop {\mathrm{\bs\min}}\limits_{x\in \R^n}}\  \frac{1}{q} \Vert Ax - b\Vert^q_q +  \lambda R(x),
\end{equation}
where $q \in (1, 2]$, $R: \R^n \to \Rinf$ is a proper lsc regularizer, and $\lambda > 0$ is a regularization parameter. The function 
$x\mapsto \frac{1}{q}\Vert x\Vert_q^q=\sum_{i=1}^n|x_i|^q$
belongs to $\maj{L_p}{p}{\R^n}$ with $p=q$
\cite[Lemma~3.1]{bourkhissi2025convergence}. 
Thus, the smooth component $f(x):=\frac{1}{q} \Vert Ax - b\Vert^q_q$ also belongs to $\maj{L_p}{p}{\R^n}$ with $p=q$.

When $p = 2$, this reduces to least-squares regression, which is well-suited for Gaussian noise. For $p < 2$; however, the formulation is more robust to outliers, making it particularly effective under heavy-tailed noise distributions such as Laplace noise \cite{chierichetti2017algorithms,li2023ell_p}. Hence, these optimization problems conform to the structure of \eqref{eq:mainproblemcom}, with the smooth component satisfying the high-order descent lemma, 
and are directly amenable to the high-order forward-backward methods developed in this paper.
\end{example}



\section{High-order forward-backward envelope}\label{sec:hifbe}

In this section, we introduce the \textit{high-order forward-backward splitting mapping} (HiFBS) and the \textit{high-order forward-backward envelope} (HiFBE), which are central to our analysis. We establish essential properties needed to develop and analyze the forward–backward methods proposed in Section~\ref{sec:HiFBEMET}.
Further properties, such as conditions for the differentiability of HiFBE, are beyond the scope of this study and will be explored in future work.

We are at the point of formally defining the key concepts of this paper.
\begin{definition}[High-order forward-backward splitting mapping and envelope]\label{def:HFBE}
Let $p> 1$, $\gamma>0$, $f:\R^n\to \R$ be Fr\'{e}chet differentiable, and $g:\R^n\to \Rinf$ be a proper lsc function.
The \textit{high-order forward-backward splitting mapping} (HiFBS) of $\gf(x):=f(x)+g(x)$ with parameter $\gamma$, $ \Tprox{\gf}{\gamma}{p}: \R^n \rightrightarrows \R^n$, is defined as
\begin{equation}\label{HFBM}
\Tprox{\gf}{\gamma}{p} (x):=\argmint{y\in \R^n}\left\{f(x)+\langle \nabla f(x) , y - x\rangle +g(y)+\frac{1}{p\gamma}\Vert x-y\Vert^p\right\}.
\end{equation}
 The \textit{high-order forward-backward envelope} (HiFBE) of $\gf$ with parameter $\gamma$, $\fgam{\gf}{p}{\gamma}: \R^n\to \R\cup\{\pm \infty\}$, is defined as
\begin{equation}\label{HFBE}
 \fgam{\gf}{p}{\gamma}(x):=\mathop{\bs{\inf}}\limits_{y\in \R^n}\left\{f(x)+\langle \nabla f(x) , y - x\rangle +g(y)+\frac{1}{p\gamma}\Vert x-y\Vert^p\right\}.
\end{equation}
\end{definition}

For $p=2$, HiFBE reduces to the standard forward-backward envelope (FBE) as defined in \cite[Definition 2.1]{Stella17}. 
The choice of regularizer power $p$ in Definition~\ref{def:HFBE} is dictated by the smoothness properties of $f$,
 particularly its membership in $\maj{L_p}{p}{\R^n}$, which affects the existence of a high-order majorant.
As will be elaborated on in Remark~\ref{rem:powofmajreg}, selecting an inappropriate power may result in the inner term of \eqref{HFBE} not majorizing $\gf$.

In the following, we define the auxiliary function $\ell: \R^n\times\R^n\to\Rinf$ as
\begin{equation}\label{eq:ellfun} 
  \ell(x,y):=f(x)+\langle \nabla f(x) , y - x\rangle +g(y),
\end{equation}
and the residual mapping as $\bs R_{\gamma , p}(x):=x - \Tprox{\gf}{\gamma}{p} (x)$.

The next remark clarifies the relationship between HiFBS at a point $x \in \R^n$ and HOPE evaluated at $x - \gamma \nabla f(x)$. 
 
\begin{remark}\label{rem:Hifbenotmor} 
For $p=2$, \cite[Eq. (1.7)]{Stella17} establishes that $\Tprox{\gf}{\gamma}{p} (x)= \prox{g}{\gamma}{p} (x-\gamma\nabla f(x))$, aligning the forward-backward splitting mapping with the standard proximal operator. 
However, for $p\neq 2$, this relationship between HiFBS at $x$ for $\gf$ and HOPE at $x-\gamma\nabla f(x)$ for $g$ does not hold, even under convexity. As an instance, consider $f, g:\R\to \R$ defined as $f(x)=\frac{1}{2}x^2$ and $g(x)=x^2$, with $\ov{x}=1$, $\gamma=2$, and $p=3$. It follows that
\[
\Tprox{\gf}{\gamma}{p} (\ov{x})=\argmint{y\in \R^n}\left\{\gh_1(y):=f(\ov{x})+\langle \nabla f(\ov{x}) , y - \ov{x}\rangle +g(y)+\frac{1}{p\gamma}\Vert \ov{x}-y\Vert^p\right\}=\left\{3-\sqrt{10}\right\},
\]
and 
\[
\prox{g}{\gamma}{p} (\ov{x}-\gamma\nabla f(\ov{x}))=\argmint{y\in \R^n}\left\{\gh_2(y):=g(y)+\frac{1}{p\gamma}\Vert \ov{x}-\gamma \nabla f(\ov{x})-y\Vert^p\right\}=\left\{-3+2\sqrt{2}\right\}.
\]
This discrepancy is also observable for $p \in (1, 2)$.
Hence, results developed for HOME and HOPE in \cite{Kabgani24itsopt} cannot be applied directly to HiFBE and HiFBS.
Figure~\ref{fig:qpdiff} illustrates the functions involved in this example for clarity.
\end{remark}
\begin{figure}[H]
\centering
    \includegraphics[width=0.8\textwidth]{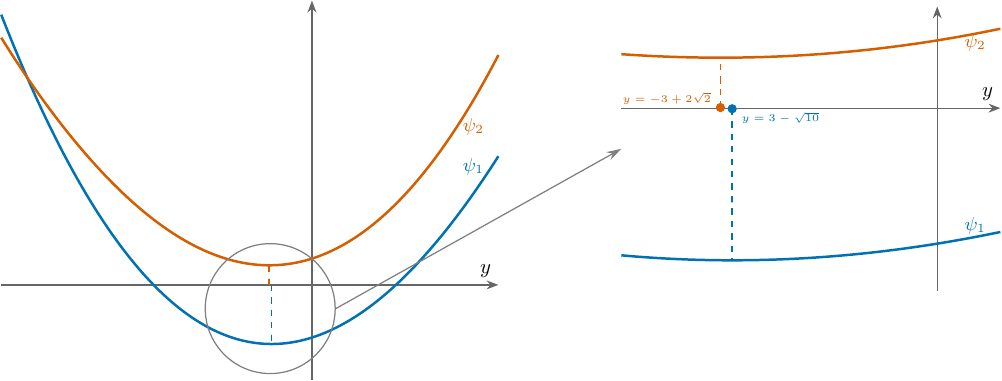}
    \caption{Comparison of minimizers of $\gh_1(y)$ and $\gh_2(y)$ as discussed in Remark~\ref{rem:Hifbenotmor}}
    \label{fig:qpdiff}
\end{figure}

The following theorem establishes fundamental properties of HiFBE.

\begin{theorem}[Fundamental properties of HiFBE]\label{th:basichifbe}
Let $p>1$, $f:\R^n\to \R$ be Fr\'{e}chet differentiable, and $g:\R^n\to \Rinf$ be a proper lsc function. For $\gf(x):=f(x)+g(x)$, the following statements hold:
\begin{enumerate}[label=(\textbf{\alph*}), font=\normalfont\bfseries, leftmargin=0.7cm]
\item \label{th:basichifbe:dom} $\dom{\fgam{\gf}{p}{\gamma}}=\R^n$ 
and $\fgam{\gf}{p}{\gamma}(x)\leq  \gf(x)$  for each $\gamma>0$ and $x\in \R^n$.
\end{enumerate}
If $f\in \maj{L_p}{p}{\R^n}$, then for each $\gamma\in (0, L_p^{-1})$, 
 \begin{enumerate}[label=(\textbf{\alph*}), font=\normalfont\bfseries, leftmargin=0.7cm, start=2]
\item \label{th:basichifbe:ineqforp} $\gf(y)\leq \ell(x,y)+\frac{1}{p\gamma}\Vert x - y\Vert^p$ for each $x,y\in \R^n$. Moreover, 
if $y\in \Tprox{\gf}{\gamma}{p} (x)$, then $\gf(y)\leq \fgam{\gf}{p}{\gamma}(x)$;
\item \label{th:basichifbe:infimforp} for each $\mu\in (\gamma, L_p^{-1})$, $\bs\inf_{z\in\R^n}\gf(z)\leq \fgam{\gf}{p}{\mu}(x)\leq  \fgam{\gf}{p}{\gamma}(x)\leq  \gf(x)$ for each $x\in \R^n$;
\item \label{th:basichifbe:infim2forp}$\mathop{\bs{\inf}}\limits_{z\in\R^n}\gf(z)=\mathop{\bs{\inf}}\limits_{z\in\R^n}\fgam{\gf}{p}{\gamma}(z)$.
\end{enumerate}
If $g$ is high-order prox-bounded with threshold $\gamma^{g, p}>0$,  then
 \begin{enumerate}[label=(\textbf{\alph*}), font=\normalfont\bfseries, leftmargin=0.7cm, start=5]
\item \label{th:basichifbe:finite} 
for each $x\in \R^n$ and $\gamma\in (0, \gamma^{g, p})$, $\fgam{\gf}{p}{\gamma}(x)$ is finite;

 \item \label{th:basichifbe:levelunif} the function $\Psi(y, x, \gamma):=\ell(x,y)+\frac{1}{p\gamma}\Vert x- y\Vert^p$  is
 level-bounded in $y\in\R^n$ locally uniformly in $(x, \gamma)\in\R^n\times (0, \gamma^{g, p})$. 
Additionally, this function is lsc;

\item \label{th:basichifbe:con} $\fgam{\gf}{p}{\gamma}$ depends continuously on $(x,\gamma)$ in $\R^n\times (0, \gamma^{g, p})$.
 \end{enumerate}
 If there exist $\gamma>0$ and $\ov{x}\in \R^n$ such that $\fgam{\gf}{p}{\gamma}(\ov{x})>-\infty$, then
  \begin{enumerate}[label=(\textbf{\alph*}), font=\normalfont\bfseries, leftmargin=0.7cm, start=8]
 \item \label{th:basichifbe:finproxb} $g$ is high-order prox-bounded.
  \end{enumerate}
\end{theorem}
\begin{proof}
 See Appendix~\ref{sec:app:proofs}.
\end{proof}

In the next remark, we discuss the relationship between the power of the regularizer in HiFBE and HiFBS and the power of the high-order majorant.
\begin{remark}\label{rem:powofmajreg}
If $f \in \maj{L_p}{p}{\R^n}$, Theorem~\ref{th:basichifbe}~\ref{th:basichifbe:ineqforp} guarantees the existence of a majorant for $\gf(x) = f(x) + g(x)$ of the form $\ell(x, y) + \frac{1}{p \gamma} \Vert x - y \Vert^p$, valid for all $x, y \in \R^n$ and $\gamma \in (0, L_p^{-1})$. Employing the same power $p$ for the regularizer in HiFBE and HiFBS is crucial for ensuring practical properties, notably those in Theorem~\ref{th:basichifbe}~\ref{th:basichifbe:infimforp} and \ref{th:basichifbe:infim2forp}.
Specifically, if $f \in \maj{L_p}{p}{\R^n}$, it need not belong to $\maj{L_q}{q}{\R^n}$ for $q \neq p$, implying that using a regularizer power $q$ could invalidate majorization. For instance, consider $f,g: \R\to\R$ defined by $f(x)=\frac{1}{2}x^2$ and $g(x)=0$. Then, $f\in \mathcal{C}^{1, 1}_{L}(\R)$ with $L=1$ and thus $f\in \maj{L_2=1}{2}{\R^n}$. 
For $p\neq 2$ and fix $x=0$, we claim that $\ell(x,y)+\frac{1}{p\gamma}\Vert x- y\Vert^p$ is not a majorant for $\gf$. Indeed, in this case,
$\ell(x,y)+\frac{1}{p\gamma}\Vert x-y\Vert^p=\frac{1}{p\gamma}\Vert y\Vert^p$,
while $\gf(y)=\frac{1}{2}y^2$. Thus, there is no $\gamma>0$ such that $\gf(y)\leq \frac{1}{p\gamma}\Vert y\Vert^p$ for all $y\in \R$. In conclusion,
the power of the high-order majorant dictates the regularizer power in Definition~\ref{def:HFBE}.
\end{remark}

The subsequent result indicates the nonemptiness and outer semicontinuity of the HiFBS operator $\Tprox{\gf}{\gamma}{p}$.
\begin{theorem}[Fundamental properties of HiFBS]\label{th:hopbfb}
Let $p>1$, $f:\R^n\to \R$ be Fr\'{e}chet differentiable, and $g:\R^n\to \Rinf$ be a proper lsc function that is high-order prox-bounded with the threshold $\gamma^{g, p}>0$. For the function $\gf(x):=f(x)+g(x)$, the following statements hold
 \begin{enumerate}[label=(\textbf{\alph*}), font=\normalfont\bfseries, leftmargin=0.7cm]
\item \label{th:hopbfb:proxb:proxnonemp} for each $x\in \R^n$ and $\gamma\in (0, \gamma^{g, p})$,
 $\Tprox{\gf}{\gamma}{p}(x)$ is nonempty and compact;

\item \label{th:hopbfb:proxb:conv} if $y^k\in \Tprox{\gf}{\gamma^k}{p}(x^k)$, $x^k\to \ov{x}$, and $\gamma^k\to \gamma\in (0, \gamma^{g, p})$, then the sequence $\{y^k\}_{k\in \mathbb{N}}$ is bounded and all its cluster points lie in $\Tprox{\gf}{\gamma}{p}(\ov{x})$.
    \end{enumerate} 
\end{theorem}
\begin{proof}
Both results follow from Theorem~\ref{th:basichifbe}~\ref{th:basichifbe:finite}-\ref{th:basichifbe:con}
and \cite[Theorem 1.17]{Rockafellar09}.
\end{proof}

Let us consider problem~\eqref{eq:mainproblemcom} under the next assumptions.
 \begin{assumption}[Main assumptions]\label{assum:mainassum}
For problem \eqref{eq:mainproblemcom}, the following conditions are assumed
\begin{enumerate}[label=(\textbf{\alph*}), font=\normalfont\bfseries, leftmargin=0.7cm]
  \item \label{assum:mainassum:f} the function $f:\R^n\to \R$ is smooth (potentially nonconvex) and possesses a
  \textit{high-order majorant with power $p>1$}; i.e. $f\in \maj{L_p}{p}{\R^n}$ (refer to Definition~\ref{def:power majorization}); 
  \item \label{assum:mainassum:g} the function $g:\R^n\to \Rinf$ is proper, lsc (potentially nonconvex and nonsmooth), and
   \textit{high-order prox-bounded} (refer to Definition~\ref{def:s-prox-bounded});
  \item \label{assum:mainassum:armin} the set of minimizers is nonempty, i.e., $\argmin{x\in \R^n} \gf(x)\neq \emptyset$. We denote a minimizer by $x^*$ and the corresponding minimal value by $\gf^*$.
\end{enumerate}
\end{assumption}
In this paper, we address the nonconvex optimization problem \eqref{eq:mainproblemcom}. Hence, our proposed methods only guarantee convergence to critical points. However, we will demonstrate that these methods can identify fixed points of the operator $\Tprox{\gf}{\gamma}{p}$, which are particular critical points. To facilitate these results, we first clarify the relevant concepts in the following definition.

\begin{definition}[Reference points]\label{def:critic}
Let Assumption~\ref{assum:mainassum} hold. A point $\ov{x}\in \dom{\gf}$  is called
\begin{enumerate}[label=(\textbf{\alph*}), font=\normalfont\bfseries, leftmargin=0.7cm]
  \item \label{def:critic:f} a \textit{Fr\'{e}chet critical point} if $0\in \widehat{\partial}\gf(\ov{x})$, denoted by $\ov{x}\in \bs{{\rm Fcrit}}(\gf)$;
 \item \label{def:critic:m} a \textit{Mordukhovich critical point} if $0\in \partial \gf(\ov{x})$, denoted by $\ov{x}\in \bs{{\rm Mcrit}}(\gf)$;
   \item \label{def:critic:pfix} a \textit{proximal fixed point} if $\ov{x}\in \Tprox{\gf}{\gamma}{p}(\ov{x})$, denoted by $\ov{x}\in \bs{{\rm Fix}}(\Tprox{\gf}{\gamma}{p})$;
\end{enumerate}
\end{definition}

We conclude this section by discussing relationships among reference points and minimizers of $\gf$, highlighting the equivalence among minimizers of $\gf$ and HiFBE, which motivates the study of such envelopes.

\begin{proposition}[Relationships among reference points]\label{prop:relcrit}
Let Assumption~\ref{assum:mainassum} hold. 
Then, for each $\gamma\in (0, \bs\min\{L_p^{-1},\gamma^{g, p}\})$, the following hold
\begin{enumerate}[label=(\textbf{\alph*}), font=\normalfont\bfseries, leftmargin=0.7cm]
\item  \label{prop:relcrit:opfix} $\argmint{x\in \R^n}\gf(x)\subseteq \bs{{\rm Fix}}(\Tprox{\gf}{\gamma}{p})\subseteq
\bs{{\rm Fcrit}}(\gf)\subseteq \bs{{\rm Mcrit}}(\gf)$;

 \item \label{prop:relcrit:inf:arg:argmin}
  $\argmint{x\in \R^n} \gf(x)=\argmint{x\in \R^n} \fgam{\gf}{p}{\gamma}(x)$.
\end{enumerate}
\end{proposition}
\begin{proof}
\ref{prop:relcrit:opfix} Let $\ov{x}\in \argmin{x\in \R^n} \gf(x)$. For any $y\in\R^n$, it holds that
\[
\begin{aligned}
\gf(\ov{x})\leq f(y)+g(y)&\leq f(\ov{x})+\langle \nabla f(\ov{x}), y- \ov{x}\rangle+\frac{L_p}{p}\Vert y- \ov{x}\Vert^p+g(y)
\\&\leq  f(\ov{x})+\langle \nabla f(\ov{x}), y- \ov{x}\rangle+g(y)+\frac{1}{p\gamma}\Vert y- \ov{x}\Vert^p.
\end{aligned}
\]
Thus, $\ov{x}\in \bs{{\rm Fix}}(\Tprox{\gf}{\gamma}{p})$. If $\ov{x}\in  \bs{{\rm Fix}}(\Tprox{\gf}{\gamma}{p})$, then by \cite[Exercise~8.8 and Theorem~10.1]{Rockafellar09},
$0\in  \nabla f(\ov{x}) +  \widehat{\partial} g(\ov{x})$, which implies $0\in \widehat{\partial}\gf(\ov{x})$.
Hence, $\bs{{\rm Fix}}(\Tprox{\gf}{\gamma}{p})\subseteq\bs{{\rm Fcrit}}(\gf)$. The inclusion
$\widehat{\partial}\gf(\ov{x})\subseteq\partial \gf(\ov{x})$ completes the proof.
\\
\ref{prop:relcrit:inf:arg:argmin} Let $\ov{x}\in \argmin{x\in \R^n} \gf(x)$. For any $x,y\in \R^n$, invoking Assertion~\ref{prop:relcrit:opfix} ensures
\[
\fgam{\gf}{p}{\gamma}(\ov{x})=\gf(\ov{x})\leq \gf(y)
\leq  f(x)+\langle \nabla f(x), y- x\rangle+g(y)+\frac{1}{p\gamma}\Vert y- x\Vert^p.
\]
Thus, for each $x\in \R^n$, $\fgam{\gf}{p}{\gamma}(\ov{x})\leq \fgam{\gf}{p}{\gamma}(x)$, which means 
$\argmin{x\in \R^n} \gf(x)\subseteq \argmin{x\in \R^n} \fgam{\gf}{p}{\gamma}(x)$.
Conversely, suppose $\ov{x}\in \argmin{x\in \R^n} \fgam{\gf}{p}{\gamma}(x)$ and
$\ov{y}\in \Tprox{\gf}{\gamma}{p} (\ov{x})$. We claim that, $\ov{x} = \ov{y}$. Otherwise, $L_p<\frac{1}{\gamma}$ implies,
\[
\begin{aligned}
f(\ov{x})+\langle\nabla f(\ov{x}), \ov{y}- \ov{x}\rangle+g(\ov{y})+\frac{1}{p\gamma}\Vert \ov{x}-\ov{y}\Vert^p&=\fgam{\gf}{p}{\gamma}(\ov{x})
=\mathop{\bs{\inf}}\limits_{x\in \R^n}  \fgam{\gf}{p}{\gamma}(x)=\mathop{\bs{\inf}}\limits_{x\in \R^n} \gf(x)\leq \gf(\ov{y})
\\&< f(\ov{x})+\langle\nabla f(\ov{x}), \ov{y}- \ov{x}\rangle+g(\ov{y})+\frac{1}{p\gamma}\Vert \ov{x}-\ov{y}\Vert^p,
\end{aligned}
\]
which is a contradiction. Hence, $\ov{x}\in  \bs{{\rm Fix}}(\Tprox{\gf}{\gamma}{p})$, which implies
$\gf(\ov{x}) = \fgam{\gf}{p}{\gamma}(\ov{x})=\mathop{\bs{\inf}}\limits_{x\in \R^n}  \fgam{\gf}{p}{\gamma}(x)=\mathop{\bs{\inf}}\limits_{x\in \R^n} \gf(x)$, i.e.,  
$ \argmin{x\in \R^n} \fgam{\gf}{p}{\gamma}(x)\subseteq \argmin{x\in \R^n} \gf(x)$.
\end{proof}



\section{Boosted high-order forward-backward methods}
\label{sec:HiFBEMET}

Let us consider the problem~\eqref{eq:mainproblemcom} under Assumption~\ref{assum:mainassum}. A natural way to generate a sequence $\{x^k\}_{k\in \Nz}$ is via a generalized forward-backward splitting method, which we call the {\it high-order forward-backward algorithm}, defined by \eqref{eq:fbsGeneral:p}.
When $p=2$, this reduces to the proximal gradient method, i.e.,
\begin{equation}\label{eq:cl:fbep:prox}
x^{k+1}\in \argmint{y\in\R^n}\left\{g(y)+\frac{1}{2\gamma}\Vert x^k-\gamma \nabla f(x^k)-y\Vert^2\right\}.
\end{equation}
For specific choices of $g$, this problem may admit a closed-form solution. In many practical settings, however, no closed-form solution is available, motivating inexact methods (see, e.g., \cite{Ahookhosh24,barre2023principled,Dvurechenskii2022,Nesterov2023a,rockafellar1976monotone,Salzo12,Solodov01}). 
Moreover, for general $p > 1$, reformulating \eqref{eq:fbsGeneral:p} into a form analogous to \eqref{eq:cl:fbep:prox} is generally not feasible, as discussed in Remark~\ref{rem:Hifbenotmor}. 
These considerations motivate a framework to approximate elements of HiFBS, which we develop in Subsection~\ref{subsec:inexacthifbe}.

Beyond \eqref{eq:fbsGeneral:p}, alternative strategies construct iterates using an element of HiFBS at the current iterate $x^k$ together with a search direction $d^k$. In Subsection~\ref{subsec:bohifbe}, we propose a general framework encompassing such strategies and introduce
the high-order inexact forward-backward algorithm (HiFBA) and its boosted variant,
which use an approximation of $\fgam{\gf}{p}{\gamma}$ as a Lyapunov function to enforce a descent condition. We subsequently analyze their convergence properties, including global convergence and linear rates.

\subsection{{\bf Inexact oracle for HiFBE}}
\label{subsec:inexacthifbe}
Our proposed methods require the following assumptions to construct approximate elements of HiFBS.

\begin{assumption}[Assumptions on inexactness]\label{assum:approx} 
For problem \eqref{eq:mainproblemcom}, we assume
\begin{enumerate}[label=(\textbf{\alph*}), font=\normalfont\bfseries, leftmargin=0.7cm]
\item \label{assum:approx:coer weak con} $p>1$, $f\in \maj{L_p}{p}{\R^n}$, and $g$ is bounded below;
\item \label{assum:approx:eps}  $\{\varepsilon_k\}_{k\in \Nz}$ is a non-increasing sequence of positive scalars with $\ov{\varepsilon}:=\sum_{k=0}^{\infty} \varepsilon_k<\infty$;    
\item \label{assum:approx:aproxep} for a given $x^k\in \R^n$, $\gamma>0$, and $\varepsilon_k> 0$, we can find
a \textit{prox approximation} $\Tprox{\gf}{\gamma}{p, \varepsilon_k}(x^k)$ such that
\begin{align}
&\dist\left(\Tprox{\gf}{\gamma}{p, \varepsilon_k}(x^k), \Tprox{\gf}{\gamma}{p}(x^k)\right)\to 0~\text{as}~ k\to\infty,\label{eq:ep-approx:dist}    
\\
&\ell\left(x^k,\Tprox{\gf}{\gamma}{p, \varepsilon_k}(x^k)\right)+\frac{1}{p\gamma}\Vert x^k-\Tprox{\gf}{\gamma}{p, \varepsilon_k}(x^k)\Vert^p< \fgam{\gf}{p}{\gamma}(x^k)+\varepsilon_k.\label{eq:ep-approx:fun}
\end{align}
\end{enumerate}
\end{assumption}
The {\it inexact function value} of $\fgam{\gf}{p}{\gamma}(x^k)$ is defined by
\begin{equation*}
\fgam{\gf}{p,\varepsilon_k}{\gamma}(x^k):=\ell\left(x^k,\Tprox{\gf}{\gamma}{p, \varepsilon_k}(x^k)\right)+\frac{1}{p\gamma}\Vert x^k-\Tprox{\gf}{\gamma}{p, \varepsilon_k}(x^k)\Vert^p,
\end{equation*}
and we define the \textit{approximated residual} as
\[
R_{\gamma,p}^{\varepsilon_k}(x^k) := x^k-\Tprox{\gf}{\gamma}{p, \varepsilon_k}(x^k).
\]
The next remark provides the concerning Assumption~\ref{assum:approx}.

\begin{remark}\label{rem:onapprox}
\begin{enumerate}[label=(\textbf{\alph*}), font=\normalfont\bfseries, leftmargin=0.7cm]
\item \label{rem:onapprox:a}
Assumption~\ref{assum:approx}~\ref{assum:approx:coer weak con} requires $g$ to be bounded below, a standard condition in minimization. By Remark~\ref{rem:prox-bounded}, this implies $\gamma^{g, p}=+\infty$, so constraints such as $\gamma\in (0, \gamma^{g, p})$ are unnecessary here.

\item \label{rem:onapprox:b}
From \eqref{HFBE} and properties of infima, for any $\varepsilon>0$ there exists $y\in \R^n$ with $\ell\left(x,y\right)+\frac{1}{p\gamma}\Vert x-y\Vert^p< \fgam{\gf}{p}{\gamma}(x)+\varepsilon$. Hence, \eqref{eq:ep-approx:fun} is well-defined. Moreover, for each $k\in \Nz$, we have $\fgam{\gf}{p}{\gamma}(x^k)\leq \fgam{\gf}{p,\varepsilon_k}{\gamma}(x^k)\leq \fgam{\gf}{p}{\gamma}(x^k)+\varepsilon_k$.

\item 
The proximity in \eqref{eq:ep-approx:dist} is a standard requirement in inexact proximal frameworks \cite{fukushima1996globally,Salzo12,rockafellar1976monotone2}. 
For convex problems with $p = 2$, \cite[Lemma~3.1]{fukushima1996globally} shows that \eqref{eq:ep-approx:fun} implies
 $\Vert \Tprox{\gf}{\gamma}{2, \varepsilon_k}(x^k)- \Tprox{\gf}{\gamma}{2}(x^k)\Vert\leq \sqrt{2\gamma_k \varepsilon_k}$. 
 Deriving such a bound for nonconvex functions is non-trivial, as the upper bound for $\dist\left(\Tprox{\gf}{\gamma}{p, \varepsilon_k}(x^k), \Tprox{\gf}{\gamma}{p}(x^k)\right)$ concerning \eqref{eq:ep-approx:fun} depends on the geometry of the function $g$.
We provide an instance to illustrate this. Define the function $\Psi^p(x, y): \R^n \times \R^n \to \Rinf$ as
$\Psi^p(x, y):= \langle\nabla f(x), y\rangle +g(y)+\frac{1}{p\gamma}\Vert x - y\Vert^p$
where $p \in (1, 2]$. Suppose that for each $x^k$, $\Psi^p(x^k, y)$ satisfies a \textit{H\"{o}lderian error bound} of order $\delta \in (0,1]$ with respect to $y$ \cite{bolte2017from,johnstone2020faster,Liao2024error}, with parameters $(\mu_k, \varepsilon_k)$ for a bounded sequence $\{\mu_k\}_{k \in \Nz}$ of positive scalars, i.e.,
\[
\dist\left(y, \Tprox{\gf}{\gamma}{p}(x)\right)^{\frac{1}{\delta}}\leq \mu_k\left(\Psi^p(x^k, y) - \Psi^p(x^k, \ov{y}^k)\right),\qquad \forall y\in \mathcal{L}(\Psi^p(x^k,\cdot), \Psi^p(x^k, \ov{y}^k)+\varepsilon_k),
\]
where $\ov{y}^k \in \Tprox{\gf}{\gamma}{p}(x^k)$ is arbitrary. Leveraging \eqref{eq:ep-approx:fun}, for each iterate $x^k$ and any $\ov{y}^k \in \Tprox{\gf}{\gamma}{p}(x^k)$, ensures
\begin{align*}
 \dist\left(\Tprox{\gf}{\gamma}{p, \varepsilon_k}(x^k), \Tprox{\gf}{\gamma}{p}(x^k)\right)^{\frac{1}{\delta}}\leq \mu_k\left(\Psi^p(x^k, \Tprox{\gf}{\gamma}{p, \varepsilon_k}(x^k)) - \Psi^p(x^k, \ov{y}^k)\right)
=\mu_k\left(\fgam{\gf}{p,\varepsilon_k}{\gamma}(x^k) -\fgam{\gf}{p}{\gamma}(x^k)\right)<\mu_k\varepsilon_k.
\end{align*}
Thus, $\dist\left(\Tprox{\gf}{\gamma}{p, \varepsilon_k}(x^k), \Tprox{\gf}{\gamma}{p}(x^k)\right)<\left(\mu_k\varepsilon_k\right)^{\delta}$, ensuring that $\dist\left(\Tprox{\gf}{\gamma}{p, \varepsilon_k}(x^k), \Tprox{\gf}{\gamma}{p}(x^k)\right)\to 0$ as $k\to \infty$.
For $\delta = \frac{1}{2}$, the H\"{o}lderian error bound is also known as the \textit{quadratic growth condition}. In \cite{Liao2024error}, the equivalence between the quadratic growth condition and several regularity conditions for weakly convex functions is thoroughly investigated. In optimization, such regularity conditions are typically employed to establish linear convergence. 
Exploring the relationships among these regularity conditions for a broader class of functions beyond weak convexity, for $\delta \neq \frac{1}{2}$, and their application to computing inexact elements of HiFBS is beyond the scope of this paper and will be addressed in future work.
\end{enumerate}
\end{remark}

\subsection{{\bf Boosted HiFBA}}
\label{subsec:bohifbe}

We are now ready to introduce HiFBA and a boosted variant (Boosted HiFBA) in Algorithm~\ref{alg:inexact}.
HiFBA follows the approach given in \eqref{eq:fbsGeneral:p}, where $x^{k+1}$ is computed approximately.
In Boosted HiFBA, generating the next iterate $x^{k+1}$ from the current point $x^k$ using a step-size $\alpha_k > 0$ and a search direction $d^k$ plays a key role. Various structures for computing the next solution have been proposed in the literature \cite{Ahookhosh21,AragonArtacho2018Accelerating,Themelis18}. In the theoretical analysis, however, it suffices to impose general properties of these structures to establish convergence. We formalize this in the following definition.

\begin{definition}[Structural iteration]\label{def:structured}
For $x\in \R^n$, let $\ov{x}=\Tprox{\gf}{\gamma}{p, \varepsilon}(x)$ for some $\varepsilon>0$. A mapping
$\st: [0,+\infty)\times\R^n\to \R^n$ is called a \textit{structural iteration} with respect to $(x, \ov{x})$ if it satisfies
\begin{enumerate}[label=(\textbf{\alph*}), font=\normalfont\bfseries, leftmargin=0.7cm]
\item\label{def:structured:a} for each fixed $d\in\R^n$, $\st(\alpha,d)\to \ov{x}$ as $\alpha\downarrow 0$;

\item\label{def:structured:b} $\Vert \st(\alpha,d) - x\Vert\leq \Vert x-\ov{x}\Vert+\Vert d\Vert$.
\end{enumerate}
\end{definition}
   
To illustrate the concept, we give several instances of structural iterations for constructing $x^{k+1}$.  

\begin{example}[Instances of structural iteration]\label{ex:struc}
For $x^k\in \R^n$, let $\ov{x}^k=\Tprox{\gf}{\gamma}{p, \varepsilon}(x^k)$ for some $\varepsilon^k>0$.
We introduce examples illustrating the construction of the next iterate $x^{k+1}$ that adheres to the requirements of a structural iteration.
\begin{enumerate}[label=(\textbf{\alph*}), font=\normalfont\bfseries, leftmargin=0.7cm]
\item\label{ex:struc:a}  Set $x^{k+1}:=\st(\alpha_k, d^k)= \ov{x}^k$. This is the standard proximal update (cf. \eqref{eq:fbsGeneral:p}). In this paper, we refer to this structure as HiFBA.

\item\label{ex:struc:b} Fix $\vartheta\in (0,1)$ and let $\alpha_k = \vartheta^m$, for $m\in \Nz$, initialized at $m=0$. Some algorithms set $\widehat{x}^{k+1}:=\st(\alpha_k, d^k)=\ov{x}^{k}+\alpha_k d^k$ and increase $m$ in an inner loop until a descent condition holds for $\alpha_k=\vartheta^{\ov{m}}$, then take $x^{k+1}=\widehat{x}^{k+1}=\ov{x}^{k}+\vartheta^{\ov{m}} d^k$.
Since $\alpha_k=\vartheta^{m}\downarrow 0$ as $m\to\infty$ and $\alpha_k<1$  for any $m\in \Nz$, the conditions of the structural iteration are satisfied for this $\st$.

\item\label{ex:struc:c} Similar to part~\ref{ex:struc:b} in finding $\alpha_k$, another approach is
$\widehat{x}^{k+1}:=\st(\alpha_k, d^k)=(1-\alpha_k)\ov{x}^{k}+\alpha_k(x^k+d^k)$.
Indeed,
\[
\Vert \st(\alpha^k,d^k) - x^k\Vert=\Vert (1-\alpha_k)\ov{x}^{k}+\alpha_k(x^k+d^k)-x^k\Vert   \overset{(i)}{\leq}  \Vert x^k-\ov{x}^k\Vert+\Vert d^k\Vert,
\]
 where $(i)$ follows from $(1-\alpha_k), \alpha_k\leq 1$. Thus, the conditions of the structural iteration are satisfied for $\st$.
\end{enumerate}
\end{example}

In what follows, we adopt a general structural iteration $\st$ and defer concrete implementations to Section~\ref{sec:numerical}. In Algorithm~\ref{alg:inexact}, we do not impose a descent condition on $d^k$; possible choices for $d^k$ are discussed in Section~\ref{sec:numerical}.


\begin{algorithm}
\caption{Boosted HiFBA (Boosted High-Order Inexact Forward-Backward Algorithm)}\label{alg:inexact}
\begin{algorithmic}[1]
\State \textbf{Initialization} Start with $x^0\in \R^n$, $\gamma\in \left(0, L_p^{-1}\right)$, $\sigma\in \left(0, \frac{1-\gamma L_p}{p\gamma}\right)$, $\vartheta\in (0,1)$, and set $k=0$;
\While{stopping criteria do not hold}
\State Compute $\ov{x}^{k}=\Tprox{\gf}{\gamma}{p, \varepsilon_k}(x^k)$ and $R_{\gamma,p}^{\varepsilon_k}(x^k)$; \Comment{lower-level}
\State Choose a direction $d^k\in \R^n$ and set $m=0$;  \Comment{upper-level};
\Repeat\label{alg:hippa:strep}
\State\label{alg:hippa:dir} Set $\alpha_k=\vartheta^m$ and
$\widehat{x}^{k+1}:=\st(\alpha_k,d^k),~~m=m+1$; \Comment{upper-level}
\State Compute $\Tprox{\gf}{\gamma}{p, \varepsilon_{k+1}}(\widehat{x}^{k+1})$ and $\fgamepsko{\gf}{p}{\gamma}(\widehat{x}^{k+1})$; \Comment{lower-level}
\Until{
\begin{equation}\label{eq:alg:ingrad:upd}
\fgamepsko{\gf}{p}{\gamma}(\widehat{x}^{k+1})\leq
\fgamepsk{\gf}{p}{\gamma}(x^k)-\sigma\Vert R_{\gamma,p}^{\varepsilon_k}(x^k)\Vert^{p} +\varepsilon_{k}+\varepsilon_{k+1},
\end{equation}
}\label{alg:hippa:endrep}
\State $x^{k+1}=\widehat{x}^{k+1}$;~$k=k+1$;
\EndWhile
\end{algorithmic}
\end{algorithm}

First, we establish the well-definedness of Algorithm~\ref{alg:inexact}, ensuring that the condition in Step~\ref{alg:hippa:endrep} is satisfied after a finite number of backtracking steps.

\begin{theorem}[Well-definedness of Boosted HiFBA]\label{th:welldefalg}
Let $\{x^k\}_{k\in \Nz}$ be the sequence generated by Algorithm~\ref{alg:inexact}. If $\ov{x}^{k}\neq x^k$, then there exists $\ov{m}\in \Nz$ for which \eqref{eq:alg:ingrad:upd} holds, equivalently, the inner loop in Step~\ref{alg:hippa:endrep} terminates after finitely many backtracking steps.
\end{theorem}
\begin{proof}
From \eqref{eq:ep-approx:fun}, Remark~\ref{rem:onapprox}~\ref{rem:onapprox:b}, and $\sigma\in \left(0, \frac{1-\gamma L_p}{p\gamma}\right)$, we obtain
\begin{align}
\fgam{\gf}{p}{\gamma}(\ov{x}^{k})\leq \gf(\ov{x}^{k})&=f(\ov{x}^k)+g(\ov{x}^k)
\leq f(x^k)+\langle \nabla f(x^k), \ov{x}^k-x^k\rangle+g(\ov{x}^k)+\frac{L_p}{p}\Vert x^k-\ov{x}^k\Vert^p
\nonumber\\&\leq \fgam{\gf}{p}{\gamma}(x^{k})-\left(\frac{1}{p\gamma}-\frac{L_p}{p}\right)\Vert x^k-\ov{x}^k\Vert^p+\varepsilon_k
<\fgamepsk{\gf}{p}{\gamma}(x^k)-\sigma\Vert R_{\gamma,p}^{\varepsilon_k}(x^k)\Vert^{p}+\varepsilon_k.\label{eq:th:welldefalg}
\end{align}
Since $\st(\vartheta^m,d^k)\to \ov{x}^{k}$ as $m\to \infty$, $\fgam{\gf}{p}{\gamma}$ is continuous, and \eqref{eq:th:welldefalg} holds, there exists $\ov{m}\in \Nz$ such that
\begin{equation}\label{eq:th:welldefalg:b}
\fgam{\gf}{p}{\gamma}\left(\st(\vartheta^{\ov{m}},d^k)\right)\leq \fgamepsk{\gf}{p}{\gamma}(x^k)-  \sigma\Vert  R_{\gamma,p}^{\varepsilon_k}(x^k)\Vert^{p}+\varepsilon_k.
\end{equation}
Setting $\widehat{x}^{k+1}:=\st(\alpha_k, d^k)$ with $\alpha_k=\vartheta^{\ov{m}}$, we use \eqref{eq:ep-approx:fun} to obtain $\fgamepsko{\gf}{p}{\gamma}(\widehat{x}^{k+1})-\varepsilon_{k+1}\leq \fgam{\gf}{p}{\gamma}(\widehat{x}^{k+1})$. Combining this with \eqref{eq:th:welldefalg:b} yields \eqref{eq:alg:ingrad:upd}, i.e., the inner loop is terminated in a finite number of backtracking steps.
\end{proof}

Next, we analyze the cluster points of the sequence generated by Algorithm~\ref{alg:inexact}. 
For clarity, we distinguish $\ov{y}^k\in \Tprox{\gf}{\gamma}{p}(x^k)$ such that
\begin{equation}\label{eq:defybar}
  \Vert \ov{x}^{k}-\ov{y}^k\Vert=\dist\left(\ov{x}^{k}, \Tprox{\gf}{\gamma}{p}(x^k)\right).
\end{equation}
The well-definedness of \eqref{eq:defybar} follows from the nonemptiness and compactness of
$ \Tprox{\gf}{\gamma}{p}(x^k)$, as established in Theorem~\ref{th:hopbfb}~\ref{th:hopbfb:proxb:proxnonemp}.   

\begin{theorem}[Subsequential convergence]\label{th:comcon}
Let $\{x^k\}_{k\in \Nz}$ and $\{\ov{x}^{k}\}_{k\in \Nz}$ be generated by Algorithm~\ref{alg:inexact}, with $\{\ov{y}^{k}\}_{k\in \Nz}$ satisfying \eqref{eq:defybar}. 
Then,
\begin{enumerate}[label=(\textbf{\alph*}), font=\normalfont\bfseries, leftmargin=0.7cm]
\item \label{th:comcon:resconv} $\sum_{k=0}^{\infty}\Vert R_{\gamma,p}^{\varepsilon_k}(x^k)\Vert^{p}<+\infty$;
\item \label{th:comcon:cluster} the sequences  $\{x^k\}_{k\in \Nz}$, $\{\ov{x}^{k}\}_{k\in \Nz}$, and $\{\ov{y}^{k}\}_{k\in \Nz}$ have identical cluster points, and every cluster point is a proximal fixed point.
\end{enumerate}
\end{theorem}
\begin{proof}
$\ref{th:comcon:resconv}$ 
From inequalities \eqref{eq:ep-approx:fun}  and \eqref{eq:alg:ingrad:upd}, for any $K\in \mathbb{N}$, we derive
\begin{align*}\label{eq:sum:th:comco}
\sigma \sum_{k=0}^{K-1}&\Vert R_{\gamma,p}^{\varepsilon_k}(x^k)\Vert^{p}\leq \fgam{\gf}{p,\varepsilon_0}{\gamma}(x^0)-\fgam{\gf}{p,\varepsilon_K}{\gamma}(x^{K})+ \sum_{k=0}^{K-1}(\varepsilon_k+\varepsilon_{k+1})
\nonumber\\
&\leq \fgam{\gf}{p}{\gamma}(x^0)+\varepsilon_0-\fgam{\gf}{p}{\gamma}(x^K)+ \sum_{k=0}^{K-1}(\varepsilon_k+\varepsilon_{k+1})
\overset{(i)}{\leq} \fgam{\gf}{p}{\gamma}(x^0)-\gf^*+ 2\ov{\varepsilon},
\end{align*}
where the inequality $(i)$ uses
$\bs\inf_{z\in\R^n}\gf(z)=\bs\inf_{z\in\R^n}\fgam{\gf}{p}{\gamma}(z)$. Since the right-hand side is finite and independent of $K$, letting $K \to \infty$ completes the proof.
\\
$\ref{th:comcon:cluster}$ From Assertion~$\ref{th:comcon:resconv}$ and \eqref{eq:ep-approx:dist}, $\Vert x^k - \ov{x}^k\Vert\rightarrow 0$ and  
$\Vert \ov{y}^k - \ov{x}^k\Vert\to 0$, hence
  $\{x^k\}_{k\in \Nz}$, $\{\ov{x}^{k}\}_{k\in \Nz}$, and $\{\ov{y}^k\}_{k\in \Nz}$ share the same cluster points.
  Let $\widehat{x}$ be any cluster point with corresponding subsequences $\{x^j\}_{j\in J\subseteq \Nz}$ and $\{\ov{y}^j\}_{j\in J}$.
   By Theorem~\ref{th:hopbfb}~$\ref{th:hopbfb:proxb:conv}$, $\widehat{x}\in \Tprox{\gf}{\gamma}{p}(\widehat{x})$, i.e., 
   $\widehat{x}\in \bs{{\rm Fix}}(\Tprox{\gf}{\gamma}{p})$.
\end{proof}

Next, we establish the convergence of the function values and their behavior at the cluster points.

\begin{theorem}[Convergence of function values]\label{th:comconfunval}
Let $\{x^k\}_{k\in \Nz}$ be generated by Algorithm~\ref{alg:inexact}, and let $\{\ov{y}^{k}\}_{k\in \Nz}$  satisfy \eqref{eq:defybar}. 
Then, there exists $\fv\in \R$ such that
 \begin{equation}\label{eq:th:comcon:limser}
 \mathop{\bs\lim}\limits_{k\to\infty}\fgam{\gf}{p}{\gamma}(x^k)= \mathop{\bs\lim}\limits_{k\to\infty} \gf(\ov{y}^k) =  \mathop{\bs\lim}\limits_{k\to\infty}\fgamepsk{\gf}{p}{\gamma}(x^{k})
  =\fv.
 \end{equation}
Moreover, if $\widehat{x}\in \R^n$ is a cluster point of $\{x^k\}_{k\in \Nz}$, then
 $\gf(\widehat{x})=\fgam{\gf}{p}{\gamma}(\widehat{x})=\fv$, that is,
$\gf$ and $\fgam{\gf}{p}{\gamma}$ are constant on the set of cluster points.
\end{theorem}
\begin{proof}
Define the sequence $\{\fv_k\}_{k\in \Nz}$ as  $\fv_k:=\fgamepsk{\gf}{p}{\gamma}(x^k)+\sum_{j=k}^{\infty}\varepsilon_{j}+\sum_{j=k+1}^{\infty}\varepsilon_{j}$.
Adding $\sum_{j=k+1}^{\infty}\varepsilon_{j}+\sum_{j=k+2}^{\infty}\varepsilon_{j}$ to both sides of inequality
\eqref{eq:alg:ingrad:upd},
we obtain $\fv_{k+1}\leq \fv_k$. Thus, $\{\fv_k\}_{k\in \Nz}$ is non-increasing and bounded from below by $\gf^*$,  implying that $\{\fv_k\}_{k\in \Nz}$  converges to some $\fv\in \R$. This implies
\begin{equation*}
 \mathop{\bs\lim}\limits_{k\to \infty}\fgamepsk{\gf}{p}{\gamma}(x^k)=\mathop{\bs\lim}\limits_{k\to \infty} \left(\fgamepsk{\gf}{p}{\gamma}(x^k)+\sum_{j=k}^{\infty}\varepsilon_{j}+\sum_{j=k+1}^{\infty}\varepsilon_{j}\right)=
\mathop{\bs\lim}\limits_{k\to \infty}\fv_k=\fv.
\end{equation*}
Applying \eqref{eq:ep-approx:fun} yields
\begin{align*}
\gf(\ov{y}^k)+\frac{1}{p\gamma}\Vert x^k - \ov{y}^k\Vert^p&=
f(\ov{y}^k)+g(\ov{y}^k)+\frac{1}{p\gamma}\Vert x^k - \ov{y}^k\Vert^p
\\&\leq  f(x^k)+\langle \nabla f(x^k), \ov{y}^k-x^k\rangle+g(\ov{y}^k)+\frac{L_p}{p}\Vert x^k-\ov{y}^k\Vert^p+\frac{1}{p\gamma}\Vert x^k - \ov{y}^k\Vert^p
\\
&=\fgam{\gf}{p}{\gamma}(x^k)+\frac{L_p}{p}\Vert x^k-\ov{y}^k\Vert^p
\leq \fgam{\gf}{p,\varepsilon_k}{\gamma}(x^k)+\frac{L_p}{p}\Vert x^k-\ov{y}^k\Vert^p
\leq \fgam{\gf}{p}{\gamma}(x^k)+\frac{L_p}{p}\Vert x^k-\ov{y}^k\Vert^p+\varepsilon_k
\\&= f(x^k)+\langle\nabla f(x^k), \ov{y}^k -x^k\rangle+g(\ov{y}^k) +\frac{L_p}{p}\Vert x^k-\ov{y}^k\Vert^p+\varepsilon_k
\\&\leq f(\ov{y}^k)+\langle\nabla f(\ov{y}^k), x^k-\ov{y}^k\rangle+\langle\nabla f(x^k), \ov{y}^k -x^k\rangle+g(\ov{y}^k) +\frac{2L_p}{p}\Vert x^k-\ov{y}^k\Vert^p+\varepsilon_k
\\&=\gf(\ov{y}^k)+\langle\nabla f(\ov{y}^k)-\nabla f(x^k), x^k-\ov{y}^k\rangle+\frac{2L_p}{p}\Vert x^k-\ov{y}^k\Vert^p+\varepsilon_k.
\end{align*}
As a result of $\Vert x^k - \ov{x}^k\Vert\to 0$ and $\Vert \ov{y}^k - \ov{x}^k\Vert\to 0$, it holds that
$\Vert x^k - \ov{y}^k\Vert\to 0$. Hence, taking limits as $k \to \infty$, we obtain \eqref{eq:th:comcon:limser}. 
Let $\widehat{x}\in \R^n$ be a cluster point, with corresponding subsequences $\{x^j\}_{j\in J\subseteq \Nz}$ and $\{\ov{y}^j\}_{j\in J}$.
 Hence, 
\begin{equation*}
\gf(\widehat{x})\leq
\mathop{\bs\liminf}\limits_{\scriptsize{\begin{matrix}j\to\infty\\j\in J \end{matrix}}}\gf(\ov{y}^j)
\leq  \mathop{\bs\limsup}\limits_{\scriptsize{\begin{matrix}j\to\infty\\j\in J \end{matrix}}}\gf(\ov{y}^j)
\overset{(i)}{\leq}   \mathop{\bs\limsup}\limits_{\scriptsize{\begin{matrix}j\to\infty\\j\in J \end{matrix}}} \fgam{\gf}{p}{\gamma}(x^j)=\fgam{\gf}{p}{\gamma}(\widehat{x})\leq \gf(\widehat{x}),
\end{equation*}
where the inequality $(i)$ obtains from Theorem~\ref{th:basichifbe}~\ref{th:basichifbe:ineqforp}.
From \eqref{eq:th:comcon:limser}, it is concluded that
\begin{equation*}
\gf(\widehat{x})=\fgam{\gf}{p}{\gamma}(\widehat{x})=
 \mathop{\bs\lim}\limits_{\scriptsize{\begin{matrix}j\to\infty\\j\in J \end{matrix}}}\fgam{\gf}{p}{\gamma}(x^j)=  \mathop{\bs\lim}\limits_{k\to\infty}\fgam{\gf}{p}{\gamma}(x^k)=\fv.
\end{equation*}
Since $\widehat{x}$ is an arbitrary cluster point, this completes the proof.
\end{proof}

\subsection{{\bf Global and linear convergence}}
\label{subsec:gllincon}

Establishing the global convergence and, under suitable conditions, linear rates for a broad class of optimization algorithms relies heavily on the Kurdyka-\L{}ojasiewicz property (see, e.g., \cite{absil2005convergence,Ahookhosh21,attouch2010proximal,Attouch2013,bolte2007lojasiewicz,Bolte2007Clarke,Bolte2014,li2023convergence,Li18,Yu2022}).
We recall this property below.
\begin{definition}[Kurdyka-\L{}ojasiewicz property]\label{def:kldef}
Let $\gh: \R^n \to \Rinf$ be a proper lsc function. We say that $\gh$ satisfies the \textit{Kurdyka-\L{}ojasiewicz (KL) property} at $\ov{x}\in\Dom{\partial \gh}$ if there exist constants $r>0$, $\eta\in (0, +\infty]$, and a desingularizing function $\phi$ such that
\begin{equation}\label{eq:intro:KLabs}
\phi'\left(\vert\gh(x) - \gh(\ov{x})\vert\right)\dist\left(0, \partial \gh(x)\right)\geq 1,
\end{equation}
whenever $x\in \mb(\ov{x}; r)\cap\Dom{\partial \gh}$ with $0<\vert \gh(x)-\gh(\ov{x})\vert<\eta$. Here $\phi:[0, \eta)\to [0, +\infty)$ is concave and continuous with $\phi(0)=0$, continuously differentiable on $(0, \eta)$, and $\phi'>0$ on $(0, \eta)$.
We further say that $\gh$ satisfies the KL property at $\ov{x}$ with \textit{exponent} $\theta$ if $\phi(t) = c t^{1 - \theta}$ for some $\theta\in [0, 1)$ and a constant $c>0$.
\end{definition}

It is noteworthy that in our formulation of the KL inequality \eqref{eq:intro:KLabs}, the argument of the gradient of the desingularizing function $\phi$, involves the absolute value. 
This version applies for all $x\in \mb(\ov{x}; r)\cap\Dom{\partial \gh}$ such that $0<\vert \gh(x)-\gh(\ov{x})\vert<\eta$.
 In contrast, the classical KL inequality, as presented in the seminal work \cite{attouch2010proximal}, typically omits the absolute value, stating the property for 
 $x\in \mb(\ov{x}; r)\cap \Dom{\partial \gh}$ where  $\gh(\ov{x})<\gh(x)<\gh(\ov{x})+\eta$. The utility of the absolute value formulation, as adopted in \eqref{eq:intro:KLabs}, particularly for the convergence analysis of nonmonotone algorithms, has been highlighted and investigated in recent contributions such as \cite{li2023convergence,Kabgani24itsopt}.
A notion related to the KL property is that of quasi-additivity, which is defined below.
\begin{definition}[Quasi-additivity property]\label{def:quasiadd}
Let $\gh: \R^n \to \Rinf$ be a proper lsc function.
We say that $\gh$ satisfies the KL property at $\ov{x} \in \Dom{\partial \gh}$ with the {\it quasi-additivity property},
if for some constant $c_{\phi} > 0$, desingularizing function $\phi$ satisfies 
\begin{equation*}
\left[\phi'(x+y)\right]^{-1}\leq c_\phi \left[\left(\phi'(x)\right)^{-1}+\left(\phi'(y)\right)^{-1}\right], \qquad\forall x, y\in (0, \eta)~\text{with}~ x+y<\eta.
\end{equation*}
\end{definition}

\begin{remark}\label{rem:onquasiadd}
If $\gh$ satisfies the KL property at $\ov{x}$ with exponent $\theta$, then it satisfies the KL property with quasi-additivity with $c_{\phi} = 1$ \cite{li2023convergence}.
\end{remark}
Next, we recall a uniformized KL property, adapted to Definition~\ref{def:kldef}.
\begin{fact}[Uniformized KL property]\cite[Lemma~2.4]{Kabgani24itsopt}\label{lem:Uniformized KL property}
Let $\gh:\R^n\to \Rinf$ be a proper lsc function that is constant on a nonempty and compact set $C\subseteq \Dom{\partial \gh}$.
If $\gh$ satisfies the KL property with the quasi-additivity property at each point of $C$, then there exist $r>0$, $\eta>0$, and a desingularizing function $\phi$, satisfying quasi-additivity, such that, for all $\ov{x}\in C$ and all $x\in X$, where
\begin{equation*}
X:=\Dom{\partial \gh}\cap\left\{x\in \R^n \mid \dist(x, C)<r\right\}\cap\left\{x\in \R^n \mid 0<\vert \gh(x) - \gh(\ov{x})\vert<\eta\right\},
\end{equation*}
we have
$
\phi'\left(\vert \gh(x) - \gh(\ov{x})\vert\right)\dist\left(0, \partial \gh(x)\right)\geq 1.
$
\end{fact}

We introduce assumptions concerning the KL property and error conditions, which are assumed to hold throughout the remainder of this section.
\begin{assumption}[Assumptions for global convergence]\label{assum:eps} 
\begin{enumerate}[label=(\textbf{\alph*}), font=\normalfont\bfseries, leftmargin=0.7cm]
 \item \label{assum:eps:a}  The sequence $\{x^k\}_{k\in \Nz}$ generated by Algorithm~\ref{alg:inexact} is bounded, and $\Omega(x^k)$ denotes the set of its cluster points.
  \item \label{assum:eps:a2} $f\in \mathcal{C}^2(U)$ where $U$ is a neighborhood of $\Omega(x^k)$.  
  
\item \label{assum:eps:b}  The function $\fgam{\gf}{p}{\gamma}$ satisfies the KL property on $\Omega(x^k)\subseteq \Dom{\partial \fgam{\gf}{p}{\gamma}}$
with a desingularizing function $\phi$   that satisfies the quasi-additivity property described in
Fact~\ref{lem:Uniformized KL property}.

\item \label{assum:eps:b2} $\{\delta_k>0\}_{k\in \Nz}$ is a non-increasing sequence such that $\delta_k\downarrow 0$, and there exists $\widehat{k}\in \Nz$ such that for each $k\geq \widehat{k}$,
$\dist\left(\Tprox{\gf}{\gamma}{p, \varepsilon_k}(x^k), \Tprox{\gf}{\gamma}{p}(x^k)\right)< \delta_k$.

\item \label{assum:eps:c} The series $ \sum_{k=0}^{\infty} \left([\phi'(w_k)]^{-1}\right)^{\frac{1}{p-1}}<\infty$, where
$w_k:=\frac{2(1-\gamma L_p)}{p\gamma}\sum_{j=k}^{\infty}\delta_j^p+3\sum_{j=k}^{\infty}\varepsilon_j$ for $k\in \Nz$.
\end{enumerate}
\end{assumption}

\begin{remark}\label{rem:onassum:eps}
\begin{enumerate}[label=(\textbf{\alph*}), font=\normalfont\bfseries, leftmargin=0.7cm]
\item \label{rem:onassum:eps:a}
By \eqref{eq:ep-approx:fun} and \eqref{eq:alg:ingrad:upd}, for each $k\in \Nz$, we have
\begin{align*}
\fgam{\gf}{p}{\gamma}(x^{k+1})&\leq\fgam{\gf}{p}{\gamma}(x^k)
+2\varepsilon_k+\varepsilon_{k+1}
\leq \fgam{\gf}{p}{\gamma}(x^0)+2\varepsilon_0+3\left(\varepsilon_1+\ldots+\varepsilon_k\right)+\varepsilon_{k+1}
\leq \fgam{\gf}{p}{\gamma}(x^0)+3\overline{\varepsilon}.
\end{align*}
Thus, 
 $\{x^k\}_{k\in \Nz}\subseteq \mathcal{L}(\fgam{\gf}{p}{\gamma}(x), \widehat{\lambda})$, with
 $\widehat{\lambda}:=\fgam{\gf}{p}{\gamma}(x^0)+3\overline{\varepsilon}$.
By Lemma~\ref{lem:coer}, if $\gh$ is coercive and $\gamma\in  (0, L_p^{-1})$, then $\fgam{\gf}{p}{\gamma}$ is also coercive. Consequently,
by Remark~\ref{rem:lem:coer}, the sublevel set $\mathcal{L}(\fgam{\gf}{p}{\gamma}(x), \widehat{\lambda})$ is bounded, providing a sufficient condition for the boundedness required in Assumption~\ref{assum:eps}~\ref{assum:eps:a}.

\item \label{rem:onassum:eps:b} From Assumptions~\ref{assum:eps}~\ref{assum:eps:a}~and~\ref{assum:eps:a2}, there exist a compact set $C\subseteq U$ and $r>0$ such that $C_0:=\{x\in \R^n\mid \dist(x, \Omega(x^k))<r\}\subseteq C$.
Hence, $\Vert \nabla^2 f(x)\Vert\leq L:=\bs{\max}_{y\in C}\Vert\nabla^2 f(y)\Vert<\infty$, for all $x\in C_0$. We use this property in the proof of Theorem~\ref{th:conKL}.
\end{enumerate}
\end{remark}

Here, we address the global convergence of the algorithm under the KL property.
\begin{theorem}[Global convergence]\label{th:conKL}
Suppose Assumption~\ref{assum:eps} holds, let $p\in (1,2]$, and let the sequences
$\{x^k\}_{k\in \Nz}$ and  $\{\ov{x}^{k}\}_{k\in \Nz}$  be generated by Algorithm~\ref{alg:inexact}, and let $\{\ov{y}^{k}\}_{k\in \Nz}$  satisfy \eqref{eq:defybar}. 
If there exists a constant $D\geq 0$ such that $\Vert d^k\Vert\leq D \Vert R_{\gamma,p}^{\varepsilon_k}(x^k)\Vert$ for all $k$,
then these sequences converge to a proximal fixed point.
\end{theorem}

\begin{proof}
Following \eqref{eq:ep-approx:fun} and \eqref{eq:alg:ingrad:upd}, for each $k\in \Nz$, it holds that
\begin{equation}\label{eq:th:conKL:anew1}
\fgam{\gf}{p}{\gamma}(x^{k+1})\leq \fgam{\gf}{p}{\gamma}(x^k) -\sigma \Vert x^k - \ov{x}^k\Vert^p +2\varepsilon_{k}+\varepsilon_{k+1}.
\end{equation}
From Assumption~\ref{assum:eps}~\ref{assum:eps:b2}, there exists $\widehat{k}\in \Nz$ such that for each $k\geq \widehat{k}$, 
$\Vert  \ov{x}^k - \ov{y}^k\Vert\leq \delta_k$. Thus, for each $k\geq \widehat{k}$,
\begin{equation}\label{eq:th:conKL:anew2}
\Vert x^k - \ov{y}^k\Vert^p \leq \left(\Vert x^k - \ov{x}^k\Vert +\Vert  \ov{x}^k - \ov{y}^k\Vert\right)^p\leq
2^{p-1}\left(\Vert x^k - \ov{x}^k\Vert^p +\delta_k^p\right).
\end{equation}
Furthermore, the second property of the structural iteration $\st$ leads to
\begin{align*}
\Vert x^{k+1}-x^k\Vert^p&= \Vert \st(\alpha_k,d^k)-x^k\Vert^p
\leq \left(\Vert x^k - \ov{x}^k\Vert +\Vert d^k\Vert\right)^p
\\& \leq (1+D)^p\Vert  x^k - \ov{x}^k\Vert^p\leq (1+D)^p\left(\Vert  x^k - \ov{y}^k\Vert+\Vert \ov{y}^k - \ov{x}^k\Vert\right)^p
\leq 2^{p-1} (1+D)^p\left(\Vert  x^k - \ov{y}^k\Vert^p+\delta_k^p\right).
\end{align*}
Combining \eqref{eq:th:conKL:anew1} and \eqref{eq:th:conKL:anew2}, for each $k \geq \widehat{k}$, we arrive at
\begin{align}\label{eq:th:conKL:j1}
&\fgam{\gf}{p}{\gamma}(x^{k+1})\leq \fgam{\gf}{p}{\gamma}(x^k)-\sigma 2^{-p} \Vert x^k - \ov{y}^k\Vert^p  -\sigma 2^{-p} \Vert x^k - \ov{y}^k\Vert^p+\sigma \delta_k^p+2\varepsilon_{k}+\varepsilon_{k+1}
\nonumber\\&\leq \fgam{\gf}{p}{\gamma}(x^k)-\sigma 2^{-p} \Vert x^k - \ov{y}^k\Vert^p  -\sigma 2^{-p} 
\frac{2^{1-p}}{(1+D)^p}\Vert x^{k+1}-x^k\Vert^p
+\sigma 2^{-p} \delta_k^p
+\sigma \delta_k^p+2\varepsilon_{k}+\varepsilon_{k+1}
\nonumber\\&=\fgam{\gf}{p}{\gamma}(x^k) -\sigma 2^{-p} \Vert x^k - \ov{y}^k\Vert^p  - 
\frac{ 2^{1-2p}\sigma}{(1+D)^p}\Vert x^{k+1}-x^k\Vert^p
+ (1+2^{-p})\sigma \delta_k^p
+2\varepsilon_{k}+\varepsilon_{k+1}.
\end{align}
Define the sequence $\{\fv_k\}_{k\in \Nz}$ as $\fv_k:=\fgamepsk{\gf}{p}{\gamma}(x^k)+v_k$, where
$v_k:=(1+2^{-p})\sigma\sum_{j=k}^{\infty}\delta_j^p+2\sum_{j=k}^{\infty}\varepsilon_j+\sum_{j=k+1}^{\infty}\varepsilon_j$.
Adding $v_{k+1}$ to both sides of the above inequality, we obtain $\fv_{k+1}\leq \fv_k$ for each $k\geq \widehat{k}$. Thus, the sequence $\{\fv_k\}_{k\geq \widehat{k}}$ is non-increasing and bounded from below by $\gf^*$,  implying that $\fv_k$  converges to some $\fv\in \R$.
Given the boundedness of $\{x^k\}_{k\in \Nz}$,  the set $\Omega(x^k)$ is nonempty and compact. For all $\widehat{x}\in \Omega(x^k)$,
Theorem~\ref{th:comconfunval} implies
\begin{equation}\label{eq:th:conKL:zz1-1}
 \mathop{\bs\lim}\limits_{k\to \infty}\fgam{\gf}{p}{\gamma}(x^k)= \mathop{\bs\lim}\limits_{k\to \infty}\left(\fgam{\gf}{p}{\gamma}(x^k)+v_k\right)=\fgam{\gf}{p}{\gamma}(\widehat{x})=\fv.
\end{equation}
Following Fact~\ref{lem:Uniformized KL property}, there exist $r>0$, $\eta>0$, and a desingularizing function $\phi$ such that for all $x\in X$, with
\begin{equation*}
X:=\Dom{\partial \fgam{\gf}{p}{\gamma}}\cap \{x\in \R^n \mid \dist(x, \Omega(x^k))<r\}\cap\{x\in \R^n \mid 0<\vert \fgam{\gf}{p}{\gamma}(x) - \fgam{\gf}{p}{\gamma}(\widehat{x})\vert<\eta\},
\end{equation*}
we have
\begin{equation}\label{eq:th:globconv:klI:a}
\phi'\left(\vert \fgam{\gf}{p}{\gamma}(x) - \fgam{\gf}{p}{\gamma}(\widehat{x})\vert\right)\dist\left(0, \partial \fgam{\gf}{p}{\gamma}(x)\right)\geq 1.
\end{equation}
We have $v_k\downarrow 0$ and since $\Vert x^{k+1}- x^k\Vert\to 0$, by increasing $\widehat{k}$ if necessary, for each $k\geq \widehat{k}$, $v_k<\frac{\eta}{2}$,
  $ \dist(x^k, \Omega(x^k))<r$, and \eqref{eq:th:conKL:zz1-1} implies $\vert\fgam{\gf}{p}{\gamma}(x^k)-\fgam{\gf}{p}{\gamma}(\widehat{x})\vert<\frac{\eta}{2}$.
Define $\Delta_k:=\phi\left(\fgam{\gf}{p}{\gamma}(x^k)-\fv+v_k\right)$. The concavity of $\phi$, the monotonic decrease of $\phi'$, and \eqref{eq:th:conKL:j1} yield
\begin{align}\label{eq:th:conKL:z1}
  \Delta_k-\Delta_{k+1} & \geq \phi'\left(\fgam{\gf}{p}{\gamma}(x^k)-\fv+v_k\right)\left[\fgam{\gf}{p}{\gamma}(x^k)+v_k-\fgam{\gf}{p}{\gamma}(x^{k+1})-v_{k+1}\right]
\nonumber  \\&\geq \phi'\left(\vert\fgam{\gf}{p}{\gamma}(x^k)-\fv\vert+v_k\right)\left[\sigma 2^{-p} \Vert x^k - \ov{y}^k\Vert^p  + 
\frac{ 2^{1-2p}\sigma}{(1+D)^p}\Vert x^{k+1}-x^k\Vert^p\right].
\end{align}
If $\fgam{\gf}{p}{\gamma}(x^k)-\fv=0$, the inequality \eqref{eq:th:conKL:z1} ensures
\begin{align}\label{eq:th:conKL:z2}
\sigma 2^{-p} \Vert x^k - \ov{y}^k\Vert^p  &+ \frac{ 2^{1-2p}\sigma}{(1+D)^p}\Vert x^{k+1}-x^k\Vert^p\leq 
\left(\Delta_k-\Delta_{k+1} \right)\left([\phi'(v_k)]^{-1}\right).
\end{align}
If $\fgam{\gf}{p}{\gamma}(x^k)-\fv\neq 0$, by increasing $\widehat{k}$ if necessary, the quasi-additivity ensures that there exists some $c_\phi>0$ such that
\begin{equation*}
  \Delta_k-\Delta_{k+1}\geq \frac{(c_\phi)^{-1}}{[\phi'\left(\vert\fgam{\gf}{p}{\gamma}(x^k)-\fv\vert\right)]^{-1}+[\phi'(v_k)]^{-1}}\left[ \sigma 2^{-p} \Vert x^k - \ov{y}^k\Vert^p  + 
\frac{2^{1-2p}\sigma}{(1+D)^p}\Vert x^{k+1}-x^k\Vert^p\right].
\end{equation*}
From \eqref{eq:th:globconv:klI:a}, Assumption~\ref{assum:eps}~\ref{assum:eps:a2}, and $\Vert x^k-\ov{y}^k\Vert\to 0$, there exists some $L>0$ (see Remark~\ref{rem:onassum:eps}~\ref{rem:onassum:eps:b}) such that
\begin{align}\label{eq:th:conKL:z3}
\sigma 2^{-p} \Vert x^k - \ov{y}^k\Vert^p  &+ \frac{ 2^{1-2p}\sigma}{(1+D)^p}\Vert x^{k+1}-x^k\Vert^p
\nonumber\\&\leq c_\phi\left(\Delta_k-\Delta_{k+1} \right)\left([\phi'\left(\vert\fgam{\gf}{p}{\gamma}(x^k)-\fv\vert\right)]^{-1}+[\phi'(v_k)]^{-1}\right)
\nonumber\\&\leq c_\phi\left(\Delta_k-\Delta_{k+1} \right)\left(\dist\left(0, \partial \fgam{\gf}{p}{\gamma}(x^k)\right)+[\phi'(v_k)]^{-1}\right)
\nonumber\\&\overset{(i)}{\leq}c_\phi\left(\Delta_k-\Delta_{k+1} \right)\left(\Vert \nabla^2f(x^k)(\ov{y}^k-x^k)\Vert+\frac{1}{\gamma}\Vert x^k-\ov{y}^k\Vert^{p-1}+[\phi'(v_k)]^{-1}\right)
\nonumber\\&\leq c_\phi\left(\Delta_k-\Delta_{k+1} \right)\left(L\Vert  x^k-\ov{y}^k\Vert+\frac{1}{\gamma}\Vert x^k-\ov{y}^k\Vert^{p-1}+[\phi'(v_k)]^{-1}\right),
\end{align}
where the inequality $(i)$ follows from Lemma~\ref{lem:frech:hifbe}. For large $k$, since $\Vert  x^k-\ov{y}^k\Vert<1$ and $p\in (1, 2]$, we have
  $\Vert  x^k-\ov{y}^k\Vert\leq \Vert x^k-\ov{y}^k\Vert^{p-1}$. Thus, from \eqref{eq:th:conKL:z2} and \eqref{eq:th:conKL:z3}, 
  setting $\widehat{m}:=\bs\max\{1, c_\phi\}$,  for each $k\geq \widehat{k}$, we obtain
\begin{align}\label{eq:th:conKL:z4}
  \Vert x^k - \ov{y}^k\Vert^p  + \frac{ 2^{1-p}}{(1+D)^p}\Vert x^{k+1}-x^k\Vert^p
\leq \widehat{m}\sigma^{-1}2^{p}\left(\Delta_k-\Delta_{k+1} \right)\left(\left(L+\frac{1}{\gamma}\right)\Vert x^k-\ov{y}^k\Vert^{p-1}+[\phi'(v_k)]^{-1}\right).
\end{align}
The inequality
 \begin{align*}
\left(\Vert x^k- \ov{y}^k\Vert+ \frac{ 2^{\frac{1-p}{p}}}{1+D}\Vert x^{k+1}-x^k\Vert\right)^p
 \leq 2^{p-1}\left(\Vert x^k - \ov{y}^k\Vert^p  + \frac{ 2^{1-p}}{(1+D)^p}\Vert x^{k+1}-x^k\Vert^p\right)
 \end{align*}
implies that
\begin{align}\label{eq:Refeq}
\Vert x^k - \ov{y}^k\Vert + \frac{ 2^{\frac{1-p}{p}}}{1+D}\Vert x^{k+1}-x^k\Vert
\leq\left[\frac{2^{2p-1}\widehat{m}(L\gamma+1)}{\sigma\gamma}\left(\Delta_k-\Delta_{k+1} \right)\left(\Vert x^k-\ov{y}^k\Vert^{p-1}+\frac{\gamma}{L\gamma+1} \left([\phi'(v_k)]^{-1}\right)\right)\right]^{\frac{1}{p}}.
\end{align}
For $p\in (1, 2]$, it is clear that $\frac{1}{p-1}\geq1$. Together with \eqref{eq:intrp:p02}, \eqref{eq:Refeq}, this ensures
\begin{align*}
\Vert x^k - \ov{y}^k\Vert&+ \frac{ 2^{\frac{1-p}{p}}}{1+D}\Vert x^{k+1}-x^k\Vert\leq\left[\frac{2^{2p-1}\widehat{m}(L\gamma+1)}{\sigma\gamma}\left(\Delta_k-\Delta_{k+1} \right)\left(\Vert x^k-\ov{y}^k\Vert^{p-1}+
\frac{\gamma}{L\gamma+1} \left([\phi'(v_k)]^{-1}\right)
\right)\right]^{\frac{1}{p}}
\\&=\left[\frac{2^{p+1}\widehat{m}(L\gamma+1)}{\sigma\gamma}\left(\Delta_k-\Delta_{k+1} \right)2^{p-2}\left(\Vert x^k-\ov{y}^k\Vert^{p-1}+
\frac{\gamma}{L\gamma+1} \left([\phi'(v_k)]^{-1}\right)\right)\right]^{\frac{1}{p}}
\\&\leq\frac{2^{p+1}\widehat{m}(L\gamma+1)}{\sigma\gamma}\left(\Delta_k-\Delta_{k+1} \right)+2^{\frac{p-2}{p-1}}\left(\Vert x^k-\ov{y}^k\Vert^{p-1}+
\frac{\gamma}{L\gamma+1} \left([\phi'(v_k)]^{-1}\right)\right)^{\frac{1}{p-1}}
\\&\leq\frac{2^{p+1}\widehat{m}(L\gamma+1)}{\sigma\gamma}\left(\Delta_k-\Delta_{k+1} \right)+\Vert x^k-\ov{y}^k\Vert+\left(\frac{\gamma}{L\gamma+1} \left([\phi'(v_k)]^{-1}\right)\right)^\frac{1}{p-1}.
\end{align*}
Hence, for each $k \geq \widehat{k}$ and $p \in (1, 2]$, we obtain
\begin{equation}\label{eq:th:conKL:e}
\Vert x^{k+1}-x^k\Vert\leq \varpi(p)\left(\Delta_k-\Delta_{k+1} \right)+ \frac{1+D}{2^{\frac{1-p}{p}}}\left(\frac{\gamma}{L\gamma+1} \left([\phi'(v_k)]^{-1}\right)\right)^\frac{1}{p-1},
\end{equation}
where $\varpi(p):=\frac{2^{\frac{p^2+2p-1}{p}}(1+D)\widehat{m}(L\gamma+1)}{\sigma\gamma}$.
Since $v_k\leq w_k$, where $w_k$ is defined in Assumption~\ref{assum:eps}~$\ref{assum:eps:c}$ and given the monotonic decrease of 
$\phi'$ due to the concavity of $\phi$, we have $\phi'(w_k)\leq \phi'(v_k)$, i.e., $[\phi'(v_k)]^{-1}\leq [\phi'(w_k)]^{-1}$. Thus, Assumption~\ref{assum:eps}~$\ref{assum:eps:c}$ implies 
$\sum_{k=0}^{\infty}([\phi'(v_k)]^{-1})^\frac{1}{p-1}<\infty$.  From \eqref{eq:th:conKL:e}, we obtain
\begin{align*}
\sum_{k=\widehat{k}}^{\infty}\Vert x^{k+1}-x^k\Vert&\leq \varpi(p)\sum_{k=\widehat{k}}^{\infty}\left(\Delta_k-\Delta_{k+1} \right)+ \frac{1+D}{2^{\frac{1-p}{p}}}\left(\frac{\gamma}{L\gamma+1}\right)^\frac{1}{p-1} \sum_{k=\widehat{k}}^{\infty}\left([\phi'(v_k)]^{-1}\right)^\frac{1}{p-1}
\\& = \varpi(p)\Delta_{\widehat{k}}+ \frac{1+D}{2^{\frac{1-p}{p}}}\left(\frac{\gamma}{L\gamma+1}\right)^\frac{1}{p-1} \sum_{k=\widehat{k}}^{\infty} \left([\phi'(v_k)]^{-1}\right)^\frac{1}{p-1}<\infty,
\end{align*}
Hence, $\{x^k\}_{k\in \Nz}$ is a Cauchy sequence and converges to some $\widehat{x}$. This implies $\ov{y}^{k}\to \widehat{x}$ and from Theorem~\ref{th:hopbfb}~$\ref{th:hopbfb:proxb:conv}$, we conclude that $\widehat{x}\in \Tprox{\gf}{\gamma}{p}(\widehat{x})$.
\end{proof}

By imposing tighter control on the error in computing inexact elements of $\Tprox{\gf}{\gamma}{p}$ and leveraging the KL property with an exponent related to the power $p$, we can establish the linear convergence of the sequence generated by Algorithm~\ref{alg:inexact}. The following assumption addresses the error control.

\begin{assumption}[Assumptions for linear convergence]\label{assum:linconv}
Let $\omega\in (0, 1)$ and $\{\beta_k\}_{k\in \Nz}$ be a sequence of positive scalars that is non-increasing and satisfies $\ov{\beta}:=\sum_{k=0}^{\infty} \beta_k<\infty$. The parameters $\delta_k$ and $\varepsilon_k$ are chosen such that
\begin{equation}\label{eq:ep-approx:oper}
\varepsilon_k^{\frac{1}{p}}, \delta_k\leq \bs\min\left\{\beta_k,\omega \Vert x^k -\ov{x}^k\Vert, \mathop{\bs{\min}}\left\{\beta_j \Vert x^i - \ov{x}^i\Vert\mid 0\leq i\leq j\leq k\right\} \right\},
\end{equation}
where $\ov{x}^{i}:=\Tprox{\gf}{\gamma}{p, \varepsilon_i}(x^i)$ and $\{\delta_k\}_{k\in \Nz}$ is specified in Assumption~\eqref{assum:eps}.
\end{assumption}
 The following remark reveals the idea behind Assumption~\ref{assum:linconv}.
  \begin{remark}\label{rem:ass:linconv}
\begin{enumerate}[label=(\textbf{\alph*}), font=\normalfont\bfseries, leftmargin=0.7cm]
\item \label{rem:ass:linconv:a}
 Since $\varepsilon_k^{\frac{1}{p}}, \delta_k\leq \beta_k$  for each $k\in \Nz$, the sequences $\left\{\varepsilon_k^{1/p}\right\}_{k\in \Nz}$ and $\{\delta_k\}_{k\in \Nz}$ are summable. 
 If $\phi(t) = c t^{\frac{1}{p}}$ for some $c>0$, i.e., $\theta=\frac{p-1}{p}$, as assumed in Theorem~\ref{th:linconv}, then by setting $\widehat{c}:=\frac{2(1-\gamma L_p)}{p\gamma}$
 and $w_k:=\widehat{c}\sum_{j=k}^{\infty}\delta_j^p+3\sum_{j=k}^{\infty}\varepsilon_j$ for $k\in \Nz$, we have
\begin{align*}
\sum_{k=0}^{\infty} \left([\phi'(w_k)]^{-1}\right)^{\frac{1}{p-1}}
&= \sum_{k=0}^{\infty} \left(\frac{1}{c(1-\theta)}w_k^\theta\right)^{\frac{1}{p-1}}
=\sum_{k=0}^{\infty} \left(\frac{p}{c}w_k^{\frac{p-1}{p}}\right)^{\frac{1}{p-1}}
=\sum_{k=0}^{\infty} \left(\frac{p}{c}\right)^{\frac{1}{p-1}}\left(\widehat{c}\sum_{j=k}^{\infty}\delta_j^p+3\sum_{j=k}^{\infty}\varepsilon_j\right)^{\frac{1}{p}}
\\
&\leq \sum_{k=0}^{\infty} \left(\frac{p}{c}\right)^{\frac{1}{p-1}}\left(\widehat{c}^{\frac{1}{p}}
\sum_{j=k}^{\infty}\delta_j+3^{\frac{1}{p}}\sum_{j=k}^{\infty}\varepsilon_j^{\frac{1}{p}}\right)<\infty.
\end{align*}

 This error control not only facilitates the computations for achieving the linear convergence but also ensures that
Assumption~\ref{assum:eps}~\ref{assum:eps:c} is satisfied.

\item \label{rem:ass:linconv:b} 
The inclusion of the term $\omega \Vert x^k -\ov{x}^k\Vert$ in \eqref{eq:ep-approx:oper} is inspired by Assumption (A) on page 880 of
\cite{rockafellar1976monotone2}. As shown in the proof of Theorem~\ref{th:conKL},
\eqref{eq:ep-approx:dist} and Assumption~\ref{assum:eps}~\ref{assum:eps:c} imply the existence of $\widehat{k}\in \Nz$ such that for each $k\geq \widehat{k}$, we have
$\Vert  \ov{x}^k - \ov{y}^k\Vert\leq \delta_k$. Thus, from \eqref{eq:ep-approx:oper}, we obtain
\[
\Vert  \ov{x}^k - \ov{y}^k\Vert\leq \delta_k\omega \Vert x^k -\ov{x}^k\Vert.
\]
This structure aligns with \cite[Assumption~(A)]{rockafellar1976monotone2} when considering $x^{k+1}=\ov{x}^k$.
\end{enumerate}
 \end{remark}
\begin{theorem}[Linear convergence]\label{th:linconv}
Suppose Assumptions~\ref{assum:eps}~and~\ref{assum:linconv} are satisfied, and the sequences
 $\{x^k\}_{k\in \Nz}$ and  $\{\ov{x}^{k}\}_{k\in \Nz}$ are produced by Algorithm~\ref{alg:inexact}, and $\{\ov{y}^{k}\}_{k\in \Nz}$, fulfilling \eqref{eq:defybar}. 
 If there exists a constant $D\geq 0$ such that $\Vert d^k\Vert\leq D \Vert R_{\gamma,p}^{\varepsilon_k}(x^k)\Vert$ for all $k$ and
$\fgam{\gf}{p}{\gamma}$ satisfies the KL inequality with an exponent $\theta=\frac{p-1}{p}$ on $\Omega(x^k)$ and $p\in (1, 2]$, then these sequences converge linearly to a proximal fixed point.
\end{theorem}
\begin{proof}
By Theorem~\ref{th:conKL}, the sequences $\{x^k\}_{k\in \Nz}$, $\{\ov{x}^{k}\}_{k\in \Nz}$, and $\{\ov{y}^{k}\}_{k\in \Nz}$ converge to a 
proximal fixed point $\widehat{x}$.
From Assumption~\ref{assum:eps}~\ref{assum:eps:b2}, there exists $\widehat{k} \in \Nz$ such that for all $k \geq \widehat{k}$, we have
$\Vert  \ov{x}^k - \ov{y}^k\Vert\leq \delta_k$.
To analyze the convergence rate, define the auxiliary sequence $S_k := \sum_{i=k}^{\infty} \Vert x^i - \ov{x}^i\Vert$ for $k \geq \widehat{k}$. Our goal is to establish that $S_k$ converges linearly.
In fact, using the structural iteration $\st$ and its second property, we bound the distances to $\widehat{x}$ as
\begin{align*}
&\Vert x^k - \widehat{x}\Vert\leq \sum_{i\geq k} \Vert x^{i+1} - x^i\Vert \leq (1+D)S_k,\\
& \Vert \ov{x}^k - \widehat{x}\Vert\leq \sum_{i\geq k} \Vert \ov{x}^{i+1} - \ov{x}^i\Vert
\leq \sum_{i\geq k}
 \left[\Vert \ov{x}^{i+1} - x^{i+1}\Vert+\Vert x^{i+1} - x^{i}\Vert+ \Vert \ov{x}^{i} - x^i \Vert \right]\leq (3+D)S_k,\\
& \Vert \ov{y}^k - \widehat{x}\Vert\leq \Vert \ov{y}^k - \ov{x}^k\Vert+\Vert \ov{x}^k - \widehat{x}\Vert\leq \delta_k +(3+D)S_k 
\\&\qquad\qquad\leq\ov{\beta} \Vert x^k - \ov{x}^k\Vert+(3+D)S_k\leq (\ov{\beta}+3+D)S_k.
\end{align*}
where the last inequality uses \eqref{eq:ep-approx:oper} to bound $\delta_k \leq \ov{\beta} \Vert x^k - \ov{x}^k \Vert$. Thus, linear convergence of $S_k$ ensures linear convergence of the sequences to $\widehat{x}$.
Let $\phi(t) = c t^{1 - \theta}$ with $c > 0$ and $\theta = \frac{p-1}{p}$ be the desingularization function associated with the KL property of $\fgam{\gf}{p}{\gamma}$. From \eqref{eq:th:globconv:klI:a}, Assumption~\ref{assum:eps}~\ref{assum:eps:a2}, and the fact that $\Vert x^k-\ov{y}^k\Vert\to 0$, there exists some $L > 0$ such that for sufficiently large $k$,
\begin{align*}
\vert \fgam{\gf}{p}{\gamma}(x^k) - \fv\vert &\leq \left[(1-\theta)c \left(\Vert \nabla^2f(x^k)(\ov{y}^k-x^k)\Vert+\frac{1}{\gamma}\Vert x^k-\ov{y}^k\Vert^{p-1}\right)\right]^{\frac{1}{\theta}}
\nonumber\\&\leq
\left[(1-\theta)c\left(L\Vert x^k-\ov{y}^k\Vert+\frac{1}{\gamma}\Vert x^k-\ov{y}^k\Vert^{p-1}\right)\right]^{\frac{1}{\theta}}.
\end{align*}
Since $p \in (1, 2]$ and $\Vert  x^k - \ov{y}^k \Vert  < 1$ for large $k$, we have $\Vert  x^k-\ov{y}^k\Vert\leq \Vert x^k-\ov{y}^k\Vert^{p-1}$, yielding
\begin{align}\label{eq:th:linconv:d1}
\vert \fgam{\gf}{p}{\gamma}(x^k) - \fv\vert &\leq 
\left[(1-\theta)c\left(\left(L+\frac{1}{\gamma}\right)\Vert x^k-\ov{y}^k\Vert^{p-1}\right)\right]^{\frac{1}{\theta}}.
\end{align}
Define $\Delta_k := \phi( \fgam{\gf}{p}{\gamma}(x^k) - \fv + v_k ) = c ( \fgam{\gf}{p}{\gamma}(x^k) - \fv + v_k )^{1 - \theta}$, as in Theorem~\ref{th:conKL}. Using \eqref{eq:intrp:p01} and \eqref{eq:th:linconv:d1}, we get
\begin{align*}
\Delta_k&\leq c (\vert\fgam{\gf}{p}{\gamma}(x^k)-\fv\vert+v_k)^{1 - \theta}
\leq c \vert\fgam{\gf}{p}{\gamma}(x^k)-\fv\vert^{1 - \theta}+c v_k^{1 - \theta}\\
&\leq c \left[\frac{(1-\theta)c(\gamma L+1)}{\gamma}\Vert x^k-\ov{y}^k\Vert^{p-1}\right]^{\frac{1-\theta}{\theta}}+c v_k^{\frac{1}{p}}\\
&=  c \left[\frac{(1-\theta)c(\gamma L+1)}{\gamma}\right]^{\frac{1}{p-1}}\Vert x^k-\ov{y}^k\Vert+c v_k^{\frac{1}{p}},
\end{align*}
since $\theta = \frac{p-1}{p}$ implies $\frac{1 - \theta}{\theta} = \frac{1}{p-1}$.
The error term is bounded as
\begin{align}\label{eq:th:linconv:d2:2}
v_k^{\frac{1}{p}}\leq\left[(1+2^{-p})\sigma\sum_{j=k}^{\infty}\delta_j^p\right]^{\frac{1}{p}}+\left[3\sum_{j=k}^{\infty}\varepsilon_j\right]^{\frac{1}{p}}
\leq\left[(1+2^{-p})\sigma\right]^{\frac{1}{p}}\sum_{j=k}^{\infty}\delta_j+3^{\frac{1}{p}}\sum_{j=k}^{\infty}\varepsilon_j^{\frac{1}{p}}.
\end{align}
It follows from \eqref{eq:ep-approx:oper} that
\begin{equation*}
\sum_{j=k}^{\infty}\delta_j\leq \Vert x^k - \ov{x}^k\Vert \sum_{j=k}^{\infty}\beta_j\leq \ov{\beta}\Vert x^k - \ov{x}^k\Vert,\qquad
\sum_{j=k}^{\infty}\varepsilon_j^{\frac{1}{p}}\leq \Vert x^k - \ov{x}^k\Vert \sum_{j=k}^{\infty}\beta_j\leq \ov{\beta}\Vert x^k - \ov{x}^k\Vert,
\end{equation*}
i.e.,
\begin{equation}\label{eq:th:linconv:k1}
v_k^{\frac{1}{p}}\leq \left(\left[(1+2^{-p})\sigma\right]^{\frac{1}{p}}+3^{\frac{1}{p}}\right)\ov{\beta}\Vert x^k - \ov{x}^k\Vert.
\end{equation}
Since $\Vert x^k - \ov{y}^k\Vert\leq  \Vert x^k - \ov{x}^k\Vert+ \delta_k\leq (1+\ov{\beta})\Vert x^k - \ov{x}^k\Vert$, we combine this with the above to obtain
\begin{align*}
\Delta_k&\leq  c \left[\frac{(1-\theta)c(\gamma L+1)}{\gamma}\right]^{\frac{1}{p-1}}\Vert x^k-\ov{y}^k\Vert+c\left(\left[(1+2^{-p})\sigma\right]^{\frac{1}{p}}+3^{\frac{1}{p}}\right)\ov{\beta}\Vert x^k - \ov{x}^k\Vert
\nonumber\\&\leq  c \left(1+\ov{\beta}\right)\left[\frac{(1-\theta)c(\gamma L+1)}{\gamma}\right]^{\frac{1}{p-1}}\Vert x^k-\ov{x}^k\Vert+c\left(\left[(1+2^{-p})\sigma\right]^{\frac{1}{p}}+3^{\frac{1}{p}}\right)\ov{\beta}\Vert x^k - \ov{x}^k\Vert
\nonumber\\&= c_1 \Vert x^k - \ov{x}^k\Vert,
\end{align*}
where $c_1:= c\left[ \left(1+\ov{\beta}\right)\left[\frac{(1-\theta)c(\gamma L+1)}{\gamma}\right]^{\frac{1}{p-1}}+\left(\left[(1+2^{-p})\sigma\right]^{\frac{1}{p}}+3^{\frac{1}{p}}\right)\ov{\beta}
\right]$.
From \eqref{eq:th:conKL:z4} with $\widehat{m} = 1$ (since $c_\phi = 1$ for the KL property with exponent $\theta$), we have
\begin{align*}
 \Vert x^k - \ov{y}^k\Vert^p \leq  \sigma^{-1}2^{p}\left(\Delta_k-\Delta_{k+1} \right)\left(\left(L+\frac{1}{\gamma}\right)\Vert x^k-\ov{y}^k\Vert^{p-1}+[\phi'(v_k)]^{-1}\right).
\end{align*}
It is concluded by
\begin{equation*}
\Vert x^k - \ov{x}^k\Vert\leq \Vert x^k - \ov{y}^k\Vert+\Vert \ov{y}^k - \ov{x}^k\Vert\leq \Vert x^k - \ov{y}^k\Vert+\delta_k\leq \Vert x^k - \ov{y}^k\Vert+\omega
\Vert x^k - \ov{x}^k\Vert, 
\end{equation*}
that $(1-\omega)\Vert x^k - \ov{x}^k\Vert\leq \Vert x^k - \ov{y}^k\Vert$ and
\begin{align}\label{eq:th:conKL:z6}
[\phi'(v_k)]^{-1}&=\frac{1}{c(1-\theta)}v_k^{\theta}=\frac{1}{c(1-\theta)}v_k^{\frac{p-1}{p}}
\nonumber\\&\leq \frac{\left(\left[(1+2^{-p})\sigma\right]^{\frac{1}{p}}+3^{\frac{1}{p}}\right)^{p-1}\ov{\beta}^{p-1}}{c(1-\theta)} \Vert x^k - \ov{x}^k\Vert^{p-1}
\leq c_2 \Vert x^k - \ov{y}^k\Vert^{p-1},
\end{align}
with $c_2:=\frac{\left(\left[(1+2^{-p})\sigma\right]^{\frac{1}{p}}+3^{\frac{1}{p}}\right)^{p-1}\ov{\beta}^{p-1}}{c(1-\theta)(1-\omega)^{p-1}}$.
From \eqref{eq:th:conKL:z6}, we obtain
\begin{align*}
  \Vert x^k - \ov{y}^k\Vert^p \leq \sigma^{-1}2^{p}\left(\Delta_k-\Delta_{k+1} \right)\left(\left(L+\frac{1}{\gamma}+c_2\right)\Vert x^k-\ov{y}^k\Vert^{p-1}\right),
\end{align*}
and thus,
\begin{align}\label{eq:th:conKL:z7}
 c_3 \Vert x^k - \ov{y}^k\Vert \leq  \Delta_k-\Delta_{k+1},
\end{align}
with $c_3=\sigma2^{-p}\left(L+\frac{1}{\gamma}+c_2\right)^{-1}$. 
Together with $(1-\omega)\Vert x^k - \ov{x}^k\Vert\leq \Vert x^k - \ov{y}^k\Vert$, this ensures
\begin{align*}
  S_k=\sum_{i\geq k}\Vert x^i - \ov{x}^i\Vert &\leq \frac{1}{1-\omega}\sum_{i\geq k} \Vert x^i - \ov{y}^i\Vert
\leq \frac{1}{c_3(1-\omega)}\sum_{i\geq k} \left(\Delta_i-\Delta_{i+1} \right)
\\
&\overset{(i)}{\leq} \frac{1}{c_3(1-\omega)}\Delta_k\leq c_4 \Vert x^k - \ov{x}^k\Vert\leq  c_4 \left(S_k-S_{k+1}\right),
\end{align*}
where $c_4=\frac{c_1}{c_3(1-\omega)}$ and $(i)$  follows from the fact that $\phi$ is continuous and $\fgam{\gf}{p}{\gamma}(x^k) - \fv + v_k\to 0$, which implies $\Delta_k \to 0$.
Thus, $S_{k+1}\leq \left(1-\frac{1}{c_4}\right)S_k$, demonstrating the asymptotic $Q$-linear convergence of $\{S_k\}_{k\in \Nz}$. Hence, $\{x^k\}_{k\in \Nz}$, $\{\ov{x}^{k}\}_{k\in \Nz}$, and $\{\ov{y}^{k}\}_{k\in \Nz}$ converge R-linearly to a proximal fixed point.
\end{proof}

The global and linear convergence results established in Theorems~\ref{th:conKL}~and~\ref{th:linconv} rely on the KL property of the function $\fgam{\gf}{p}{\gamma}$ over the set $\Omega(x^k)$. By Theorem~\ref{th:comcon}, we have $\Omega(x^k) \subseteq \bs{\rm Fix}(\Tprox{\gf}{\widehat{\gamma}}{p})$ with $\widehat{\gamma}>0$. We conclude this section by introducing a sufficient condition for the KL property with a specific exponent to hold at these proximal fixed points.

\begin{proposition}[KL property of HiFBE]\label{pro:KLp:hope}
Suppose Assumption~\ref{assum:mainassum} holds. If the function $\ell(x, y)$ satisfies the KL property with exponent 
$\theta\in \left[\frac{p-1}{p}, 1\right)$ at each $(x, x) \in  \Dom{\partial  \ell}$ with $x \in \bs{\rm Fcrit}(\gf)$, then for each
$\ov{x}\in \bs{\rm Fix}(\Tprox{\gf}{\widehat{\gamma}}{p})$ with $\widehat{\gamma}>0$,  the function $\fgam{\gf}{p}{\gamma}$ satisfies the KL property with exponent $\theta$ at $\ov{x}$ for all $\gamma \in (0, \bs{\min} \{ L_p^{-1}, \gamma^{g, p}, \widehat{\gamma} \})$.
\end{proposition}
\begin{proof}
 See Appendix~\ref{sec:app:proofs}.
\end{proof}



\section{Preliminary numerical experiments}
\label{sec:numerical}

In this section, we investigate the application of HiFBA and Boosted HiFBA to two classes of problems: regularized linear inverse problems (Subsection~\ref{subsec:invprob}) and regularized nonnegative matrix factorization (Subsection~\ref{subsec:RegularizedNMF}). We begin by outlining the implementation details, including the computation of proximal approximations and the framework for comparative analysis. Subsequently, we describe the problem settings and parameter configurations used in the experiments. Preliminary numerical results demonstrate that the proposed algorithms outperform several subgradient-based methods in regularized linear inverse problems, and the Bregman proximal gradient method (given by \eqref{eq:fbsGeneral} with $\bs\zeta(y,x^k)=D_h(y,x^k)$) in regularized nonnegative matrix factorization.

\vspace{-4mm}
\subsection{{\bf Implementation issues}}\label{sub:impiss}
We first provide an overview of the implementation details for HiFBA and Boosted HiFBA, alongside the subgradient-based methods and the Bregman forward-backward methods used for comparison. The key components of these implementations are summarized in the following:
\begin{description}
\item [{\bf (i)}] \textbf{(Approximation of HiFBS)} 
In the regularized linear inverse problem, to compute $\Tprox{\gf}{\gamma}{p}(x^k)$ for the current iterate $x^k$, we solve the optimization problem \eqref{HFBM} approximately using a subgradient method with geometrically decaying step-sizes (SG-GDSS), as detailed in Algorithm~\ref{alg:SG-DSS}. The algorithm terminates if the iteration count exceeds 25 or if the change in iterates satisfies $\Vert y^{k+1} - y^k \Vert < 10^{-3}$.

 \vspace{-2mm}
\begin{algorithm}[H]
\caption{Generic SubGradient algorithm}\label{alg:SG-DSS}
\begin{algorithmic}[1]
\State \textbf{Initialization} Start with $y^0=x^k$, $\beta\in (0, 1)$. Set $i=0$ and $\beta_0=\beta$.
\While{stopping criteria do not hold  }
\State Choose $\zeta^i\in \nabla f(x^k)+\partial g(y^i)+\frac{1}{\gamma}\Vert y^i -x^k\Vert^{p-2}(y^i - x^k)$;
\State Set $y^{i+1}=y^i -\beta_i \frac{\zeta^i}{\Vert \zeta^i\Vert}$;
\State Set $i=i+1$ and $\beta_i = \beta^{i+1}$;
\EndWhile
\end{algorithmic}
\end{algorithm}

\vspace{-3mm}
 In the regularized nonnegative matrix factorization, an element of $\Tprox{\gf}{\gamma}{p}(x^k)$ can be computed by solving a cubic equation, as detailed in Subsection~\ref{subsec:RegularizedNMF}.
 \vspace{2mm}

\item [{\bf (ii)}] \textbf{(Boosted HiFBA: Search direction, structural iteration, and accuracy)} 
For Boosted HiFBA, the search direction is defined as $d^k:=-\omega_k R_\gamma^{\varepsilon_k}(x^k)$ at each iteration \cite{laCruz2006spectral}. To determine $\omega_k$, we set $\omega_{\textbf{min}} = 10^{-1}$ and $\omega_{\textbf{max}} = 10^{10}$, with an initial value $\omega_0 = 1$. We compute
$\widehat{\omega}_k=\frac{\langle s^k, s^k\rangle}{\langle s^k, y^k\rangle}$, where
$s^k=x^k - x^{k-1}$ and $y^k=R_\gamma^{\varepsilon_k}(x^k) - R_\gamma^{\varepsilon_{k-1}}(x^{k-1})$. If $\vert \widehat{\omega}_k\vert\in [\omega_{\textbf{min}}, \omega_{\textbf{max}}]$, we set $\omega_k=\vert \widehat{\omega}_k\vert$. Otherwise, 
\begin{equation*}
  \omega_k=\begin{cases}
                    1 &   \mbox{if } \Vert R_\gamma^{\varepsilon_k}(x^k)\Vert>1, \\
                    10^5 &  \mbox{if }  \Vert R_\gamma^{\varepsilon_k}(x^k)\Vert<10^{-5}, \\
                    \Vert R_\gamma^{\varepsilon_k}(x^k)\Vert^{-1} & \mbox{otherwise}.
                  \end{cases}
\end{equation*}
The structural iteration is defined as $\st(\alpha_k,d^k)=(1-\alpha_k)\ov{x}^k+\alpha_k (x_k+d^k)$.
We select the sequence $\{\varepsilon_k=\frac{1}{(k+1)^2}\}_{k\in \Nz}$ for the update step in \eqref{eq:alg:ingrad:upd}

 \vspace{2mm}
\item [{\bf (iii)}] \textbf{(Implementation and comparative analysis)} 
In Subsection~\ref{subsec:invprob}, we evaluate the performance of HiFBA and Boosted HiFBA against the SG-GDSS algorithm. Furthermore, we compare these methods with the subgradient method employing a constant step-size (SG-CSS) for step-sizes $\alpha \in \{0.01, 0.1, 1\}$.
Moreover, in Subsection~\ref{subsec:RegularizedNMF}, we compare HiFBA and Boosted HiFBA against the Bregman proximal gradient method (BPG) and convex-concave inertial BPG (CoCaIn) \cite{bolte2018first,Mukkamala2019Beyond,Mukkamala20coc}.
In Subsection~\ref{subsec:invprob}, all algorithms are terminated after 2 seconds, whereas in Subsection~\ref{subsec:RegularizedNMF}, each algorithm is executed for 1000 iterations. All experiments are conducted in Python on a laptop equipped with a 12th Gen Intel$\circledR$ Core$^{\text{TM}}$ i7-12800H CPU (1.80 GHz) and 16 GB of RAM. An implementation of HiFBA is publicly available at \url{https://sites.google.com/view/akabgani/software}.
\end{description}
\vspace{-4mm}
\subsection{{\bf Linear inverse problem with clipped quadratic penalty}}
\label{subsec:invprob}
We apply HiFBA and Boosted HiFBA to a nonsmooth nonconvex regression problem, as outlined in Example~\ref{ex:RLInv}, specifically \eqref{eq:lininv:gform}. 
The regularizer $R$ is the clipped quadratic penalty \cite{Chen2014Convergence} to promote sparsity.
The objective function is given by
\begin{equation} \label{eq:objective_problem:prob1}
    {\mathop {\mathrm{\bs\min}}\limits_{x\in \R^n}}\  \frac{1}{q} \Vert Ax - b\Vert^q_q +  \lambda\sum_{i=1}^{n} \left(\left(2 \vert x_i\vert - x_{i}^2\right)\iota_{\vert x_i\vert\leq 1}+\iota_{\vert x_i\vert> 1}\right),
\end{equation}
where for a set $C$, $\iota_C (x)=1$ if $x \in C$ and $\iota_{C} (x) = 0$ otherwise. 
The regularization parameter is set to $\lambda = 1$. The data fidelity term $f(x) := \frac{1}{q} \Vert Ax - b\Vert^q_q$ belongs to $\maj{L_p}{p}{\R^n}$ with $p=q$, and the regularizer is proper, nonsmooth lsc function.
We assess the impact of $p =q \in \{1.1, 1.5, 1.75, 2\}$ on algorithmic performance, measured using
the signal-to-noise ratio (SNR).
For HiFBA and Boosted HiFBA, we initialize the parameters as $\gamma = \sigma = 1$. At each iteration, we compute $L_p = \frac{\Vert \nabla f(x_k) - \nabla f(x_{k-1}) \Vert}{\Vert x_k - x_{k-1} \Vert^{p-1}}$ and update $\gamma = 0.99 L_p^{-1}$ and $\sigma = 0.99 \left( \frac{1 - \gamma L_p}{p \gamma} \right)$. Additionally, we set $\vartheta = 0.72$.

We generate synthetic data as follows. The data matrix $A \in \mathbb{R}^{m \times n}$, with $m = 500$ and $n = 1000$, has entries drawn independently from a standard Gaussian distribution $\mathcal{N}(0, \frac{1}{m})$. The true signal $x_{\text{true}} \in \mathbb{R}^n$ is sparse, with a sparsity level of $10\%$, where the positions of non-zero elements are selected uniformly at random, and their magnitudes are sampled from $\mathcal{N}(0, 5^2)$. The clean observation vector is computed as $b_{\text{true}} = A x_{\text{true}}$. To simulate non-Gaussian, heavy-tailed noise, we corrupt $b_{\text{true}}$ with additive Laplace noise $\nu \in \mathbb{R}^m$, leading to the observation vector $b = b_{\text{true}} + \nu$, where each component $\nu_i$ is independently drawn from a Laplace distribution with zero mean and scale parameter $s = \frac{0.1 \Vert b_{\text{true}} \Vert}{\sqrt{2m}}$, corresponding to a $10\%$ noise level.

We evaluate the impact of the parameter $q$ on algorithmic performance, measured via SNR in decibels, defined as
$\mathrm{SNR}_{\mathrm{dB}} = 10\log_{10}\frac{\Vert x_{\text{true}}\Vert}{\Vert \widehat{x}- x_{\text{true}}\Vert}$,
where $\widehat{x}$ is the output of algorithm.
The initial parameter $\beta_0$ and resulted $\mathrm{SNR}_{\mathrm{dB}}$ values are presented in Table~\ref{tab:rep:num:clip}. 
Figure~\ref{fig:invprob:clip} illustrates the SNR and relative error as functions of iteration count for $q = 1.1$, comparing Boosted HiFBA, HiFBA, and subgradient-based methods. 
The findings indicate that lower $q$ values yield higher $\mathrm{SNR}_{\mathrm{dB}}$, suggesting improved reconstruction accuracy. Furthermore, HiFBA and Boosted HiFBA achieve higher $\mathrm{SNR}_{\mathrm{dB}}$ values compared to subgradient-based methods, demonstrating their superior efficiency. Moreover, Figure~\ref{fig:invprob:clip} demonstrates that Boosted HiFBA achieves target SNR and relative error levels in fewer iterations than other methods, indicating faster convergence.

\vspace{-4mm}
\begin{table}[ht]
\centering
\begin{tabular}{c c c c c  c c c c c}
\toprule
$p$ &  \multicolumn{2}{c}{HiFBA} &  \multicolumn{2}{c}{Boosted HiFBA} & \multicolumn{2}{c}{SG-GDSS} & \multicolumn{3}{c}{SG-CSS}  \\
\cmidrule{8-10}
&                  &                                                      &&& & & ($\alpha=0.01$) & ($\alpha=0.1$) & ($\alpha=1$) \\
& $\beta_0$  & $\mathrm{SNR}_{\mathrm{dB}}$ &  $\beta_0$  & $\mathrm{SNR}_{\mathrm{dB}}$ & $\beta_0$ &$\mathrm{SNR}_{\mathrm{dB}}$  & $\mathrm{SNR}_{\mathrm{dB}}$ & $\mathrm{SNR}_{\mathrm{dB}}$ &$\mathrm{SNR}_{\mathrm{dB}}$ \\
\midrule
1.10  &0.97 &11.47 & 0.86 & \textbf{11.70} & 0.97 & 0.69 & 0.66 & 9.34 & 11.44 \\
1.50  & 0.80  &10.45 & 0.76 & \textbf{10.83} & 0.80 & 0.17 & 1.37 & 7.25 & 10.35 \\
1.75 & 0.94  & 9.53  & 0.86& \textbf{10.33} & 0.94 & 0.72 & 1.59 & 6.02 & 9.53\\
2.00  & 0.99  &7.51 & 0.75 & \textbf{8.54} & 0.99 & 2.65 & 1.79 & 5.25 & 7.47 \\
\bottomrule
\end{tabular}
\caption{$\beta_0$ values and $\mathrm{SNR}_{\mathrm{dB}}$ results for different methods across $p$-values}
\label{tab:rep:num:clip}
\end{table}

\begin{figure}[ht]
    \centering
    \begin{subfigure}{0.42\textwidth}
        \centering
        \includegraphics[width=1.1\textwidth]{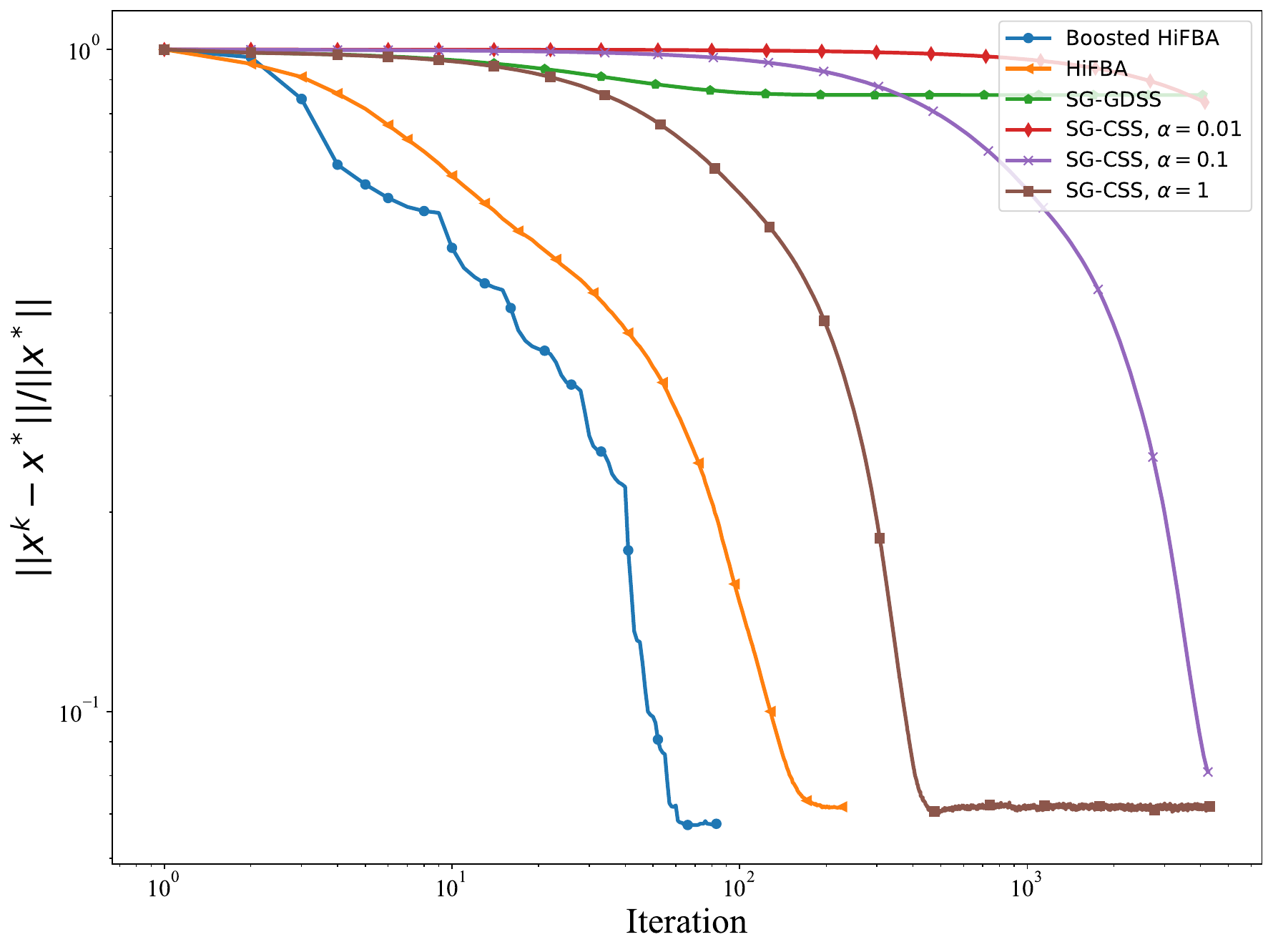}
        \caption{Relative error versus iteration}
    \end{subfigure}
    \qquad\qquad\quad
    \begin{subfigure}{0.42\textwidth}
        \centering
        \includegraphics[width=1.1\textwidth]{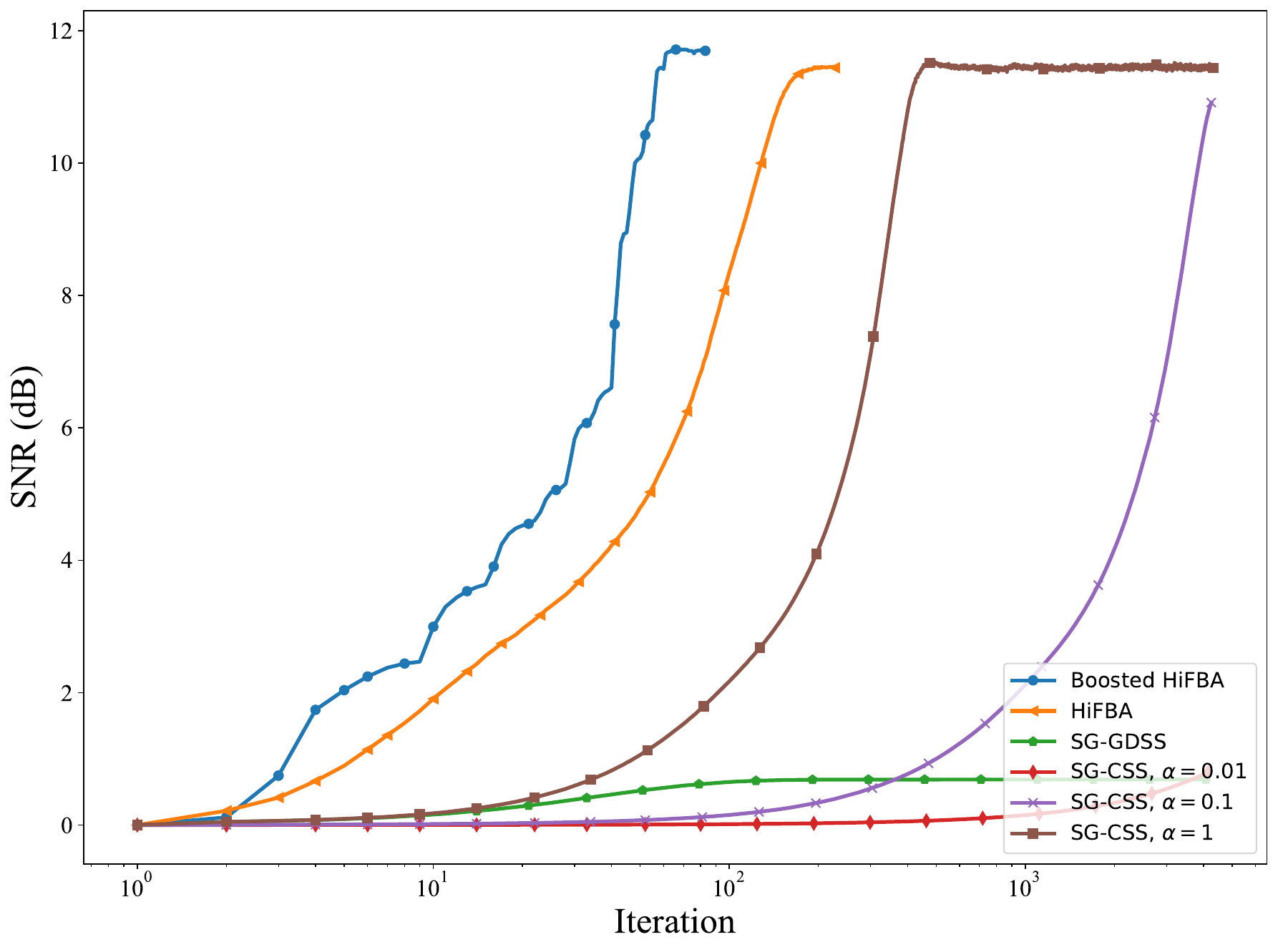}
        \caption{$\mathrm{PSNR}_{\mathrm{dB}}$ versus iteration}
    \end{subfigure}

    \caption{(a) Relative errors and (b) $\mathrm{SNR}_{\mathrm{dB}}$ values versus iteration count for Boosted HiFBA, HiFBA, and subgradient methods with $p = q = 1.1$. Boosted HiFBA achieves target SNR and relative error levels in fewer iterations.}
    \label{fig:invprob:clip}
\end{figure}

\subsection{{\bf Regularized nonnegative matrix factorization}}
\label{subsec:RegularizedNMF}
In this subsection, we apply HiFBA and Boosted HiFBA to a regularized nonnegative matrix factorization (NMF) problem, considered as a special case of matrix factorization introduced in Example~\ref{ex:relsmooth:pp}.  
Let $X\in\R^{m\times n}$ and let the rank parameter $r\ll\bs\min\{m,n\}$.  
The regularized NMF problem is formulated as
\begin{equation}\label{eq:mainp:ex:relsmooth:pp:num}
{\mathop {\mathrm{\bs\min}}\limits_{U\in\R^{m\times r},\;V\in\R^{n\times r}}}\; \gf(U,V) \;=\; \frac{1}{2}\Vert X- U V^\top\Vert_F^2 +\iota_{\{U\ge 0\}}(U) + \iota_{\{V\ge 0\}}(V)+\lambda\left(\Vert U\Vert_1+\Vert V\Vert_1\right).
\end{equation}
Following the approach of~\cite{Mukkamala2019Beyond}, we apply both the Bregman Proximal Gradient (BPG) and CoCaIn methods directly to solve~\eqref{eq:mainp:ex:relsmooth:pp:num}, using the kernel $h$ defined in~\eqref{eq:ex:relsmooth:pp:1}.  
For HiFBA and Boosted HiFBA, we define
$f(U,V)=\frac{1}{2}\Vert X- U V^\top\Vert_F^2$ and
\begin{equation}\label{eq:gform:ex:relsmooth:pp:num}
g(U,V) = \iota_{\{U\ge 0\}}(U) + \iota_{\{V\ge 0\}}(V)+\lambda\left(\Vert U\Vert_1+\Vert V\Vert_1\right),
\end{equation}
and solve the composite problem~\eqref{eq2:mainp:ex:relsmooth:pp} using these two algorithms.  
In both HiFBA and Boosted HiFBA, we set $p=2$, which allows a closed-form computation of $\Tprox{\gf}{\gamma}{p}(x^k)$ via the root of a cubic equation, as described in Lemma~\ref{lem:closedform}.  
In Boosted HiFBA, the parameters are chosen as $\vartheta = 0.95$ and $\sigma = \tfrac{0.99}{2\gamma}$.  
Indeed, from Algorithm~\ref{alg:inexact}, $\sigma$ must satisfy $\sigma\in\big(0,\tfrac{1-\gamma L_p}{p\gamma}\big)$.  
Since $L_p$ in~\eqref{eq:ex:relsmooth:pp:1} can be freely adjusted, we choose it such that $\gamma L_p$ is close to zero, allowing us to take $\sigma = \tfrac{0.99}{2\gamma}$.

For the numerical experiments, we set $\lambda=0.1$ and consider two datasets:  
(i) Medulloblastoma dataset (M-dataset)~\cite{Brunet04Metagenes} with $X\in\R^{5893\times34}$, and  
(ii) Synthetic dataset (S-dataset) with $X\in\R^{200\times200}$ generated with uniformly distributed nonnegative entries, i.e., $X\sim\mathcal{U}(0,0.1)$.  

In both cases, the factorization rank is set to $r=5$, and the initial matrices are initialized as $U^0 = 0.1\,\mathbf{1}_{m\times r}$ and $V^0 = 0.1\,\mathbf{1}_{n\times r}$.  
For the Medulloblastoma dataset, we set $\gamma=0.95$ for HiFBA and $\gamma=9$ for Boosted HiFBA, whereas for the synthetic dataset, we set $\gamma=300$ for both HiFBA and Boosted HiFBA.

Figure~\ref{fig:nmf} presents the objective value versus iteration count for BPG, CoCaIn, HiFBA, and Boosted HiFBA on these datasets.  
Subfigure~\ref{fig:nmf}(a) illustrates that HiFBA achieves lower objective values than BPG on the Medulloblastoma dataset, while Boosted HiFBA further accelerates convergence and attains smaller function values in fewer iterations than CoCaIn.  
For the synthetic dataset, HiFBA and BPG exhibit comparable performance, whereas Boosted HiFBA consistently outperforms the other methods, demonstrating faster convergence and improved solution quality.

\begin{figure}[ht]
    \centering
    \begin{subfigure}{0.42\textwidth}
        \centering
        \includegraphics[width=1.1\textwidth]{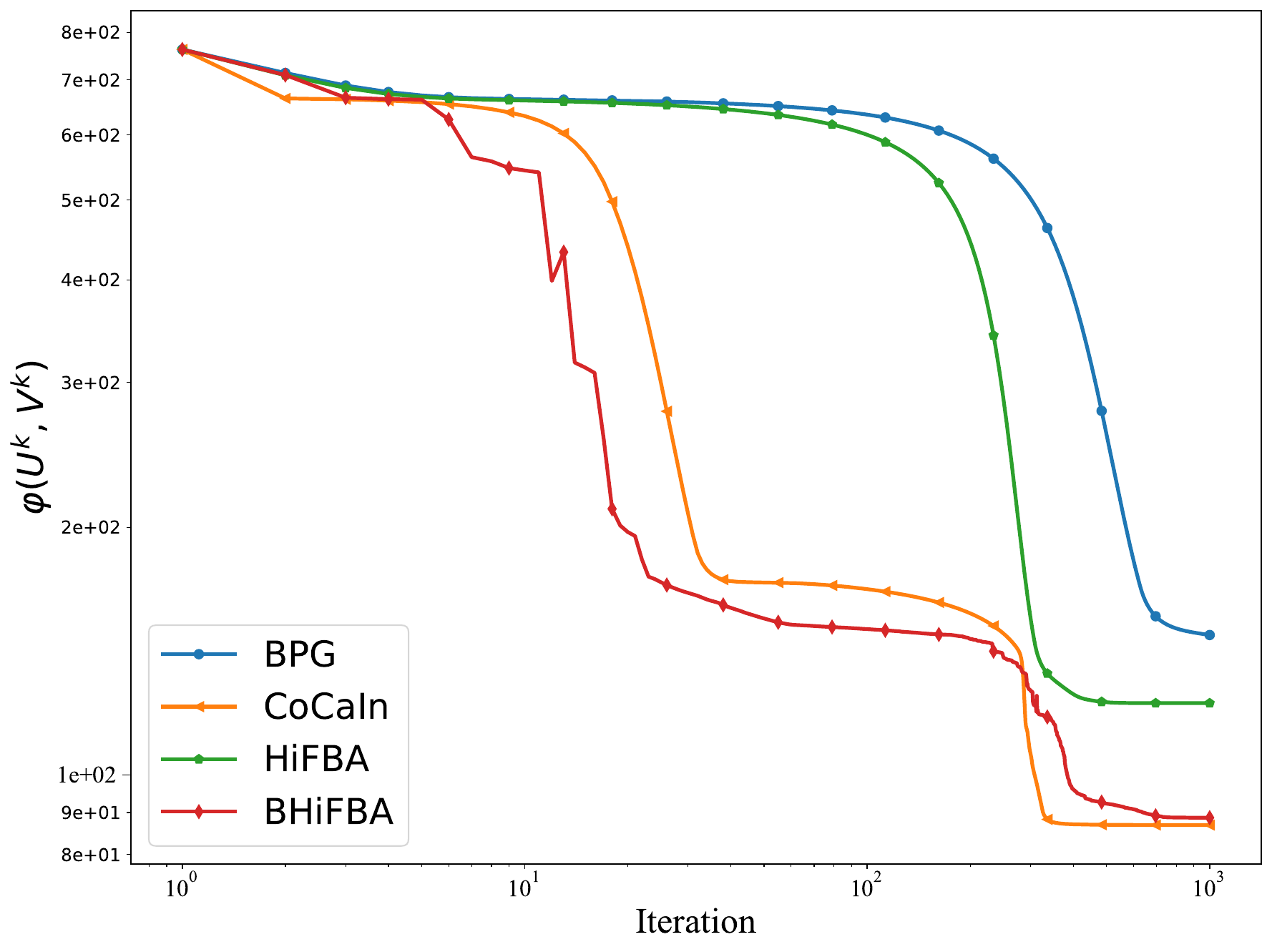}
        \caption{Objective value versus iteration (M-dataset)}
    \end{subfigure}
    \qquad\qquad\quad
    \begin{subfigure}{0.42\textwidth}
        \centering
        \includegraphics[width=1.1\textwidth]{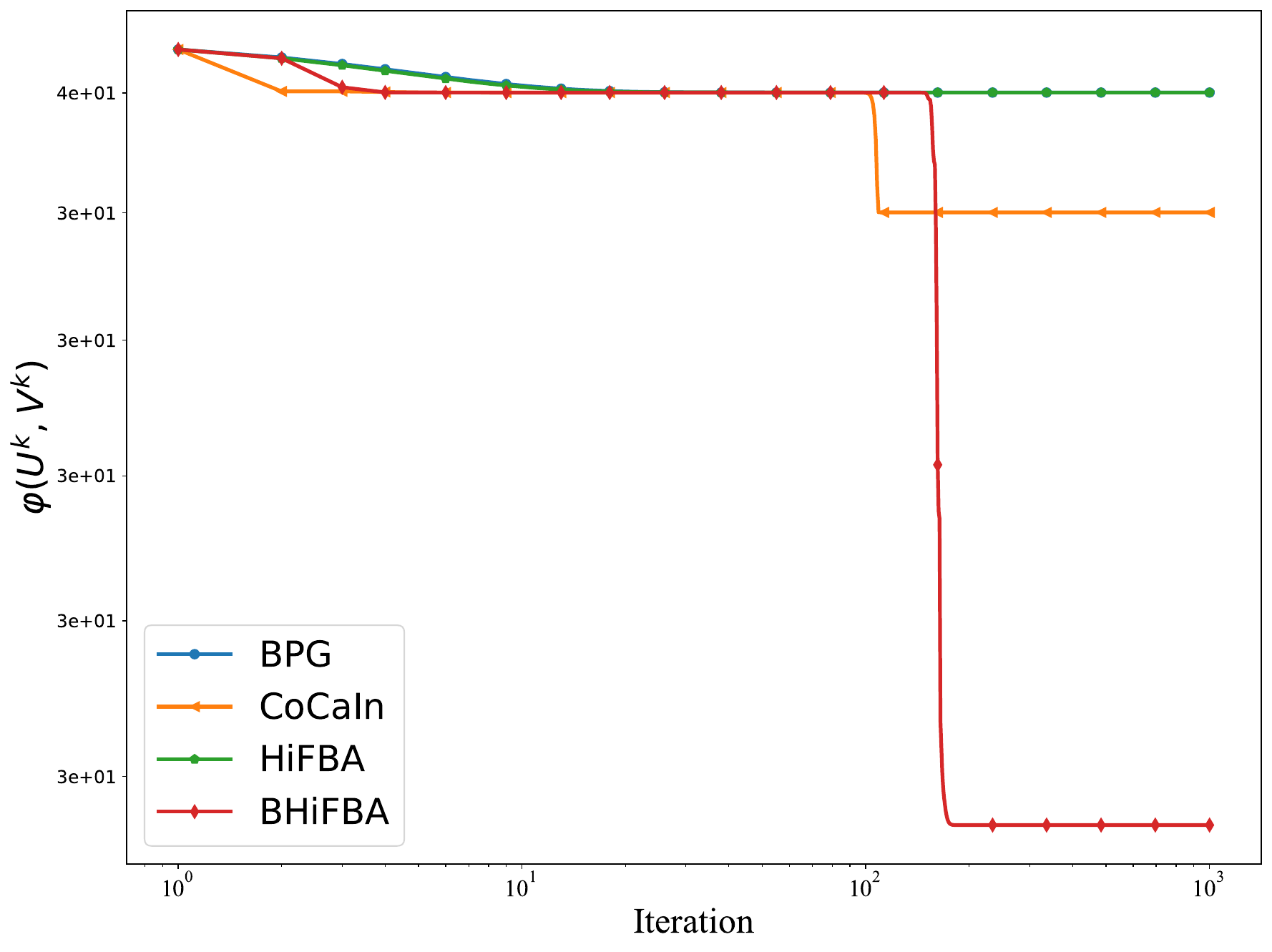}
        \caption{Objective value versus iteration (S-dataset)}
    \end{subfigure}

    \caption{Objective values versus iteration count for BPG, CoCaIn, HiFBA, and Boosted HiFBA in regularized NMF.}
    \label{fig:nmf}
\end{figure}


\section{Conclusion}\label{sec:disc}
In this paper, we introduced a first-order majorization-minimization framework to tackle structured nonsmooth and nonconvex optimization problems. To this end, we presented the concept of a high-order majorant and established its connection to the class of paraconcave functions -- involving functions with H\"{o}lder continuous gradients for $p \in (1, 2] $, concave functions for any $p>1$, and beyond. In the composite setting of problem \eqref{eq:mainproblemcom}, this majorant reduces to a variant of the celebrated descent lemma, which we analyze across several classes of optimization problems. This framework enables the design of flexible splitting algorithms. Specifically, we develop an inexact high-order forward-backward splitting algorithm (HiFBA) and its line-search-accelerated variant (Boosted HiFBA), both leveraging an inexact oracle for the high-order forward-backward envelope (HiFBE), whose fundamental properties we thoroughly investigated. We established the subsequential convergence under mild inexactness conditions, as well as the global and linear convergence under the Kurdyka–\L{}ojasiewicz property. Our preliminary numerical experiments on some linear inverse and regularized nonnegative matrix factorization problems support and validate our theoretical results.

\appendix
\section{Auxiliary results}
\label{sec:app}

\renewcommand{\thetheorem}{A.\arabic{theorem}}
\setcounter{theorem}{0} 

In this section, we present auxiliary results that facilitate the description of methods and proofs. We begin with a lemma establishing the relationship between the coercivity of a function and its HiFBE.

\begin{lemma}[Coercivity]\label{lem:coer}
Let $p>1$, $f:\R^n\to \R$ be Fr\'{e}chet differentiable, and $g:\R^n\to \Rinf$ be a proper lsc function. If $f\in \maj{L_p}{p}{\R^n}$ and the function $\gf:=f+g$ is coercive, then for each $\gamma\in (0, L_p^{-1})$, the function $\fgam{\gf}{p}{\gamma}$ is coercive.
\end{lemma}
\begin{proof}
By the properties of infima, for any $\varepsilon>0$ and $x\in\R^n$,
there exists $y^\varepsilon_x\in\R^n$ such that, for each $\gamma\in  (0, L_p^{-1})$,
\begin{equation}\label{eq:hiordermor:coer}
f(y^\varepsilon_x)+g(y^\varepsilon_x)\leq f(x)+\langle \nabla f(x), y^\varepsilon_x - x\rangle+g(y^\varepsilon_x)+\frac{L_p}{p}\Vert y^\varepsilon_x - x\Vert^p
\leq \fgam{\gf}{p}{\gamma}(x)+\varepsilon.
\end{equation}
By contradiction, suppose that $\fgam{\gf}{p}{\gamma}$ is not coercive. 
Then, there exists a sequence $\{x^k\}_{k\in \mathbb{N}} \subseteq \R^{n}$ with $\Vert x^k\Vert \rightarrow \infty$ and $\bs\lim_{k\to \infty}\fgam{\gf}{p}{\gamma}(x^k)<\infty$. Let $\{y^\varepsilon_{x^k}\}_{k\in \mathbb{N}}$ be a corresponding sequence to $\{x^k\}_{k\in \mathbb{N}}$ such that each $y^\varepsilon_{x^k}$ satisfies \eqref{eq:hiordermor:coer}. Since $\varepsilon$ is fixed and, by convention $\infty - \infty = \infty$, relation \eqref{eq:hiordermor:coer} implies that 
$\bs\lim_{k\to \infty}\gf(y^\varepsilon_{x^k})<\infty$ and $\bs\lim_{k\to  \infty}\Vert x^k-y^\varepsilon_{x^k}\Vert^p<\infty$.
From the coercivity of $\gf$ and $\bs\lim_{k\to \infty}\gf(y^\varepsilon_{x^k})<\infty$, we deduce that 
$\bs\lim_{k\to \infty}\Vert y^\varepsilon_{x^k}\Vert< \infty$. 
However, using the inequality
$\Vert x^k\Vert^p\leq 2^{p-1}\left(\Vert x^k -y^\varepsilon_{x^k}\Vert^p+\Vert y^\varepsilon_{x^k}\Vert^p\right)$,
since $\Vert y^\varepsilon_{x^k}\Vert$ remains bounded while $\Vert x^k\Vert\to \infty$, we have
$\bs\lim_{k \to\infty}\Vert x^k -y^\varepsilon_{x^k}\Vert=\infty$, which is a contradiction. Thus, $\fgam{\gf}{p}{\gamma}$ is coercive.
\end{proof}

\begin{remark}\label{rem:lem:coer}
From Lemma~\ref{lem:coer} and \cite[Proposition~11.12]{Bauschke17}, for each $\lambda\in \R$,  the sublevel set 
$\mathcal{L}(\fgam{\gf}{p}{\gamma}, \lambda):=\{x\in \R^n\mid \fgam{\gf}{p}{\gamma}(x)\leq \lambda\}$, is bounded.
\end{remark}

We present results about the Mordukhovich subdifferential of HiFBE.
\begin{lemma}[Subdifferential of HiFBE]\label{lem:frech:hifbe}
Let $p>1$, $f\in \mathcal{C}^2(U)$ where $U \subset \R^n$ is an open neighborhood of $\ov{x}\in \dom{\gf}$,
and let $g$ be a high-order prox-bounded function with threshold $\gamma^{g, p}>0$.
Then, for each $\gamma\in (0, \gamma^{g, p})$ and $\ov{y}\in \Tprox{\gf}{\gamma}{p}(\ov{x})$,
we have
\begin{equation}\label{eq:cor:morsubhifbe}
\partial\fgam{\gf}{p}{\gamma}(x)\subseteq \left\{\nabla^2f(\ov{x})(\ov{y}-\ov{x})+\frac{1}{\gamma}\Vert \ov{x}-\ov{y}\Vert^{p-2}(\ov{x}-\ov{y})\right\}.
\end{equation}
\end{lemma}
\begin{proof}
Let $\eta\in \widehat{\partial} \fgam{\gf}{p}{\gamma}(\ov{x})$ and $\ov{y}\in \Tprox{\gf}{\gamma}{p}(\ov{x})$ be arbitrary.
Define the function $\gh:\R^n\to \Rinf$ as 
\[
\gh(x):=f(x)+\langle \nabla f(x), \ov{y}-x\rangle+\frac{1}{p\gamma}\Vert \ov{y}-x\Vert^p.
\]
 For any $x\in \R^n$, we have
\[
\begin{aligned}
 \fgam{\gf}{p}{\gamma}(x)- \fgam{\gf}{p}{\gamma}(\ov{x})-\langle\eta, x - \ov{x}\rangle
 &\leq  f(x)+\langle\nabla f(x), \ov{y}-x\rangle+g(\ov{y})+\frac{1}{p\gamma}\Vert  \ov{y}-x\Vert^p
 \\&~~~- f(\ov{x})-\langle \nabla f(\ov{x}), \ov{y}-\ov{x}\rangle-g(\ov{y})-\frac{1}{p\gamma}\Vert  \ov{x}-\ov{y}\Vert^p -\langle\eta, x - \ov{x}\rangle
 \\&=\gh(x)-\gh(\ov{x})-\langle\eta, x - \ov{x}\rangle.
\end{aligned}
\]
Thus,
\[
0\leq {\mathop {\mathop {\bs\liminf }\limits _{ x\to \ov{x},}} \limits _{x\neq \ov{x}}}\frac{\fgam{\gf}{p}{\gamma}(x)- \fgam{\gf}{p}{\gamma}(\ov{x})-\langle\eta, x - \ov{x}\rangle}{\Vert x - \ov{x}\Vert}\leq 
{\mathop {\mathop {\bs\liminf }\limits _{ x\to \ov{x},}} \limits _{x\neq \ov{x}}}
\frac{\gh(x)-\gh(\ov{x})-\langle\eta, x - \ov{x}\rangle}{\Vert x - \ov{x}\Vert}.
\]
which implies that $\eta\in \widehat{\partial} \gh(\ov{x})$.
By \cite[Exercise 8.8 (c)]{Rockafellar09},  we obtain
$\eta=\nabla^2f(\ov{x})(\ov{y}-\ov{x})+\frac{1}{\gamma} \Vert  \ov{x}-\ov{y}\Vert^{p-2}(\ov{x}-\ov{y})$.
Now, let $\ov{\zeta}\in \partial\fgam{\gf}{p}{\gamma}(\ov{x})$. By definition, there are sequences $x^k\to \ov{x}$ and $\zeta^k\in \widehat{\partial}\fgam{\gf}{p}{\gamma}(x^k)$ such that $\fgam{\gf}{p}{\gamma}(x^k)\to \fgam{\gf}{p}{\gamma}(\ov{x})$ and $\zeta^k\to \ov{\zeta}$. For each $k$, we have 
  \[\zeta^k=\nabla^2f(x^k)(y^k-x^k)+\frac{1}{\gamma} \Vert x^k - y^k \Vert^{p-2} (x^k - y^k),\]
where $y^k\in \Tprox{\gf}{\gamma}{p}(x^k)$. By Theorem~\ref{th:hopbfb}~$\ref{th:hopbfb:proxb:conv}$, the sequence $\{y^k\}_{k\in \mathbb{N}}$ has a cluster point 
$\ov{y}\in \Tprox{\gf}{\gamma}{p}(\ov{x})$, such that $\ov{\zeta} =\nabla^2f(\ov{x})(\ov{y}-\ov{x}) \frac{1}{\gamma} \Vert \ov{x} - \ov{y} \Vert^{p-2} (\ov{x} - \ov{y})$. 
\end{proof}

The following lemma extends the argument of Lemma~D.2 and Proposition~D.3 
in \cite{Mukkamala2019Beyond} to include the proximal quadratic term. 
\begin{lemma}[Proximal update with nonnegativity and $\ell_1$ regularization]\label{lem:closedform}
Consider problem~\eqref{eq:mainp:ex:relsmooth:pp:num}, the kernel $h$ given in \eqref{eq:ex:relsmooth:pp:1} with $a=3$ and $b=\|X\|_F$, and 
the function $g$ as \eqref{eq:gform:ex:relsmooth:pp:num}.
For some $(U^k,V^k)\in\R^{m\times r}\times\R^{n\times r}$, define
\[
P^k :=\nabla_U f(U^k,V^k)-\nabla_U h(U^k,V^k), \qquad 
Q^k :=\nabla_V f(U^k,V^k)-\nabla_V h(U^k,V^k).
\]
For $\gamma>0$, consider the proximal subproblem
\begin{equation}\label{eq:prox-subproblem}
{\mathop {\mathrm{\bs\min}}\limits_{U\in\R^{m\times r}, V\in\R^{n\times r}}}  \Psi(U,V):=
\left\langle P^k,U\right\rangle+\left\langle Q^k, V\right\rangle+g(U,V)+h(U,V)
+\frac{1}{2\gamma}\left(\|U-U^k\|_F^2+\|V-V^k\|_F^2\right).
\end{equation}
Define the terms $\widehat P^k=P^k-\frac{1}{\gamma}U^k$ and $\widehat Q^k=Q^k-\frac{1}{\gamma}V^k$,
and set $\Theta_U=\widehat P^k+\lambda\,\mathbf{1}_m\mathbf{1}_r^\top$ and 
$\Theta_V=\widehat Q^k+\lambda\,\mathbf{1}_n\mathbf{1}_r^\top$, where $\mathbf{1}_d$ is the all-ones vector in $\R^d$.
Let $A_U=\|\Pi_+(-\Theta_U)\|_F$ and $A_V=\|\Pi_+(-\Theta_V)\|_F$,
where $\Pi_+(\cdot)$ denotes the projection onto the nonnegative orthant. 
Then every minimizer of \eqref{eq:prox-subproblem} is of the form
\[
U^{k+1}=r\,\Pi_+\left(-\Theta_U\right), \qquad
V^{k+1}=r\,\Pi_+\left(-\Theta_V\right),
\]
where $r\ge 0$ is the unique nonnegative real root of
\begin{equation}\label{eq:cubic}
3(A_U^2+A_V^2)\,r^3+\left(\|X\|_F+\frac{1}{\gamma}\right)r-1=0.
\end{equation}
\end{lemma}
\begin{proof}
We begin by rewriting the objective function in~\eqref{eq:prox-subproblem} as
\[
\Psi(U,V)
=\left[\langle P^k,U\rangle+g_U(U)+\frac{1}{2\gamma}\|U-U^k\|_F^2\right]
+\left[\langle Q^k,V\rangle+g_V(V)+\frac{1}{2\gamma}\|V-V^k\|_F^2\right]
+h(U,V),
\]
where $g_U(U)=\iota_{\{U\ge0\}}(U)+\lambda\|U\|_1$ and 
$g_V(V)=\iota_{\{V\ge0\}}(V)+\lambda\|V\|_1$. 
All terms are separable in $(U,V)$ except $h(U,V)$, which couples them
through the term $\|U\|_F^2+\|V\|_F^2$.
For any fixed Frobenius norms $t_U=\|U\|_F$ and $t_V=\|V\|_F$, 
the function $h(U,V)$ is constant, and the minimization over  
 $U$ and $V$ decouples as
\[
{\mathop {\mathrm{\bs\min}}\limits_{U\ge0,\,\|U\|_F=t_U}} 
  \left\{\langle P^k,U\rangle+\lambda\|U\|_1
  +\frac{1}{2\gamma}\|U-U^k\|_F^2\right\},
\qquad
{\mathop {\mathrm{\bs\min}}\limits_{V\ge0,\,\|V\|_F=t_V}}
  \left\{\langle Q^k,V\rangle+\lambda\|V\|_1
  +\tfrac{1}{2\gamma}\|V-V^k\|_F^2\right\}.
\]
Expanding 
$\|U-U^k\|_F^2=\|U\|_F^2-2\langle U^k,U\rangle+\|U^k\|_F^2$
and discarding constants independent of $U$,
the objective depending on $U$ can be rewritten as
\[
\langle P^k-\tfrac{1}{\gamma}U^k,\,U\rangle+\lambda\|U\|_1
+\tfrac{1}{2\gamma}\|U\|_F^2+{\rm const}.
\]
The term $\tfrac{1}{2\gamma}\|U\|_F^2$ depends only on $t_U$
and will be handled together with $h(U,V)$ in the outer minimization.
By defining $\widehat P^k :=P^k-\tfrac{1}{\gamma}U^k$ and $\widehat Q^k :=Q^k-\tfrac{1}{\gamma}V^k$,
 the inner subproblems for fixed $(t_U,t_V)$ reduce to
\[
{\mathop {\mathrm{\bs\min}}\limits_{U\ge0,\,\|U\|_F\leq t_U}} \,
   \left\langle \widehat P^k,U \right\rangle+\lambda\|U\|_1,
\qquad
{\mathop {\mathrm{\bs\min}}\limits_{V\ge0,\,\|V\|_F\leq t_V}} \,
   \left\langle \widehat Q^k,V \right\rangle+\lambda\|V\|_1.
\]
These two problems are completely independent and
their solutions determine the optimal solutions for $U$ and $V$, 
while the shared norms $(t_U,t_V)$ will later be optimized jointly
through the coupling term $h(U,V)$.
Hence, the problem \eqref{eq:prox-subproblem} can be decomposed as a two-level minimization
\begin{equation}\label{eq:profclosedform}
{\mathop {\mathrm{\bs\min}}\limits_{U\in\R^{m\times r}, V\in\R^{n\times r}}}  \Psi(U,V)=
{\mathop {\mathrm{\bs\min}}\limits_{t_U,t_V \ge 0}}
\left[
{\mathop {\mathrm{\bs\min}}\limits_{U \ge 0,\, \|U\|_F \le t_U}}
      \psi_U(U)
      +
  {\mathop {\mathrm{\bs\min}}\limits_{V \ge 0,\, \|V\|_F \le t_V}}  
      \psi_V(V)
    +
    h(t_U,t_V)
  \right],
\end{equation}
where
\[
\psi_U(U) = \langle \widehat P^k, U \rangle + \lambda \|U\|_1, 
\qquad
\psi_V(V) = \langle \widehat Q^k, V \rangle + \lambda \|V\|_1,
\]
and
\[
h(t_U,t_V)
= \frac{a}{4}(t_U^2+t_V^2)^2
+ \frac{b}{2}(t_U^2+t_V^2)
+ \frac{1}{2\gamma}(t_U^2+t_V^2).
\]
Since $\|U\|_1 = \langle \mathbf{1}_m\mathbf{1}_r^\top, U \rangle$ for $U\ge0$,
each subproblem can be written as
\[
{\mathop {\mathrm{\bs\min}}\limits_{U\ge0,\,\|U\|_F\leq t_U}} 
    \left\langle \Theta_U, U \right\rangle,
\qquad 
\Theta_U := \widehat P^k + \lambda\,\mathbf{1}_m\mathbf{1}_r^\top.
\]
By  \cite[Lemma~D.1]{Mukkamala2019Beyond}, the minimum value is 
$-t_U\|\Pi_+(-\Theta_U)\|_F$,
and the corresponding minimizer is
\[
U^*(t_U) = 
t_U\,\frac{\Pi_+(-\Theta_U)}{A_U},
\qquad 
A_U := \|\Pi_+(-\Theta_U)\|_F.
\]
An analogous argument for $V$ yields 
$V^*(t_V) = t_V\,\Pi_+(-\Theta_V)/A_V$ with 
$A_V = \|\Pi_+(-\Theta_V)\|_F$.
Substituting the optimal solutions into~\eqref{eq:profclosedform}
reduces the problem to
\[
{\mathop {\mathrm{\bs\min}}\limits_{t_U,t_V\ge0}} 
  -t_UA_U - t_VA_V
  + \tfrac{a}{4}(t_U^2+t_V^2)^2
  + \tfrac{1}{2}\big(b+\tfrac{1}{\gamma}\big)(t_U^2+t_V^2).
\]
The first-order optimality conditions are
\begin{align*}
-A_U + a(t_U^2+t_V^2)t_U + \big(b+\tfrac{1}{\gamma}\big)t_U &= 0,\\
-A_V + a(t_U^2+t_V^2)t_V + \big(b+\tfrac{1}{\gamma}\big)t_V &= 0.
\end{align*}
If $(A_U,A_V)\neq(0,0)$, these imply the proportionality 
$t_U=rA_U$ and $t_V=rA_V$ for some $r\ge0$.
Substituting these relations into either equation gives the cubic
\[
a(A_U^2+A_V^2)\,r^3 + \big(b+\tfrac{1}{\gamma}\big)r - 1 = 0.
\]
Because the left-hand side is strictly increasing on $[0,\infty)$, 
there exists a unique nonnegative real root $r$.
Finally, substituting $t_U=rA_U$ and $t_V=rA_V$ into 
$U^*(t_U)$ and $V^*(t_V)$ yields
\[
U^{k+1}=r\,\Pi_+(-\Theta_U), 
\qquad 
V^{k+1}=r\,\Pi_+(-\Theta_V),
\]
which are the desired updates. 
\end{proof}
\section{Omitted proofs}
\label{sec:app:proofs}

\begin{proof}[\textbf{Proof of Theorem~\ref{th:basichifbe}}]
\ref{th:basichifbe:dom} HiFBE can be written as
\[
   \fgam{\gf}{p}{\gamma}(x)=
   \mathop{\bs{\inf}}\limits_{\substack{
                                    (v,w)\in \R^n\times\R^n\\ 
                                    v-w=x}}\left\{f(x)+\langle \nabla f(x) , w\rangle +g(v)+\frac{1}{p\gamma}\Vert w\Vert^p\right\}.
   \]
Since $\dom{\langle \nabla f(x) , w\rangle +\frac{1}{p\gamma}\Vert w\Vert^p}=\R^n$ and $g$ is proper, we get $\dom{\fgam{\gf}{p}{\gamma}}=\R^n$.
The second part is followed by setting $y=x$ in the inner term in the curly brackets of \eqref{HFBE}.
\\
\ref{th:basichifbe:ineqforp}
It holds that
  \[
\gf(y)=f(y)+g(y) \leq f(x) + \langle \nabla f(x), y - x \rangle +\frac{L_p}{p}\Vert x - y \Vert ^{p}+g(y), \quad \forall x, y\in\R^n.
\]
Regarding $L_p<\frac{1}{\gamma}$, the claim obtains. The second claim follows
from this and $\fgam{\gf}{p}{\gamma}(x)=\ell(x,y)+\frac{1}{p\gamma}\Vert x - y\Vert^p$ for $y\in \Tprox{\gf}{\gamma}{p} (x)$.
\\
\ref{th:basichifbe:infimforp}
From Assertions~\ref{th:basichifbe:dom}~and~\ref{th:basichifbe:ineqforp}, $\bs\inf_{z\in\R^n}\gf(z)\leq \fgam{\gf}{p}{\gamma}(x)\leq \gf(x)$, for each $x\in\R^n$. 
The other inequality obtains from the definitions.
\\
\ref{th:basichifbe:infim2forp} This follows from part $\ref{th:basichifbe:infimforp}$.
\\
\ref{th:basichifbe:finite}  Let $\gamma\in (0, \gamma^{g, p})$ and $x\in\R^n$ be arbitrary. Then, there exists some $\lambda\in (0, 1)$ such that $\gamma<\lambda \gamma^{g, p}$. Now, from the convexity of function $y\mapsto f(x)+\langle \nabla f(x) , y - x\rangle $ for each $x\in \R^n$ and  Fact~\ref{fact:level-bound+locally uniform},
we have
\[
\begin{aligned}
-\infty&< \mathop{\bs{\inf}}\limits_{y\in \R^n} \left\{ g(y)+\frac{1}{p\frac{\gamma}{\lambda}}\Vert x- y\Vert^p\right\}
+ \mathop{\bs{\inf}}\limits_{y\in \R^n} \left\{f(x)+\langle \nabla f(x) , y - x\rangle+\frac{1}{p\frac{\gamma}{1-\lambda}}\Vert x- y\Vert^p\right\}
\nonumber\\
 &\leq
 \mathop{\bs{\inf}}\limits_{y\in \R^n} \left\{f(x)+\langle \nabla f(x) , y - x\rangle+g(y)+\frac{1}{p\gamma}\Vert x- y\Vert^p\right\}
 =\fgam{\gf}{p}{\gamma}(x).
\end{aligned}
\]
Together with Assertions~\ref{th:basichifbe:dom}, this implies $\fgam{\gf}{p}{\gamma}(x)\neq \pm\infty$.
\\
\ref{th:basichifbe:levelunif} Assume that $\Psi(y, x, \gamma)$  is not
 level-bounded in $y$ locally uniformly in $(x, \gamma)$. Then, there exist $(\ov{x}, \ov{\gamma})\in \R^n\times (0, \gamma^{g, p})$, $\theta\in \R$, sequences $(x^k, \gamma^k)\to (\ov{x}, \ov{\gamma})$, and $\{y^k\}_{k\in \mathbb{N}}$ with $\Vert y^k\Vert\to \infty$ such that for each $k\in \mathbb{N}$,
\begin{equation}\label{eq1:level-bound+locally uniform}
\Psi(y^k, x^k,\gamma^k)=f(x^k)+\langle \nabla f(x^k), y^k - x^k\rangle+g(y^k)+\frac{1}{p\gamma^k}\Vert x^k- y^k\Vert^p\leq \theta.
\end{equation}
 On the other hand, from Assertion~\ref{th:basichifbe:finite}, for $\hat{\gamma}\in (\ov{\gamma}, \gamma^{g, p})$,  there exists some $\beta\in \R^n$ such that
\[
f(\ov{x})+\langle \nabla f(\ov{x}), y^k -\ov{x}\rangle+g(y^k)+\frac{1}{p\hat{\gamma}}\Vert \ov{x}- y^k\Vert^p\geq \beta.
\]
   Together with \eqref{eq1:level-bound+locally uniform}, this implies
\[
\begin{aligned}
f(x^k) - f(\ov{x})+\langle \nabla f(x^k), y^k - x^k\rangle - \langle \nabla f(\ov{x}), y^k -\ov{x}\rangle 
+ \frac{1}{p\gamma^k}\Vert x^k- y^k\Vert^p-\frac{1}{p\hat{\gamma}}\Vert \ov{x}-y^k\Vert^p\leq \theta -\beta.
\end{aligned}
\]
Dividing both sides of above inequality by $\Vert y^k\Vert^p$ and taking $k\to +\infty$, we will have
$
0<\frac{1}{p\ov{\gamma}}-\frac{1}{p\hat{\gamma}}\leq 0,
$
which is a contradiction. As $\ell$ is lsc, then $\Psi$ has this property as well.
\\
\ref{th:basichifbe:con} This follows from Assertion~\ref{th:basichifbe:levelunif} and \cite[Theorem 1.17]{Rockafellar09}.
\\
\ref{th:basichifbe:finproxb}
From the assumption, there exists some $\beta_0\in \R$ such that for each $y\in \R^n$, we have
\begin{equation}\label{eq:lem:finproxb}
\beta_0\leq \fgam{\gf}{p}{\gamma}(\ov{x}) \leq f(\ov{x})+\langle \nabla f(\ov{x}), y - \ov{x}\rangle + g(y)+\frac{1}{p\gamma}\Vert y- \ov{x}\Vert^p.
\end{equation}
We claim that $\fgam{g}{\gamma,p}{}(\ov{x})>-\infty$. By contradiction, assume that  $\fgam{g}{\gamma,p}{}(\ov{x})=-\infty$. Thus, there is a sequence
$\{y^k\}_{k\in \mathbb{N}}\subseteq\R^n$ such that $\bs\lim_{k\to\infty}  \left(g(y^k)+\frac{1}{p\gamma}\Vert y^k- \ov{x}\Vert^p\right)=-\infty$. This shows that 
$\{\Vert y^k- \ov{x}\Vert \}_{k\in \mathbb{N}}$ is bounded from above. Thus 
\[
\mathop{\bs\lim}\limits_{k\to \infty} \langle \nabla f(\ov{x}), y^k - \ov{x}\rangle\leq \mathop{\bs\lim}\limits_{k\to \infty} \Vert \nabla f(\ov{x})\Vert \Vert y^k - \ov{x}\Vert<\infty.
\]
From these and \eqref{eq:lem:finproxb}, we get
\[
\beta_0\leq  \mathop{\bs\lim}\limits_{k\to\infty}  \left( f(\ov{x})+\langle \nabla f(\ov{x}), y^k - \ov{x}\rangle + g(y^k)+\frac{1}{p\gamma}\Vert y^k- \ov{x}\Vert^p\right)=-\infty,
\]
which is a contradiction.
\end{proof}

\begin{proof}[\textbf{Proof of Proposition~\ref{pro:KLp:hope}}]
Consider the auxiliary function $\Phi(x, y) := \ell(x, y) + \frac{1}{p \gamma} \| x - y \Vert^p$ and let $\ov{x}\in \bs{\rm Fix}(\Tprox{\gf}{\widehat{\gamma}}{p})$ be arbitrary. 
From Proposition~\ref{prop:relcrit}, $\ov{x}\in \bs{{\rm Fcrit}}(\gf)$. Additionally, 
$(\ov{x}, \ov{x})\in\Dom{\partial \Phi}$.
We have
\begin{equation}\label{eq:KLdefinition:subphi}
\partial \ell(x,y)=\left[\begin{matrix}
\nabla^2 f(x)(y-x)\\ \nabla f(x)+\partial g(y)
\end{matrix}\right],
\qquad
\partial \Phi(x,y)=\left[\begin{matrix}
\nabla^2 f(x)(y-x)+
\frac{1}{\gamma}\Vert x -y\Vert^{p-2} (x-y)\\ \nabla f(x)+\partial g(y)-\frac{1}{\gamma}\Vert x - y\Vert^{p-2} (x - y)
\end{matrix}\right].
\end{equation}
From the assumption,
there exist constants $c, r>0$  and $\eta\in (0, +\infty]$ such 
that 
\begin{equation}\label{eq:theorem:KLdefinition}
\dist^{\frac{1}{\theta}}(0, \partial \ell(x,y)))\geq c\left\vert\ell(x,y)-\ell(\ov{x}, \ov{x})\right\vert,~~\forall (x,y)\in \mb((\ov{x}, \ov{x}); r)\cap\Dom{\partial \ell},
\end{equation}
with $0<\left\vert \ell(x,y)-\ell(\ov{x}, \ov{x})\right\vert<\eta$.
Hence, 
\begin{equation}\label{eq:theorem:KLdefinition2}
\left(\Vert \nabla^2 f(x)(y-x)\Vert+\mathop{\bs{\inf}}\limits_{\zeta\in \partial g(y)}\Vert\nabla f(x)+\zeta\Vert\right)^{\frac{1}{\theta}}\geq c\left\vert\ell(x,y)-\ell(\ov{x}, \ov{x})\right\vert,~~\forall (x,y)\in \mb((\ov{x}, \ov{x}); r)\cap\Dom{\partial \ell},
\end{equation}
with $0<\left\vert \ell(x,y)-\ell(\ov{x}, \ov{x})\right\vert<\eta$.
Using \eqref{eq:KLdefinition:subphi} and \cite[Lemma~2.2]{Li18}, there exists a constant $c_0>0$, such that for $\lambda\in (0,\frac{1}{2})$, setting
$\eta_1=\lambda^{p-1}$ and $\eta_2=\left(\frac{\lambda}{1-\lambda}\right)^{p-1}$, we obtain for all $(x,y)\in \Dom{\partial \Phi}$, 
\begin{align*}
\dist^{\frac{1}{\theta}}(0, \partial \Phi(x,y))&\geq c_0\left(
\Vert \nabla^2 f(x)(y-x)\Vert^{\frac{1}{\theta}}+
(\frac{1}{\gamma})^{\frac{1}{\theta}}\Vert x - y\Vert^{\frac{p-1}{\theta}}+\mathop{\bs{\inf}}\limits_{\zeta\in \partial g(y)}\Vert \nabla f(x)+\zeta -\frac{1}{\gamma}\Vert x - y\Vert^{p-2} (x-y)\Vert^{\frac{1}{\theta}}\right)
\\&\overset{(i)}{\geq} c_0\left(
\Vert \nabla^2 f(x)(y-x)\Vert^{\frac{1}{\theta}}+
(\frac{1}{\gamma})^{\frac{1}{\theta}}\Vert x - y\Vert^{\frac{p-1}{\theta}}+\mathop{\bs{\inf}}\limits_{\zeta\in \partial g(y)}\eta_1\Vert \nabla f(x)+ \zeta\Vert^{\frac{1}{\theta}}-\eta_2(\frac{1}{\gamma})^{\frac{1}{\theta}}\Vert x - y\Vert^{\frac{p-1}{\theta}}\right)
\\&\geq c_1\left(\Vert \nabla^2 f(x)(y-x)\Vert^{\frac{1}{\theta}}+\mathop{\bs{\inf}}\limits_{\zeta\in \partial g(y)}\Vert\nabla f(x)+ \zeta\Vert^{\frac{1}{\theta}}+\Vert x - y\Vert^{\frac{p-1}{\theta}}\right),
\end{align*}
where
$c_1:=c_0~\bs\min\{\eta_1, (1-\eta_2)\left(\frac{1}{\gamma}\right)^{\frac{1}{\theta}}\}$, and $(i)$ follows from the fact that for each $a, b\in \R^n$ and $\lambda\in (0,1)$ and $p\geq 1$,  \cite[Lemma~2.1]{Kabgani24itsopt}
\[
\Vert a+b\Vert^p\geq \lambda^{p-1}\Vert a\Vert^p-\left(\frac{\lambda}{1-\lambda}\right)^{p-1}\Vert b\Vert^p.
\]
For $(x, y) \in \mb((\ov{x}, \ov{x}); r) \cap \Dom{\partial \Phi}$ with $0 < \left\vert \Phi(x, y) - \Phi(\ov{x}, \ov{x}) \right\vert< \eta$, shrink $r$ such that $2r < 1$. Since $\frac{p-1}{\theta} \leq p$ for $\theta \geq \frac{p-1}{p}$, and $\Vert x - y \Vert \leq 2r < 1$, we have
\begin{align*}
\dist^{\frac{1}{\theta}}\left(0, \partial \Phi(x,y)\right)&
\overset{(i)}{\geq} c_1\left(\Vert \nabla^2 f(x)(y-x)\Vert^{\frac{1}{\theta}}+\mathop{\bs{\inf}}\limits_{\zeta\in \partial g(y)}\Vert\nabla f(x)+ \zeta\Vert^{\frac{1}{\theta}}+\Vert x - y\Vert^{p}\right)
\\
&\geq c_1\left(2^{\frac{\theta - 1}{\theta}}\left(\Vert \nabla^2 f(x)(y-x)\Vert+\mathop{\bs{\inf}}\limits_{\zeta\in \partial g(y)}\Vert\nabla f(x)+ \zeta\Vert\right)^{\frac{1}{\theta}}+\Vert x - y\Vert^{p}\right)
\\
&\geq c_2\left(c^{-1}\left(\Vert \nabla^2 f(x)(y-x)\Vert+\mathop{\bs{\inf}}\limits_{\zeta\in \partial g(y)}\Vert\nabla f(x)+ \zeta\Vert\right)^{\frac{1}{\theta}}+\frac{1}{p\gamma}\Vert x - y\Vert^{p}\right),
\end{align*}
where $c_2:=c_1c~\bs\min \{2^{\frac{\theta - 1}{\theta}},c^{-1}p\gamma\}$ and $(i)$ follows from $\Vert x-y\Vert\leq 2r<1$ and $\frac{p-1}{\theta}\leq p$.
Together with \eqref{eq:theorem:KLdefinition}, this yields
\begin{align}\label{eq:theorem:KLdefinitionb}
\dist^{\frac{1}{\theta}}(0, \partial \Phi(x,y))
&\geq c_2(\left\vert\ell(x,y)-\ell(\ov{x}, \ov{x})\right\vert+\frac{1}{p\gamma}\Vert x - y\Vert^{p})
\nonumber\\&\geq c_2(\vert\ell(x,y)-\ell(\ov{x}, \ov{x})+\frac{1}{p\gamma}\Vert x - y\Vert^{p}\vert)
=c_2 \vert\Phi(x,y)-\Phi(\ov{x}, \ov{x})\vert,
\end{align}
i.e., $\Phi$ satisfies the KL property of the exponent $\theta$ at $(\ov{x}, \ov{x})$. 
 Given that $\ov{x}\in \bs{\rm Fix}(\Tprox{\gf}{\gamma}{p})$, for any $y\in\R^n$ and for any $\gamma\in (0, \widehat{\gamma})$, 
  \begin{align*}
\gf(\ov{x})=\fgam{\gf}{p}{\widehat{\gamma}}(\ov{x})\leq  f(\ov{x})+\langle \nabla f(\ov{x}), y- \ov{x}\rangle+g(y)+\frac{1}{p\widehat{\gamma}}\Vert y - \ov{x}\Vert^p
<f(\ov{x})+\langle \nabla f(\ov{x}), y- \ov{x}\rangle+g(y)+\frac{1}{p\gamma}\Vert y - \ov{x}\Vert^p.
  \end{align*}
Thus,  $\Tprox{\gf}{\gamma}{p}(\ov{x})=\{\ov{x}\}$.
From the KL property of $\Phi$  of the exponent $\theta$ at $(\ov{x}, \ov{x})$, 
\cite[Theorem~3.1]{Yu2022} implies that $\fgam{\gf}{p}{\gamma}$  satisfies the KL property with the exponent $\theta$ at $\ov{x}$.
\end{proof}

\bibliographystyle{spbasic}
\bibliography{references}

\end{document}